\documentclass[11pt]{article}

\usepackage{amssymb,amsmath,bm}
\usepackage{amsthm}

\usepackage{textcomp}
\usepackage{enumerate}      
\usepackage{graphicx}        
\usepackage{caption}

\usepackage{tabu}

\usepackage[export]{adjustbox}

\usepackage{mathrsfs}

\usepackage{url}

\usepackage{array}
\newcommand{\PreserveBackslash}[1]{\let\temp=\\#1\let\\=\temp}
\newcolumntype{C}[1]{>{\PreserveBackslash\centering}p{#1}}
\newcolumntype{R}[1]{>{\PreserveBackslash\raggedleft}p{#1}}
\newcolumntype{L}[1]{>{\PreserveBackslash\raggedright}p{#1}}

\renewcommand{\setminus}{{\smallsetminus}}

\newcommand{\cp}[1]{\vcenter{\hbox{#1}}}

\usepackage{amssymb,amsmath,bm}
\usepackage{amsthm}
\usepackage{hyperref}
\usepackage{mathrsfs}
\usepackage{textcomp}
\usepackage{enumerate}
\usepackage{graphicx}
\usepackage{tikz}
\usepackage{verbatim}

\newtheorem{theorem}{Theorem}[section]
\newtheorem{lemma}[theorem]{Lemma}
\newtheorem{proposition}[theorem]{Proposition}
\newtheorem{definition}[theorem]{Definition}
\newtheorem{corollary}[theorem]{Corollary}
\newtheorem{conjecture}[theorem]{Conjecture}
\newtheorem{question}[theorem]{Question}
\newtheorem{problem}[theorem]{Problem}

\theoremstyle{remark}
\newtheorem{remark}[theorem]{Remark}

\theoremstyle{remark}
\newtheorem{example}[theorem]{Example}

\numberwithin{equation}{section}

\begin{document}
\title{\bf Asymptotics of quantum $6j$-symbols and generalized hyperbolic tetrahedra}

\author{Giulio Belletti and Tian Yang}

\date{}

\maketitle

\begin{abstract}We establish the geometry behind the quantum $6j$-symbols under only the admissibility conditions as  in the definition of the Turaev-Viro invariants of $3$-manifolds. As a classification, we show that  the $6$-tuples in the quantum $6j$-symbols give in a precise way to the dihedral angles of (1) a spherical tetrahedron, (2) a generalized Euclidean tetrahedron, (3) a generalized hyperbolic tetrahedron or (4) in the degenerate case the angles between four oriented straight lines in the Euclidean plane.  We also show that for a large proportion of the cases, the $6$-tuples always give the dihedral angles of a generalized hyperbolic tetrahedron and the exponential growth rate of the corresponding quantum $6j$-symbols equals the suitably defined volume of this generalized hyperbolic tetrahedron.  It is worth mentioning that the volume of a generalized hyperbolic tetrahedron can be negative, hence the corresponding sequence of the quantum $6j$-symbols could decay exponentially. This is a phenomenon that has never been aware of before.
\end{abstract}

\section{Introduction}

Quantum $6j$-symbols  are the main building blocks of the Turaev-Viro invariants of $3$-manifolds\,\cite{TV}; and the asymptotic behavior of the former plays a central role in understanding that of the latter\,\cite{CY, BEL, BDKY}. It is proved in \cite{C} (see also \cite{CM, BDKY}) that if the $6$-tuples of a sequence of quantum $6j$-symbols correspond in a precise way to the dihedral angles of a hyperbolic hyperideal tetrahedron, then the quantum $6j$-symbols grow exponentially, and the growth rate is given by the volume of this hyperbolic hyperideal tetrahedron. It is believed that this should hold without the condition on the type of the vertices of the tetrahedron. On the other hand, the case where the sequence of the $6$-tuples does not correspond to a hyperbolic tetrahedron has not been studied; although it is equally crucial to the asymptotics of the Turaev-Viro invariants and seems to hold abundant geometric significance, as numerical computation shows that the sequence of quantum $6j$-symbols could either grow or decay exponentially in this case! The goal of this paper is to answer the following fundamental question.

\begin{question}\label{question} What is the asymptotic behavior of a sequence of quantum $6j$-symbols $\Bigg|\begin{matrix}
a_1^{(r)}  & a_2^{(r)}  & a_3^{(r)} \\
   a_4^{(r)}  & a_5 ^{(r)} & a_6^{(r)} 
  \end{matrix}\Bigg|$ evaluated at $q=e^{\frac{2\pi \sqrt{-1}}{r}}$ under only the  $r$-admissibility conditions of the $6$-tuples $\{(a_1^{(r)}, \dots, a_6^{(r)})\}?$ 
  \end{question}

Under the condition of Question \ref{question}, for each $k\in\{1,\dots,6\},$ let 
$$\alpha_k=\lim_{r\to\infty} \frac{2\pi a_k^{(r)}}{r};$$
and let 
$$\theta_k=|\pi -\alpha_k|,$$
or equivalently,
$$\alpha_k=\pi\pm\theta_k.$$
Then $(\alpha_1,\dots,\alpha_6)$ satisfies the \emph{admissibility conditions}. (See Definition \ref{adm}.) To answer Question \ref{question}, we as one of the main results of this paper in Theorem \ref{main1} classify all the admissible $6$-tuples $(\alpha_1,\dots,\alpha_6)$ into the dihedral angles of a spherical, a generalized Euclidean or a generalized hyperbolic tetrahedron, together with one degenerate extra case. We defined these terms respectively in Sections \ref{ght} and \ref{get}, and give criteria of them  in Theorem \ref{characterization} and Theorem \ref{ge}  in terms of the Gram matrix of the $6$-tuples. For the generalized hyperbolic tetrahedra, there is a natural way to define the \emph{edge lengths} (see Definition \ref{el})  and the \emph{volume} (see Definition \ref{v}), which extend the edge lengths and volume of  a hyperbolic tetrahedron, and still satisfy the Schl\"afli formula (see Proposition \ref{Schlafli}) and additivity (see (\ref{additivity})). We then obtain an explicit formula for the volume of a generalized hyperbolic tetrahedron in terms of the admissible $6$-tuple $(\alpha_1,\dots,\alpha_6).$ (See Theorem \ref{volume} and Theorem \ref{volume2}.)  Another main result of this paper is Theorem \ref{asymp1} that under one extra condition of the $6$-tuple $(\alpha_1,\dots,\alpha_6)$ other than the necessary admissibility conditions,  $(\theta_1,\dots,\theta_6)$ is the set of dihedral angles of a generalized hyperbolic tetrahedron, and the exponential growth rate of the corresponding sequence of quantum $6j$-symbols equals the volume of this generalized hyperbolic tetrahedron.  It worths mentioning that the volume of a generalized hyperbolic tetrahedron can be negative, which implies that the  sequence of quantum $6j$-symbols exponentially decays. This phenomenon has never been aware of before. We also would like to mention that in the only exceptional case that Theorem \ref{asymp1} does not cover (Case (1) of Proposition \ref{class}), there is something more interesting happening, namely, we found $6$-tuples giving the dihedral angles of a generalized hyperbolic tetrahedron with positive volume, yet the exponential growth rate of the corresponding sequence of quantum $6j$-symbols is negative. (See Example \ref{example}). This deserves a further study.

\begin{theorem} \label{main1} Let  $(\alpha_{1},\dots,\alpha_{6})$ be an admissible $6$-tuple real numbers, and for $k\in\{1,\dots,6\}$ let 
$$\theta_k=|\pi-\alpha_k|.$$  Then $(\theta_{1},\dots,\theta_{6})$ is of one of the following four possibilities:
\begin{enumerate}[(1)]
\item   the set of dihedral angles of a spherical tetrahedron,
\item   the set of  dihedral angles of a generalized Euclidean tetrahedron,
\item    the set of  dihedral angles of a generalized hyperbolic tetrahedron, and
\item  the set of  angles between four oriented straight lines in the Euclidean plane.
\end{enumerate}
\end{theorem}

\begin{remark} A hyperbolic tetrahedron was called a generalized hyperbolic tetrahedron in Ushijima\,\cite{U} to indicate that the vertices can be not only regular,  but also ideal or hyperideal. In this article, we will reserve the terminology \emph{generalized hyperbolic tetrahedra} for a more general set of geometric objects. See Definition \ref{ght}. The definition for generalized Euclidean tetrahedron will be given in Section \ref{get}.
\end{remark}

\begin{theorem} \label{asymp1} Let  $\{(a_1^{(r)}, \dots, a_6^{(r)})\}$ be a sequence of $r$-admissible $6$-tuples.
For each $k\in\{1,\dots,6\},$ let 
$$\alpha_k=\lim_{r\to\infty} \frac{2\pi a_k^{(r)}}{r},$$
and let $$\theta_k=|\pi -\alpha_k|.$$ Let $G$ be the Gram matrix of $(\theta_1,\dots,\theta_6),$ and for $i\in\{1,2,3,4\}$ let $G_{ii}$ be the $ii$-th cofactor of $G.$ 
If 
$$G_{ii}< 0$$ 
for at least one $i\in\{1,2,3,4\},$
then 
\begin{enumerate}[(1)]
\item $(\theta_1,\dots,\theta_6)$ is the set of dihedral angles of a generalized hyperbolic tetrahedra $\Delta,$ and
\item as $r$ varies over all positive odd integers,
$$\lim_{r\to\infty}\frac{2\pi}{r}\ln \Bigg|\begin{matrix}
a_1^{(r)}  & a_2^{(r)}  & a_3^{(r)} \\
   a_4^{(r)}  & a_5 ^{(r)} & a_6^{(r)} 
  \end{matrix}\Bigg|_{q=e^{\frac{2\pi \sqrt{-1}}{r}}}=\mathrm{{Vol}}(\Delta).$$
  \end{enumerate}
\end{theorem}

In a sequel paper, we will address the asymptotic expansion of sequences of quantum $6j$-symbols, which is a refinement of Theorem \ref{asymp1}. See Theorem \ref{asymp2} below. To be precise, let  $\{(a_1^{(r)}, \dots, a_6^{(r)})\}$ be a sequence of $r$-admissible $6$-tuples such that for any $k\in \{1,\dots,6\},$ either $a_k^{(r)}>\frac{r}{2}$ for all $r$ or $a_k^{(r)}<\frac{r}{2}$ for all $r.$  In the former case we let $\mu_k=1$ and in the latter case we let $\mu_k=-1,$ and we let
$$\theta^{(r)}_k=\mu_k\bigg(\frac{2\pi a_k^{(r)}}{r}-\pi\bigg)$$
for each $k\in \{1,\dots,6\}.$ 
For each $r,$ let $G^{(r)}$ be the Gram matrix of the $6$-tuple  $(\theta^{(r)}_1,\dots,\theta^{(r)}_6),$  and for $i\in\{1,2,3,4\},$ let $G_{ii}^{(r)}$ be the $ii$-th cofactor of  $G^{(r)}.$ If $G_{ii}^{(r)}<0$ for for at least one $i\in\{1,2,3,4\},$ then by Theorem \ref{asymp1} (1), $(\theta^{(r)}_1,\dots,\theta^{(r)}_6)$ is the set of dihedral angles of a generalized hyperbolic tetrahedron $\Delta^{(r)}.$  We let $\mathrm{Vol}(\Delta^{(r)})$ and  $(l^{(r)}_1,\dots,l^{(r)}_6)$ respectively be the volume and the set of edge lengths of $\Delta^{(r)},$ and also denote by $G(\Delta^{(r)})$ the Gram matrix of $\Delta^{(r)}.$  Finally we assume that  $\{(\theta^{(r)}_1,\dots,\theta^{(r)}_6)\}$  converges  as $r$ tends to infinity, to the set of  dihedral angles of a generalized hyperbolic tetrahedron without ideal vertices.

\begin{theorem}\label{asymp2}  Under the above assumptions, as $r$ varies over all positive odd integers,
\begin{equation*}
\Bigg|\begin{matrix}
a_1^{(r)}  & a_2^{(r)}  & a_3^{(r)} \\
   a_4^{(r)}  & a_5 ^{(r)} & a_6^{(r)} 
  \end{matrix}\Bigg|_{q=e^{\frac{2\pi \sqrt{-1}}{r}}}
=\frac{\sqrt{2}\pi}{r^{\frac{3}{2}}}\frac{ e^{-\frac{1}{2}\sum_{k=1}^{6}\mu_kl_k^{(r)}}}{\sqrt[4]{-\det G(\Delta^{(r)})}}
e^{\frac{r}{2\pi}\mathrm{Vol}(\Delta^{(r)})} \bigg( 1 + O\Big(\frac{1}{r}\Big)\bigg).
\end{equation*}
\end{theorem}
\bigskip

\noindent\textbf{Acknowledgments.}  The authors would like to thank Francis Bonahon, Jihoon Sohn and Ka Ho Wong for helpful discussions. The second author is supported by NSF Grants DMS-1812008 and DMS-2203334.

\section{Preliminaries}

\subsection{Hyperbolic geometry in dimension  $3$}

{\bf Hyperboloid model.} The Lorentzian space $\mathbb E^{3,1}$ is the vector space $\mathbb R^4$ with the inner product $\langle,\rangle$ defined for  $\mathbf x = (x_1,x_2,x_3,x_4)$ and $\mathbf y= (y_1,y_2,y_3,y_4)$ by
$$\langle \mathbf x, \mathbf y\rangle=x_1y_1+x_2y_2+x_3y_3-x_4y_4.$$
The de Sitter space is
 $$\mathbb S(1) = \{\mathbf v\in\mathbb E^{3,1}\ |\ \langle \mathbf v,\mathbf v \rangle =1 \}.$$ 
 The light cone is
  $$\mathbb S(0)=\{\mathbf v\in\mathbb E^{3,1}\ |\ \langle \mathbf v,\mathbf v \rangle = 0\},$$
  and the upper- and lower-light cones are respectively
 $$\mathbb L^3_+ = \{\mathbf v\in\mathbb S(0)\ |\ v_4>0 \}\quad\text{and}\quad\mathbb L^3_-= \{\mathbf v\in\mathbb S(0)\ |\ v_4<0 \}$$
 so that $\mathbb S(0)=\mathbb L^3_+ \cup\mathbb L^3_-\cup\{\mathbf 0\}.$
  The hyperboloid is
  $$\mathbb S(-1)=\{\mathbf v\in\mathbb E^{3,1}\ |\ \langle \mathbf v,\mathbf v \rangle = -1\},$$
  and the upper- and lower-hyperboloids are respectively
 $$\mathbb H^3_+ = \{\mathbf v\in\mathbb S(-1)\ |\ v_4>0 \}\quad\text{and}\quad\mathbb H^3_-= \{\mathbf v\in\mathbb S(-1)\ |\ v_4<0 \}$$
 so that $\mathbb S(-1)=\mathbb H^3_+ \cup\mathbb H^3_-.$ The inner product $\langle,\rangle$ on $\mathbb E^{3,1}$ restricts to a Riemannian metric on $\mathbb S(-1)$ with constant sectional curvature $-1,$  call the \emph{hyperbolic metric}. We also consider the spaces 
$$\mathbb B^{3,1}=\{\mathbf v\in\mathbb E^{3,1}\ |\ \langle \mathbf v,\mathbf v \rangle <0 \},$$
and
$$\overline{\mathbb B^{3,1}}=\mathbb B^{3,1}\cup\mathbb S(0)=\{\mathbf v\in\mathbb E^{3,1}\ |\ \langle \mathbf v,\mathbf v \rangle \leqslant 0 \},$$

For each $\mathbf v\in\mathbb S(1),$ let 
$$\mathbf \Pi_{\mathbf v}=\{ \mathbf w \in \mathbb E^{3,1}\ |\ \langle \mathbf v, \mathbf w\rangle =0\}$$ 
be the hyperplane planing containing all the vectors perpendicular to $\mathbf v.$ Then $\mathbf \Pi_{\mathbf v}$ intersects $\mathbb B^{3,1}.$ This is gives a one-to-one correspondence between hyperplanes $\mathbf \Pi$  intersecting $\mathbb B^{3,1}$ and pairs of opposite vectors $\{\pm \mathbf v\}$ in $\mathbb S(1)$ perpendicular to $\Pi$.  We defined the \emph{orientation} of a hyperplane  intersecting $\mathbb B^{3,1}$ to be a specification of one of the two vectors $\{\pm\mathbf v\}.$ Suppose $\mathbf v$ defines the orientation of $\mathbf \Pi_{\mathbf v},$ then we call $\mathbf v$  the \emph{outward normal vector} of $\mathbf \Pi_{\mathbf v},$ and call the intersection 
$$\Pi_{\mathbf v}=\mathbf \Pi_{\mathbf v}\cap \mathbb S(-1)$$ the \emph{plane of truncation} at $\mathbf v.$ 
\\


\noindent{\bf Projective model.} Let the affine hyperplane $\mathbb P^3_1=\{\mathbf x\in\mathbb E^{3,1}\ | x_4=1\},$ and let 
$\mathrm{p}: \mathbb E^{3,1} \setminus \{\mathbf x\in\mathbb E^{3,1}\ | x_4=0\}\to \mathbb P^3_1$
be the radial projection along the ray from the origin $\mathbf 0.$ Then $\mathrm{p}$ continuously extends to
$$\mathrm{p}:\mathbb E^{3,1}\setminus\{\mathbf 0\}\to\mathbb P^3_1\cup\mathbb P^3_{\infty},$$
where  $\mathbb P^3_\infty$ is the set of lines in the linear subspace $\{\mathbf x\in\mathbb E^{3,1}\ | x_4=0\}$ passing through the origin $\mathbf 0.$ The projective model is then
$$\mathbb P^3=\mathbb P^3_1\cup \mathbb P^3_{\infty}.$$

The radial projection $\mathrm{p}$ restricts to a two-to-one map from $\mathbb S(-1)$ to the open unit ball   in $\mathbb P^3_1,$ which is a homeomorphism on each of $\mathbb H^3_+$ and $\mathbb H^3_-,$ hence the unit ball inherits the hyperbolic metric from them via $\mathrm{p}.$ The Klein model of the hyperbolic space, denoted by $\mathbb H^3,$ is the unit ball of $\mathbb P^3_1$ with this induced metric.  The image of the light cone $\mathbb S(0)\setminus \{\mathbf 0\}$ is the unit circle $\mathbb S^2_\infty$ in $\mathbb P^3_1,$ and the closure of the hyperbolic space in the projective model  is 
$$\overline{\mathbb H^3}=\mathbb H^3\cup\mathbb S^2_\infty.$$

We notice that the spaces $\mathbb B^{3,1}$ and $\overline{\mathbb B^{3,1}}$ are respectively the pre-images of $\mathbb H^3$ and $\overline{\mathbb H^3}$ under the radial projection $\mathrm{p}.$ 
\\


\noindent{\bf Hyperbolic tetrahedra.} A \emph{hyperbolic tetrahedron} is the convex hull of the vertices of a quadruple  of linearly independent vectors $\mathbf v_1,$ $\mathbf v_2,$ $\mathbf v_3$ and $\mathbf v_4$ in  $\mathbb H^3_+\cup\mathbb L^3_+\cup \mathbb S(1)$ or in $\mathbb H^3_-\cup\mathbb L^3_-\cup \mathbb S(1)$   such that: (1) for each pair $\{i,j\}\subset\{1,2,3,4\},$ the straight  line  segment $L^+_{ij}$ passing through the vertices of $\mathbf v_i$ and $\mathbf v_j$ intersects $\overline{\mathbb B^{3,1}},$ and (2) for any $\mathbf v_i\in\mathbb S(1),$  all the other vertices $\mathbf v_j,\mathbf v_k,\mathbf v_l$ are on the other side of the hyperplane $\mathbf \Pi_i$ perpendicular to $\mathbf v_i.$  We call the vertex $\mathbf v_i$ of a hyperbolic tetrahedron a \emph{regular vertex} if $\mathbf v_i \in\mathbb S(-1),$ an \emph{ideal vertex} if $\mathbf v_i\in\mathbb S(0)$ and a \emph{hyperideal vertex} if $\mathbf v_i\in\mathbb S(1).$ We call a hyperbolic tetrahedron \emph{regular}, \emph{ideal} or \emph{hyperideal} if all the vertices of it are respectively so.

In the projective model, by considering the radial projection $\mathrm{p},$ a hyperbolic tetrahedron is the convex hull of four points $v_1,v_2,v_3,v_4$ in the general positions  in $\mathbb P^3$  such that, (1) for any $v_i$ and $v_j$ that are not in $\overline{\mathbb H^3},$ the straight line segment (inherited from the affine structure of $\mathbb P^3_1$) connecting them intersects  $\overline{\mathbb H^3},$ and (2) for each $v_i$ that  is not in $\overline{\mathbb H^3},$  all the other vertices $v_j, v_k, v_l$ are on the other side of the radial projection of the plane of truncation $\Pi_i.$ In the projective model, a vertex $v_i$ is \emph{regular} if $v_i\in\mathbb H^3,$ \emph{ideal} if $v_i\in\mathbb S^2_\infty$ and \emph{hyperideal} if $v_i\notin\overline{\mathbb H^3};$ and a hyperbolic tetrahedron is \emph{regular}, \emph{ideal} or \emph{hyperideal} if all the vertices of it are respectively so.

Back to the hyperboloid model, the \emph{face} $F_i$ opposite to the vertex $\mathbf v_i$ is the plane containing the vertices of $\mathbf v_j,\mathbf v_k,\mathbf v_l,$ $\{j,k,l\}=\{1,2,3,4\}\setminus\{i\},$ and the dihedral angle $\theta_{ij}$ is the angle between the faces $F_i$ and $F_j.$  
The  \emph{Gram matrix} of a $6$-tuple $(\theta_{12},\dots,\theta_{34})$ of real numbers is the following $4\times4$ matrix 
$$G=\left[\begin{matrix}
1& -\cos\theta_{12} & -\cos\theta_{13} & -\cos\theta_{14}\\
-\cos\theta_{12} & 1& -\cos\theta_{23} & -\cos\theta_{24} \\
-\cos\theta_{13} & -\cos\theta_{23} & 1&  -\cos\theta_{34} \\
-\cos\theta_{14} & -\cos\theta_{24} & -\cos\theta_{34} &  1\\
 \end{matrix}\right].$$
If $(\theta_{12},\dots,\theta_{34})$ is the set of dihedral angles of a hyperbolic tetrahedron $\Delta,$ then $G$ is called the  \emph{Gram matrix} of $\Delta.$

\begin{theorem}[Luo\,\cite{L}, Ushijima\,\cite{U}]
Suppose $(\theta_{12},\dots,\theta_{34})$ is a $6$-tuple of numbers in $[0,\pi].$ Then the following statements are equivalent.
\begin{enumerate}[(1)]
\item $(\theta_{12},\dots,\theta_{34})$ is the set of dihedral angles of a hyperbolic tetrahedron.
\item  The Gram matrix $G$ of $(\theta_{12},\dots,\theta_{34})$ satisfies the following two conditions:
\begin{enumerate}[(a)]
\item  $G$ has signature $(3,1),$ 
\item  the $ij$-th cofactor $G_{ij}>0$ for any $\{i,j\}\subset \{1,2,3,4\}.$ 
\end{enumerate}
\end{enumerate}
Moreover, when the above conditions are satisfied, the vertex $\mathbf v_i$ is regular (resp. ideal, hyperideal) if $G_{ii}>0$ (resp. $G_{ii}=0,$ $G_{ii}<0$). 
\end{theorem}

\begin{theorem}[Bonahon-Bao\,\cite{BB}]\label{BoB} Suppose $(\theta_{12},\dots,\theta_{34})$ is a $6$-tuple of numbers in $[0,\pi].$ Then the following statements are equivalent.
\begin{enumerate}[(1)]
\item $(\theta_{12},\dots,\theta_{34})$ is the set of dihedral angles of a hyperideal hyperbolic tetrahedron.
\item  For each $i\in\{1,2,3,4\},$ $\theta_{jk}+\theta_{jl}+\theta_{kl}<\pi,$ where $\{j,k,l\}=\{1,2,3,4\}\setminus\{i\}.$
\end{enumerate}
In particular, the spaces of hyperideal hyperbolic tetrahedra parametrized by the dihedral angles is a convex open polytope in $[0,\pi]^6.$
\end{theorem}


\subsection{Quantum $6j$-symbols}\label{6j}

Let $r$ be an odd integer and $q$ be an $r$-th root of unity. For the context of this paper we are only interested in the case $q=e^{\frac{2\pi\sqrt{-1}}{r}},$ but the definitions in this section work with any choice of $q.$

As is customary we define $[n]=\frac{q^n-q^{-n}}{q-q^{-1}},$ and the quantum factorial 
$$[n]!=\prod_{k=1}^n[k].$$

A triple $(a_1,a_2,a_3)$ of integers in $\{0,\dots,r-2\}$ is \emph{$r$-admissible} if 
\begin{enumerate}[(1)]
\item  $a_i+a_j-a_k\geqslant 0$ for $\{i,j,k\}=\{1,2,3\}.$
\item $a_1+a_2+a_3\leqslant 2(r-2),$ 
\item $a_1+a_2+a_3$ is even.
\end{enumerate}

For an $r$-admissible triple $(a_1,a_2,a_3),$ define 
$$\Delta(a_1,a_2,a_3)=\sqrt{\frac{[\frac{a_1+a_2-a_3}{2}]![\frac{a_2+a_3-a_1}{2}]![\frac{a_3+a_1-a_2}{2}]!}{[\frac{a_1+a_2+a_3}{2}+1]!}}$$
with the convention that $\sqrt{x}=\sqrt{-1}\sqrt{|x|}$ when the real number $x$ is negative.

A  6-tuple $(a_1,\dots,a_6)$ is \emph{$r$-admissible} if the triples $(a_1,a_2,a_3),$ $(a_1,a_5,a_6),$ $(a_2,a_4,a_6)$ and $(a_3,a_4,a_5)$ are $r$-admissible.

\begin{definition}
The \emph{quantum $6j$-symbol} of an $r$-admissible 6-tuple $(a_1,\dots,a_6)$ is 
\begin{multline*}
\bigg|\begin{matrix} a_1 & a_2 & a_3 \\ a_4 & a_5 & a_6 \end{matrix} \bigg|
= \sqrt{-1}^{-\sum_{i=1}^6a_i}\Delta(a_1,a_2,a_3)\Delta(a_1,a_5,a_6)\Delta(a_2,a_4,a_6)\Delta(a_3,a_4,a_5)\\
\sum_{z=\max \{T_1, T_2, T_3, T_4\}}^{\min\{ Q_1,Q_2,Q_3\}}\frac{(-1)^z[z+1]!}{[z-T_1]![z-T_2]![z-T_3]![z-T_4]![Q_1-z]![Q_2-z]![Q_3-z]!},
\end{multline*}
where
 $$T_1=\frac{a_1+a_2+a_3}{2},\quad T_2=\frac{a_1+a_5+a_6}{2},\quad  T_3=\frac{a_2+a_4+a_6}{2}, \quad  T_4=\frac{a_3+a_4+a_5}{2},$$ 
 $$Q_1=\frac{a_1+a_2+a_4+a_5}{2}, \quad Q_2=\frac{a_1+a_3+a_4+a_6}{2}, \quad Q_3=\frac{a_2+a_3+a_5+a_6}{2}.$$
\end{definition}

Let $\big\{(a^{(r)}_1,\dots,a^{(r)}_6)\big\}$ be a sequence of quantum $6j$-symbols. For each $i\in\{1,\dots,6\},$ let 
$$\alpha_i=\lim_{r\to\infty} \frac{2\pi a_i^{(r)}}{r};$$
and let 
$$\theta_i=|\pi -\alpha_i|,$$
or equivalently,
$$\alpha_i=\pi\pm\theta_i.$$

\begin{theorem}[Costantino\cite{C}, Chen-Murakami\,\cite{CM}] \label{CM}
If \begin{enumerate}[(1)]
\item $(\theta_1,\dots,\theta_6)$ is the set of dihedral angles of a hyperbolic tetrahedron $\Delta,$ 
\item $\alpha_i=\pi-\theta_i$  for all $i\in\{1,\dots, 6\},$ and
\item at least one vertex of $\Delta$ is ideal or hyperideal, 
\end{enumerate}
then  as $r$ varies over all positive odd integers,
$$\lim_{r\to\infty}\frac{2\pi}{r}\ln \Bigg|\begin{matrix}
a_1^{(r)}  & a_2^{(r)}  & a_3^{(r)} \\
   a_4^{(r)}  & a_5 ^{(r)} & a_6^{(r)} 
  \end{matrix}\Bigg|_{q=e^{\frac{2\pi \sqrt{-1}}{r}}}=\mathrm{Vol}(\Delta).$$
\end{theorem}

\begin{theorem}[Costantino\,\cite{C}, Belletti-Detcherry-Kalfagianni-Yang\,\cite{BDKY}]\label{BDKY} If  for each triple $(i,j,k)$ around a vertex:
\begin{enumerate}[(1)]
  \item $0\leqslant \alpha_i+\alpha_j-\alpha_k\leqslant 2\pi,$ and
  \item $2\pi \leqslant  \alpha_i+\alpha_j+\alpha_k\leqslant 4\pi,$ 
\end{enumerate}
then 
\begin{enumerate}[(1)]
\item $(\theta_1,\dots,\theta_6)$ is the set of dihedral angles of a hyperbolic tetrahedron $\Delta$ with all the vertices ideal or hyperideal, and  
\item  as $r$ varies over all positive odd integers,
 $$\lim_{r\to\infty}\frac{2\pi}{r}\ln \Bigg|\begin{matrix}
a_1^{(r)}  & a_2^{(r)}  & a_3^{(r)} \\
   a_4^{(r)}  & a_5 ^{(r)} & a_6^{(r)} 
  \end{matrix}\Bigg|_{q=e^{\frac{2\pi \sqrt{-1}}{r}}}=\mathrm{Vol}(\Delta).$$
\end{enumerate}
\end{theorem}

\begin{remark} We notice that both Theorem \ref{CM} and Theorem \ref{BDKY} are special cases of Theorem \ref{asymp1}.
\end{remark}


\subsection{A remark on the notations} Keep in mind that as ordered $6$-tuples, $(\theta_{12},\dots,\theta_{34})$ and $(\theta_1,\dots,\theta_6)$ are different, and are related by their Gram matrices
\begin{equation*}
\begin{split}
\left[\begin{matrix}
1& -\cos\theta_{12} & -\cos\theta_{13} & -\cos\theta_{14}\\
-\cos\theta_{12} & 1& -\cos\theta_{23} & -\cos\theta_{24} \\
-\cos\theta_{13} & -\cos\theta_{23} & 1&  -\cos\theta_{34} \\
-\cos\theta_{14} & -\cos\theta_{24} & -\cos\theta_{34} &  1\\
 \end{matrix}\right]
 =\left[\begin{matrix}
1& -\cos\theta_1 & -\cos\theta_2 & -\cos\theta_6\\
-\cos\theta_1 & 1& -\cos\theta_3 & -\cos\theta_5 \\
-\cos\theta_2 & -\cos\theta_3 & 1&  -\cos\theta_4 \\
-\cos\theta_6 & -\cos\theta_5 & -\cos\theta_4 &  1\\
 \end{matrix}\right],
 \end{split}
 \end{equation*}
 i.e., $\theta_1=\theta_{12},$ $\theta_2=\theta_{13},$ $\theta_3=\theta_{23},$ $\theta_4=\theta_{34},$ $\theta_5=\theta_{24}$ and $\theta_6=\theta_{14}.$ See Figure \ref{convention}.
 For example, in the former, the first three angles $(\theta_{12},\theta_{13},\theta_{14})$ are at the edges around the face $F_1$ opposite to the vertex $\mathbf v_1,$ whereas in the latter, the first three angles $(\theta_1,\theta_2,\theta_3)$ are at the edges around the vertex $\mathbf v_4.$   In this paper, we will use both notations for difference purpose; and when we switch from one to the other, we change the ordered $6$-tuples. 

\begin{figure}[htbp]
\centering
\includegraphics[scale=0.3]{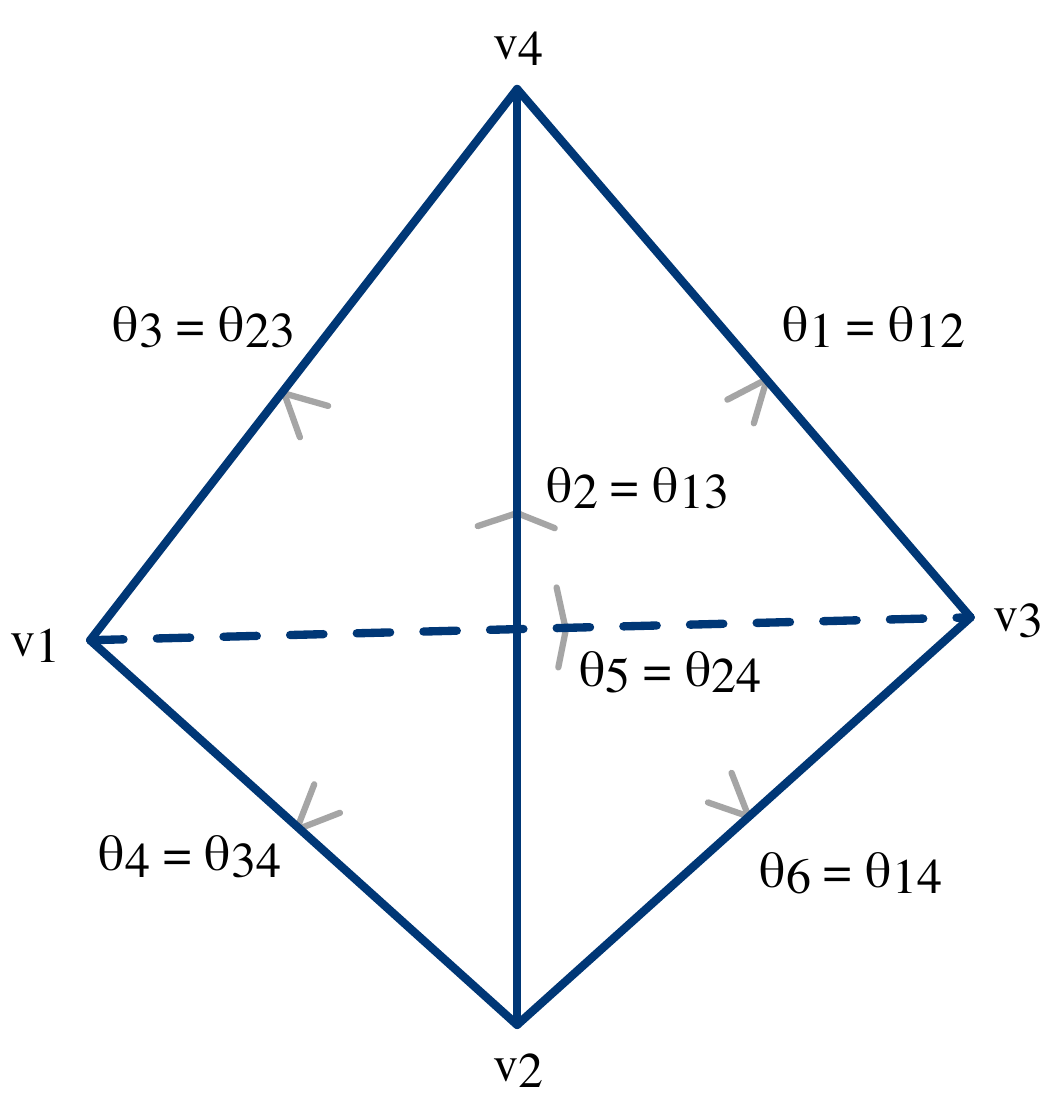}
\caption{}
\label{convention}
\end{figure}


\section{Generalized hyperbolic tetrahedra}\label{ght}

\begin{definition}[Generalized hyperbolic tetrahedron]\label{ght} A \emph{generalized hyperbolic tetrahedron} is  a quadruple  of linearly independent vectors $\mathbf v_1,$ $\mathbf v_2,$ $\mathbf v_3$ and $\mathbf v_4$ in $\mathbb S(-1)\cup\mathbb S(0)\cup \mathbb S(1)\subset\mathbb E^{3,1}$ such that for each pair $\{i,j\}\subset\{1,2,3,4\},$ the straight  line $L_{ij}$ passing through the vertices of $\mathbf v_i$ and $\mathbf v_j$ intersects $\overline{\mathbb B^{3,1}}.$  
\end{definition}

\begin{remark} The last condition is equivalent to that the radial projection of $L_{ij}$ into the projective model $\mathbb P^3$ intersects $\overline{\mathbb H^3}.$ 
\end{remark}

\begin{remark} The main difference between a generalized hyperbolic tetrahedron and a hyperbolic tetrahedron is that: (1) we allow the vertices to be in $\mathbb H^3_+\cup \mathbb L^3_+$ and $\mathbb H^3_-\cup \mathbb L^3_-$ at the same time, not only in one of them, (2) we allow the whole straight line $L_{ij}$ passing through $\mathbf v_i$ and $\mathbf v_j$  to intersect $\overline{\mathbb B^{3,1}},$  not only the line segment $L^+_{ij}$ between them, and (3) we allow $\mathbf v_i$ and other vertices $\mathbf v_j,$ $\mathbf v_k$ or $\mathbf v_l$ to be on the same side of the hyperplane $\mathbf \Pi_i$ perpendicular to $\mathbf v_i.$
\end{remark}

\begin{remark} Such  objects were also studied by Bonahon and Sohn \cite{BS}, where they also considered the case with deep truncations, that is to say, the case where some $L_{ij}$s do not intersect $\mathbb B^{3,1}.$ For the purpose of studying the asymptotics of quantum $6j$-symbols, we will not consider those cases in this article. The quantum content of deeply truncated tetrahedra can be found in Belletti-Yang\,\cite{BY}.
\end{remark}

Two generalized hyperbolic tetrahedra are \emph{isometric} if the two quadruples are related by an element of $O(3,1).$


\subsection{Dihedral angles, Gram matrix  and a criterion}

The vectors $\mathbf v_1,\dots,\mathbf v_4$ are the  \emph{vertices} of the  generalized hyperbolic tetrahedron. We call a vertex \emph{regular} if it is in $\mathbb S(-1),$ \emph{ideal} if it is in $\mathbb S(0),$ and \emph{hyperideal} if is in $\mathbb S(1).$ A regular or ideal vertex is \emph{positive} if it is in $\mathbb H^3_+\cup \mathbb L^3_+,$ and is \emph{negative} if it is in $\mathbb H^3_-\cup \mathbb L^3_-.$ 

For each $i\in\{1,\dots,4\},$  the \emph{face} opposite to $\mathbf v_i$ is the plane  $F_i$  containing the vertices of $\mathbf v_j,$ $\mathbf v_k$ and $\mathbf v_l,$ $\{j,k,l\}=\{1,2,3,4\}\setminus\{i\},$ i.e.,
$$F_i=\{c_j\mathbf v_j+c_k\mathbf v_k+c_l\mathbf v_l\ |\ c_j,c_k,c_l\in\mathbb R, c_j+c_k+c_l=1\}.$$

  The \emph{outward normal vector} of $F_i$ is the vector $\mathbf u_i$ such that 
\begin{enumerate}[(1)]
\item 
$\mathbf u_i\in\mathbb S(1),$
\item 
$\langle \mathbf u_i,\mathbf v\rangle =0$ for any vector $\mathbf v\in F_i,$ 
and 
\item
$\langle \mathbf u_i,\mathbf v_i\rangle <0.$
\end{enumerate}
Since $L_{jk},$ $L_{jl}$ and $L_{kl}$ intersect $\mathbb B^{3,1},$ so does the plane $F_i$ containing them. Therefore, the outward normal vector $\mathbf u_i$ of $F_i$ exists uniquely.

\begin{definition}[Dihedral angles] \label{da} For $\{i,j\}\subset\{1,2,3,4\}.$ the \emph{dihedral angle} $\theta_{ij}$  between $F_i$ and $F_j$ is  defined as $\pi$ minus the angle between the outward normal vectors $\mathbf u_i$ and $\mathbf u_j.$ I.e.,
$$\theta_{ij}=\pi-\cos^{-1}\langle \mathbf u_i,\mathbf u_j\rangle,$$
or equivalently,
$$\langle \mathbf u_i,\mathbf u_j\rangle=-\cos\theta_{ij}.$$
\end{definition}

We see from the definition that for $i\in\{1,2,3,4\},$  $\{\theta_{jk},\theta_{jl},\theta_{kl}\}$ is the set of dihedral angles at the edges around the vertex $\mathbf v_i,$ where $\{j,k,l\}=\{1,2,3,4\}\setminus\{i\}.$

\begin{definition}[Gram matrix] The  \emph{Gram matrix} of a $6$-tuple $(\theta_{12},\dots,\theta_{34})$ of real numbers is the following $4\times4$ matrix 
$$G=\left[\begin{matrix}
1& -\cos\theta_{12} & -\cos\theta_{13} & -\cos\theta_{14}\\
-\cos\theta_{12} & 1& -\cos\theta_{23} & -\cos\theta_{24} \\
-\cos\theta_{13} & -\cos\theta_{23} & 1&  -\cos\theta_{34} \\
-\cos\theta_{14} & -\cos\theta_{24} & -\cos\theta_{34} &  1\\
 \end{matrix}\right].$$
If $(\theta_{12},\dots,\theta_{34})$ is the set of dihedral angles of a generalized hyperbolic tetrahedron $\Delta,$ then $G$ is called the  \emph{Gram matrix} of $\Delta.$
\end{definition}

\begin{theorem} \label{characterization} Suppose $(\theta_{12},\dots,\theta_{34})$ is a $6$-tuple of numbers in $[0,\pi].$ Then the following statements are equivalent.
\begin{enumerate}[(1)]
\item $(\theta_{12},\dots,\theta_{34})$ is the set of dihedral angles of a generalized hyperbolic tetrahedron.
\item  The Gram matrix $G$ of $(\theta_{12},\dots,\theta_{34})$ has signature $(3,1).$ 
\end{enumerate}
\end{theorem}

\begin{proof} The proof follows the same idea of that of Luo\,\cite[Theorem]{L} and Ushijima\,\cite[Theorem 3.2]{U}.

Suppose (1) holds. Then $G$ is the Gram matrix of a generalized hyperbolic tetrahedron $\Delta.$ Let $\mathbf u_1,\dots, \mathbf u_4$ in $\mathbb S(1)$ be the outward normal vectors of the faces of $\Delta$ and let 
$$U=[\mathbf u_1,\mathbf u_2,\mathbf u_3,\mathbf u_4]$$
 be the $4\times 4$ matrix with $\mathbf u_i$'s as the columns. Then
$$G=U^T\cdot I_{3,1}\cdot U,$$
where $U^T$ is the transpose of $U,$ and $I_{3,1}$ is the matrix
$$I_{3,1}=\left[\begin{matrix}
1& 0& 0 & 0\\
0 & 1&0 &0\\
0& 0& 1&0 \\
0&0 &0 &  -1\\
 \end{matrix}\right].$$
 Hence $G$ has signature $(3,1),$ and (b) holds.

Suppose (2) holds, so that $G$ has signature $(3,1).$ Then 
$$\det G<0,$$
and by Sylvester's Law of Inertia,
$$G=U^T\cdot I_{3,1}\cdot U$$
for an invertible matrix $U,$ uniquely determined up to an action of  elements of $O(3,1).$  Let $\mathbf u_1,\dots, \mathbf u_4$ be the columns of $U$ considered as vectors in $\mathbb E^{3,1}$ and let
$$\mathbf w_i=\sum_{j=1}^4G_{ij}\mathbf u_j,$$
where $G_{ij}$ is the $ij$-th cofactor of $G.$ 
Then we have
$$\langle \mathbf w_i,\mathbf u_j\rangle =\delta_{ij}\det G,$$
where $\delta_{ij}$ is the Kronecker symbol, and as a consequence, 
$$\langle \mathbf w_i,\mathbf w_j\rangle =G_{ij}\det G.$$
For $i\in \{1,2,3,4\},$ we define the vector $\mathbf v_i$ as follows. If $G_{ii}=0,$ then let
$$\mathbf v_i=\mathbf w_i;$$
and if $G_{ii}\neq 0,$ then let 
$$\mathbf v_i=\frac{\mathbf w_i}{\sqrt{|G_{ii}\det G|}}.$$

Next we show that 
\begin{enumerate}[(a)]
\item $\mathbf v_1,\dots, \mathbf v_4$ define a generalized hyperbolic tetrahedron $\Delta.$
\item $\mathbf u_1,\dots,\mathbf u_4$ are the outward normal vectors of $\Delta$, so that $G$ is the Gram matrix of $\Delta.$ 
\end{enumerate}
From this it follows that $(\theta_{12},\dots,\theta_{34})$ is the set of dihedral angles of $\Delta,$ and (a) holds.

To prove (a), we need following three steps. 
\begin{enumerate}[\text{Step} 1.]
\item We show that $\mathbf v_1,\dots, \mathbf v_4$ are linearly independent. Indeed, let 
$$W=[\mathbf w_1,\mathbf w_2,\mathbf w_3,\mathbf w_4]$$
be the $4\times 4$ matrix containing $\mathbf v_i$'s as the columns, then
$$W=U\cdot \mathrm{Ad}(G),$$
where $\mathrm{Ad}(G)$ is the adjoint matrix of $G.$ By Cramer's rule,
$$\mathrm{Ad}(G)=\det G\cdot G^{-1}.$$
Therefore, $\det \mathrm{Ad}(G)\neq 0,$ and as a consequence, 
$$\det W=\det U\cdot \det \mathrm{Ad}(G)\neq 0, $$
and the columns $\mathbf w_1,\dots, \mathbf w_4$ of $W$ are linearly independent. Since $\mathbf v_1,\dots,\mathbf v_4$ are non-zero scalar multiples of $\mathbf w_1,\dots, \mathbf w_4,$ they are linearly independent.

\item We show that for each $i,$ $\mathbf v_i\in\mathbb S(-1)\cup \mathbb S(0)\cup \mathbb S(1).$ Indeed,
if $G_{ii}>0,$ then
$$\langle \mathbf v_i,\mathbf v_i\rangle =   \frac{\langle \mathbf w_i, \mathbf w_i\rangle}{-G_{ii}\det G}  =-1,$$
and $$\mathbf v_i\in\mathbb S(-1);$$
if $G_{ii}=0,$ then
$$\langle \mathbf v_i,\mathbf v_i\rangle = \langle \mathbf w_i, \mathbf w_i \rangle = G_{ii}\det G=0,$$
and $$\mathbf v_i\in\mathbb S(0);$$
and if $G_{ii}<0,$ then
$$\langle \mathbf v_i,\mathbf v_i\rangle =    \frac{\langle \mathbf w_i, \mathbf w_i\rangle}{G_{ii}\det G} = 1,$$
and $$\mathbf v_i\in\mathbb S(1).$$

\item We show that for any  $\mathbf v_i$ and $\mathbf v_j$, $L_{ij}$ intersects $\overline{\mathbb B^{3,1}}.$ If one of $\mathbf v_i$ or $\mathbf v_j$ is already in $\mathbb S(-1)\cup\mathbb S(0),$ then $L_{ij}$ intersects $\overline{\mathbb B^{3,1}}.$ Hence we only need to consider the case that $\mathbf v_i,\mathbf v_j\in\mathbb S(1).$ By Step (2) above, this is equivalent to that $G_{ii}<0$ and $G_{jj}<0.$ Now by Jacobi's Theorem (see \cite[2.5.1. Theorem]{P}) and that $\det G<0,$
we have
$$G_{ij}^2-G_{ii}G_{jj}=(\cos^2\theta_{ij}-1)\det G\geqslant 0.$$
Then either $G_{ij}\geqslant \sqrt{G_{ii}G_{jj}}$ or  $G_{ij}\leqslant -\sqrt{G_{ii}G_{jj}}.$ In either of the cases, by  Lemma \ref{intersect} below,  $L_{ij}$ intersects $\overline{\mathbb B^{3,1}}.$
\end{enumerate}
This completes the proof of (a).

For (b), we verify the conditions of an outward normal vector in the following steps.
\begin{enumerate}[\text{Step} 1.]
\item Since $\langle \mathbf u_i,\mathbf u_i\rangle$ equals the $i$-th diagonal entry of $G$ which equals $1,$ $$\mathbf u_i\in\mathbb S(1).$$
\item For $j\neq i,$ we have
$$\langle \mathbf u_i, \mathbf v_j \rangle = c_j\cdot \langle \mathbf u_i, \mathbf w_j \rangle = \delta_{ij}\det G=0,$$
where $c_j=1$ if $G_{jj}=0,$ and $c_j=\frac{1}{\sqrt{|G_{jj}\det G|}}$ if $G_{jj}\neq 0.$  Therefore, for any vector $\mathbf v$ in the plane $F_i$ containing  $\mathbf v_j,$ $\mathbf v_k$ and $\mathbf v_l,$ $\{j,k,l\}=\{1,2,3,4\}\setminus \{i\},$ 
$$\langle \mathbf u_i, \mathbf v \rangle =0.$$
\item For each $i\in\{1,2,3,4\},$ since $c_i>0$ and $\det G<0,$ we have
$$\langle \mathbf u_i, \mathbf v_i \rangle = c_i\cdot \langle \mathbf u_i, \mathbf w_i \rangle= c_i\det G <0.$$
\end{enumerate}
This completes the proof of (b).
 \end{proof}

\begin{lemma}\label{intersect} Suppose $G_{ii}<0$ and $G_{jj}<0.$ Let $L_{ij}^+$ be the line segment connecting $\mathbf v_i$ and $\mathbf v_j,$ and let $L_{ij}^-=L_{ij}\setminus L_{ij}^+.$ Then

\begin{enumerate}[(1)]
\item  $L_{ij}^+$ intersects $\overline{\mathbb B^{3,1}}$ if and only if 
\begin{equation}\label{>}
G_{ij}\geqslant\sqrt{G_{ii}G_{jj}};
\end{equation} and

\item $L_{ij}^-$ intersects $\overline{\mathbb B^{3,1}}$ if and only if 
\begin{equation}\label{<}
G_{ij}\leqslant -\sqrt{G_{ii}G_{jj}}.
\end{equation}
\end{enumerate}

 \end{lemma}
  
\begin{proof} First, notice that a point on $L^+_{ij}$ has the form $ t\mathbf v_i+(1-t)\mathbf v_j$ for $t$ in the interval $(0,1),$ and a point on $L^-_{ij}$ has the same form for $t$ in $(-\infty,0)\cup(1,\infty).$

By the computation $$\langle \mathbf v_i,\mathbf v_j\rangle =-\frac{G_{ij}}{\sqrt{G_{ii}G_{jj}}},$$ we have that (\ref{>}) is equivalent to $\langle \mathbf v_i,\mathbf v_j\rangle\leqslant -1$ and  that (\ref{<}) is equivalent to $\langle \mathbf v_i,\mathbf v_j\rangle\geqslant 1.$  Since $\langle \mathbf v_i, \mathbf v_i\rangle = \langle \mathbf v_j,\mathbf v_j \rangle =1,$ we have that 
$$\langle t\mathbf v_i+(1-t)\mathbf v_j, t\mathbf v_i+(1-t)\mathbf v_j\rangle =  2(1-\langle \mathbf v_i, \mathbf v_j\rangle)t^2+2(\langle \mathbf v_i,\mathbf v_j\rangle-1)t+1.$$
Then (1) follows from the fact that the quadratic inequality $2(1-\langle \mathbf v_i, \mathbf v_j\rangle)t^2+2(\langle \mathbf v_i,\mathbf v_j\rangle-1)t+1\leqslant0$  has a solution in $(0,1)$ if and only if $\langle \mathbf v_i,\mathbf v_j\rangle\leqslant -1;$ and (2) follows from the fact that the same quadratic inequality  has a solution in $[-\infty,0)\cup(1,\infty]$ if and only if $\langle \mathbf v_i,\mathbf v_j\rangle\geqslant1.$
  \end{proof}

Similar to Theorem \ref{characterization}, we have the following criterion for Gram matrices with signature $(2,1),$ which will be need in the classification of admissible angles in Section \ref{caa}.
    
\begin{theorem}\label{hyperdegenerate} Let $(\theta_{12},\dots,\theta_{34})$ be a $6$-tuple of numbers in $[0,\pi],$ and let $G$ be its Gram matrix. 
  If the signature of $G$ is $(2,1),$ then $(\theta_{12},\dots,\theta_{34})$  is the set of  angles between four oriented geodesics in the hyperbolic plane $\mathbb H^2$ that mutually intersect in $\overline{\mathbb H^2}.$ 
\end{theorem}

\begin{remark}  Here the orientation of a geodesic is defined to the a specification of its normal vector,  and the angle between two geodesics is $\pi$ minus the angle between the two normal vectors that define the orientation of the geodesics. See Figure \ref{4h}.
\end{remark}

\begin{figure}[htbp]
\centering
\includegraphics[scale=0.2]{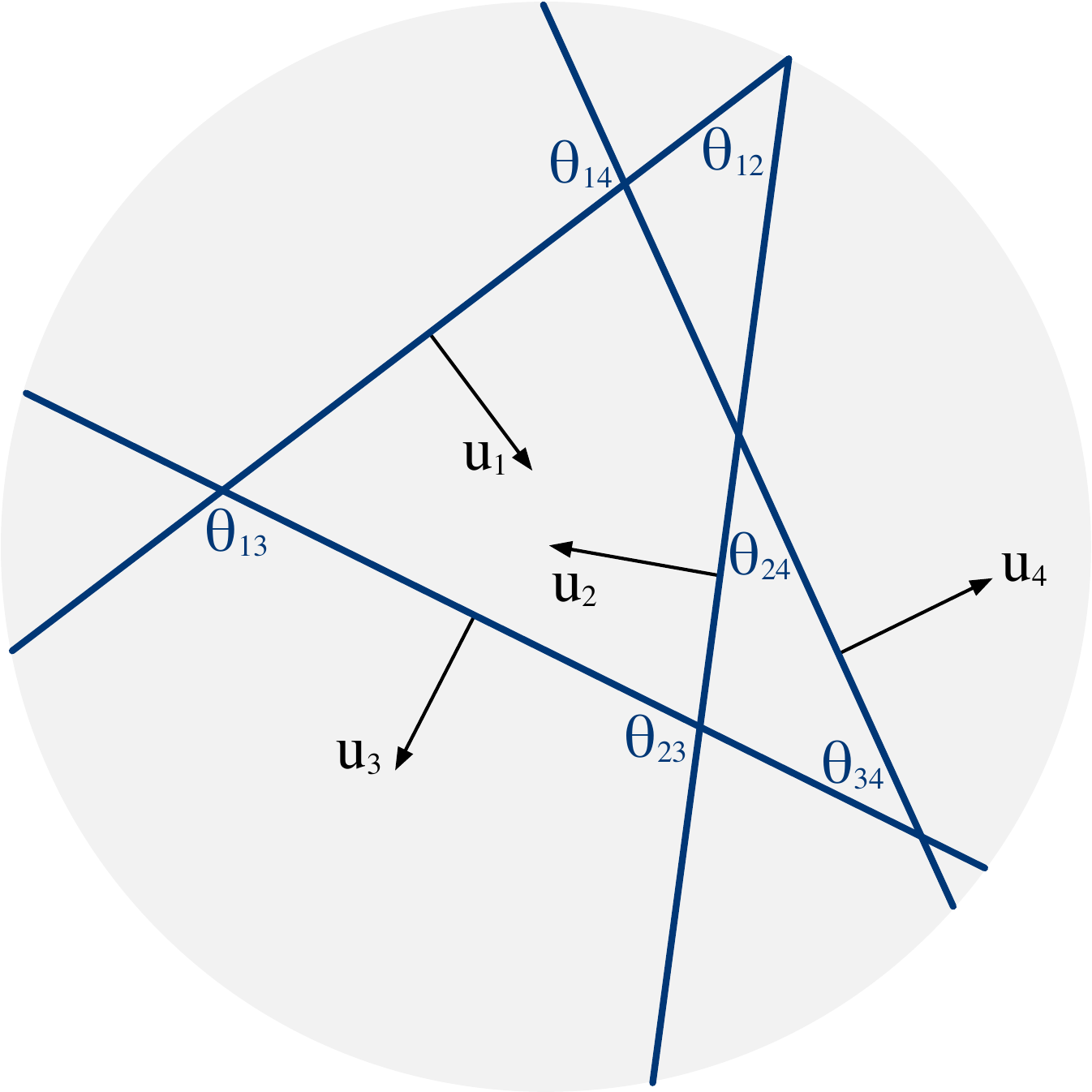}
\caption{In the figure, the grey disc is the closure $\overline{\mathbb H^2}$ of the hyperbolic plane in the projective model.}
\label{4h}
\end{figure}

\begin{proof} By Sylvester's Law of Inertia, 

$$G=W^T\cdot I_{2,1}\cdot W$$
for a $4\times 4$ matrix $W,$ where $I_{2,1}$ is the matrix
$$I_{2,1}=\left[\begin{matrix}
1& 0& 0 & 0\\
0 & 1&0 &0\\
0& 0& -1&0 \\
0&0 &0 &  0\\
 \end{matrix}\right].$$ 
Let $\mathbf w_1,\dots,\mathbf w_4$ be the columns of $W,$ and for each $i\in\{1,\dots,4\},$ let $\mathbf u_i $ be the vector in $\mathbb E^{2,1}$ obtained from $\mathbf w_i$ by erasing  the last component. If $\langle, \rangle$ denotes the inner product of signature $(2,1)$ on $\mathbb E^{2,1},$ then we have
$$\langle \mathbf u_i,\mathbf u_i \rangle=\mathbf w_i^T\cdot I_{2,1}\cdot  \mathbf w_i=1$$
for $i\in\{1,2,3,4\},$ and
$$\langle \mathbf u_i,\mathbf u_j \rangle=  \mathbf w_i^T\cdot I_{2,1}\cdot  \mathbf w_j=-\cos\theta_{ij}$$
for $\{i,j\}\subset\{1,2,3,4\}.$ 
In particular, for each $i\in\{1,2,3,4\},$
$$\mathbf u_i\in\mathbb S(1),$$
the de Sitter space in $\mathbb E^{2,1}$ consisting of vectors of norm $1.$ For each $i\in\{1,2,3,4\},$ let $\Pi_i$ be the plane in $\mathbb E^{2,1}$ perpendicular to $\mathbf u_i$  and let $L_i$ be the intersection of $\Pi_i$ with $\mathbb H^2_+,$ the upper sheet of the hyperboloid consisting of vectors of norm $-1$ and the last component positive. We orient $L_i$ as the direction of $\mathbf u_i.$  Then the angle between $L_i$ and $L_j$ is $\pi-\cos^{-1}\langle \mathbf u_i,\mathbf u_j\rangle=\theta_{ij},$ and $L_1,\dots,L_4$ with this orientation are the desired oriented geodesics.
\end{proof}

To see properties of the dihedral angles and to define the edge lengths and the volume of generalized hyperbolic tetrahedra, we need to use the projective model $\mathbb P^3.$

First, we look at the dihedral angles of the generalized hyperbolic tetrahedra respectively with vertices $\{-\mathbf v_i,\mathbf v_j, \mathbf v_k, \mathbf v_l\}$ and $\{-\mathbf v_i, -\mathbf v_j, \mathbf v_k, \mathbf v_l\},$ $\{i,j,k,l\}=\{1,2,3,4\}.$ We  denote by $\Delta_0$ the radial projection of the convex hull of the vertices of $\{\mathbf v_1,\dots, \mathbf v_4\}.$ For $i\in\{1,\dots, 4\},$ let $\Delta_i$ be the radial projection of the convex hull of the vertices of the quadruple $\{-\mathbf v_i, \mathbf v_j, \mathbf v_k,\mathbf v_l\};$ and for $\{i,j\}\subset\{1,2,3,4\},$ let $\Delta_{ij}$ be the radial projection of the convex hull of the vertices of the quadruple $\{-\mathbf v_i, -\mathbf v_j, \mathbf v_k,\mathbf v_l\}.$ Then $\Delta_0, \Delta_1,\Delta_2,\Delta_3,\Delta_4,\Delta_{12}, \Delta_{13}, \Delta_{23}$ provide a decomposition of $\mathbb P^3$ as  they are the connected components of the complement of the radial projections of the faces $F_1, \dots, F_4$ of $\Delta.$ As an immediate consequence, we have the following

\begin{proposition}\label{dac} Let $\theta_{12},\dots,\theta_{34}$ be the dihedral angles of $\Delta.$ 
\begin{enumerate}[(1)] 
\item Suppose $\theta^*_{12},\dots, \theta^*_{34}$ is the set of dihedral angles of $\Delta_i,$  $i\in\{1,2,3,4\}.$ Then for $j\neq i,$
$$\theta^*_{ij}=\pi-\theta_{ij};$$
 and  for $j,k\neq i,$ 
$$\theta^*_{jk}=\theta_{jk}.$$

\item Suppose $\theta^{**}_{12},\dots, \theta^{**}_{34}$ is the set of dihedral angles of $\Delta_{ij},$  $\{i,j\}\in\{1,2,3,4\}.$ Then for $\{s,t\}=\{i,j\}$ or $\{1,2,3,4\}\setminus\{i,j\},$
$$\theta^{**}_{st}=\theta_{st};$$
 and for $\{s,t\}\neq \{i,j\}$ nor $\{1,2,3,4\}\setminus\{i,j\},$
 $$\theta^{**}_{st}=\pi-\theta_{st}.$$
\end{enumerate}
\end{proposition}

We call the operation sending a $6$-tuple $\{\theta_{st}\}_{\{s,t\}\subset\{1,2,3,4\}}$ to $\{\theta^*_{st}\}_{\{s,t\}\subset\{1,2,3,4\}}$ the \emph{change of angles operation} opposite to the vertex $\mathbf v_i.$ In the case that $\theta_{st}$'s are the dihedral angles of generalized hyperbolic tetrahedron with vertices $\mathbf v_1,\dots,\mathbf v_4,$ this operation corresponds to changing the dihedral angles this generalized hyperbolic tetrahedron  to that of the  generalized hyperbolic tetrahedron with vertices $\{-\mathbf v_i,\mathbf v_j,\mathbf v_k,\mathbf v_l\}.$ We also notice that sending $\{\theta_{st}\}_{\{s,t\}\subset\{1,2,3,4\}}$ to $\{\theta^{**}_{st}\}_{\{s,t\}\subset\{1,2,3,4\}}$ corresponds to doing the change of angles operations twice, respectively opposite to $\mathbf v_i$ and $\mathbf v_j.$ Later we will need the following lemma which  says that a change of angles operation does not change the signature of the Gram matrix.

\begin{lemma}\label{change} Let $G$ be the Gram matrices of the $6$-tuple 
$\{\theta_{ij},\theta_{ik},\theta_{il},\theta_{jk},\theta_{jl},\theta_{kl}\},$
and let $G^*$ be the  Gram matrix of the $6$-tuple 
$\{\pi-\theta_{ij},\pi-\theta_{ik},\pi-\theta_{il},\theta_{jk},\theta_{jl},\theta_{kl}\},$ $\{i,j,k,l\}=\{1,2,3,4\}.$
Then $G$ and $G^*$ have the same signature.
\end{lemma}
\begin{proof} Suppose $G$ has signature $(p,q).$ Then by Sylvester' Law of Inertia,
$$G=U^T\cdot I_{p,q}\cdot U$$
for some $4\times 4$ matrix $U$ with columns $\mathbf u_1,\mathbf u_2,\mathbf u_3,\mathbf u_4,$ where $I_{p,q}$ is the $4\times 4$ diagonal matrix with the first $p$ diagonal entries $1,$ the next $q$ diagonal entries $-1$ and other diagonal entries $0.$ Let $U^*$ be the matrix obtained by changing the $i$-th column $\mathbf u_i$ of $U$ to $-\mathbf u_i.$ Then we have 
$$G^*=U^{*T}\cdot I_{p,q}\cdot U^*,$$
and hence $G^*$ has the same  signature $(p,q)$  as $G$.
\end{proof}


\subsection{Edge lengths}

Next, we define the distance between the vertices, and define the edge lengths of a generalized hyperbolic tetrahedron.

The \emph{distance} $d_{ij}$ between vertices $\mathbf v_i$ and $\mathbf v_j$ is defined as follows. Let 
$G$ be the Gram matrix of a generalized hyperbolic tetrahedra $\Delta$ with vertices $\mathbf v_1,$ $\mathbf v_2,$ $\mathbf v_3$ and $\mathbf v_4,$ and let $G_{ij}$ be its $ij$-th cofactor.
\begin{enumerate}[(1)]
\item If $G_{ii}>0$ and $G_{jj}>0,$ then $\mathbf v_i,\mathbf v_j\in\mathbb S(-1),$ and $d_{ij}$ is defined to be the hyperbolic distance between their radial projections $v_i$ and $v_j$ in $\mathbb H^3,$ i.e.,
\begin{equation}\label{rr}
d_{ij}=\cosh^{-1}\frac{|G_{ij}|}{\sqrt{G_{ii}G_{jj}}}.
\end{equation}

\item  If $G_{ii}>0$ and $G_{jj}<0,$ then $\mathbf v_i\in \mathbb S(-1)$ and $\mathbf v_j\in \mathbb S(1),$ and $d_{ij}$ is the defined to be the hyperbolic distance between the radial projections $v_i$ and $\Pi_j$ in $\mathbb H^3,$ i.e.,
\begin{equation}\label{rh}
d_{ij}=\sinh^{-1}\frac{|G_{ij}|}{\sqrt{-G_{ii}G_{jj}}}.
\end{equation}

\item  If $G_{ii}<0$ and $G_{jj}<0,$ then $\mathbf v_i,\mathbf v_j\in\mathbb S(1),$  and $d_{ij}$ is defined to be the hyperbolic distance between the radial projections $\Pi_i$ and $\Pi_j$  in $\mathbb H^3,$ i.e.,
\begin{equation}\label{hh}
d_{ij}=\cosh^{-1}\frac{|G_{ij}|}{\sqrt{G_{ii}G_{jj}}}.
\end{equation}
\end{enumerate}

We define the \emph{geometric edges} of $\Delta$  as follows. If $\mathbf v_i, \mathbf v_j\in\mathbb S(-1)\cup\mathbb S(0),$ then $v_i,v_j\in\overline{\mathbb H^3}$ and the geometric edge connecting $\mathbf v_i$ and $\mathbf v_i$ is the geodesic in $\mathbb H^3$ connecting $v_i$ and $v_j;$ if $\mathbf v_i\in S(-1)\cup\mathbb S(0)$ and $\mathbf v_j\in \mathbb S(1),$  then $v_i\in \overline{\mathbb H^3}$ and $\Pi_j\subset \mathbb H^3,$ and the geometric edge connecting $\mathbf v_i$ and $\mathbf v_i$ is the shortest geodesic in $\mathbb H^3$ between $v_i$ and $\Pi_j;$ and if $\mathbf v_i, \mathbf v_j\in\mathbb S(1),$ then $\Pi_i, \Pi_j\subset\mathbb H^3,$ and the geometric edge connecting $\mathbf v_i$ and $\mathbf v_i$ is the shortest geodesic in $\mathbb H^3$ between $\Pi_i$ and $\Pi_j.$ Then the distance $d_{ij}$ between $\mathbf v_i$ and $\mathbf v_j$ is the hyperbolic length of the  geometric edge connecting them. Suggested by Proposition \ref{classification} below, we make the following

\begin{definition}[Edge lengths] \label{el}  The \emph {length} $l_{ij}$ of the geometric edge between the faces $F_i$ and $F_j$ (which is the geometric edge connecting the vertices $\mathbf v_k$ and $\mathbf v_l$) is defined by 
$$l_{ij}=\left\{
\begin{array}{rcl}
d_{kl} & \text{if } & G_{kl} > 0,\\ -d_{kl} & \text{if } & G_{kl}\leqslant 0,
\end{array}\right.$$
where $\{k,l\}=\{1,2,3,4\}\setminus\{i,j\}.$
\end{definition}

We call a geometric edge \emph{positive}, \emph{non-positive} or \emph{negative} if its length is respectively so. 
\begin{proposition}\label{classification}
Let 
$G$
 be the Gram matrix of a generalized hyperbolic tetrahedra $\Delta$ with vertices $\mathbf v_1,$ $\mathbf v_2,$ $\mathbf v_3$ and $\mathbf v_4,$ and let $G_{ij}$ be the  $ij$-th cofactor of $G.$
 
 \begin{enumerate}[(1)]
 \item If $G_{ii}>0$ and $G_{jj}>0,$  or equivalently $\{\mathbf v_i,\mathbf v_j\}\subset  \mathbb S(-1)=\mathbb H^3_+\cup \mathbb H^3_-,$ then $G_{ij}\neq 0.$   \begin{enumerate}[(a)]
\item If $G_{ij}>0,$ then $\{\mathbf v_i,\mathbf v_j\}\subset \mathbb H^3_+$ or  $ \mathbb H^3_-.$ (See Figure \ref{nonproper1} (a) for the radial projection in $\overline{\mathbb H^3}.$)
 \item If $G_{ij}<0,$ then  $\{\mathbf v_i, -\mathbf v_j\}\subset \mathbb H^3_+$ or  $\mathbb H^3_-.$ (See Figure \ref{nonproper1} (b) for the radial projection in $\overline{\mathbb H^3}.$)
  \end{enumerate}  
  \begin{figure}[htbp]
\centering
\includegraphics[scale=0.08]{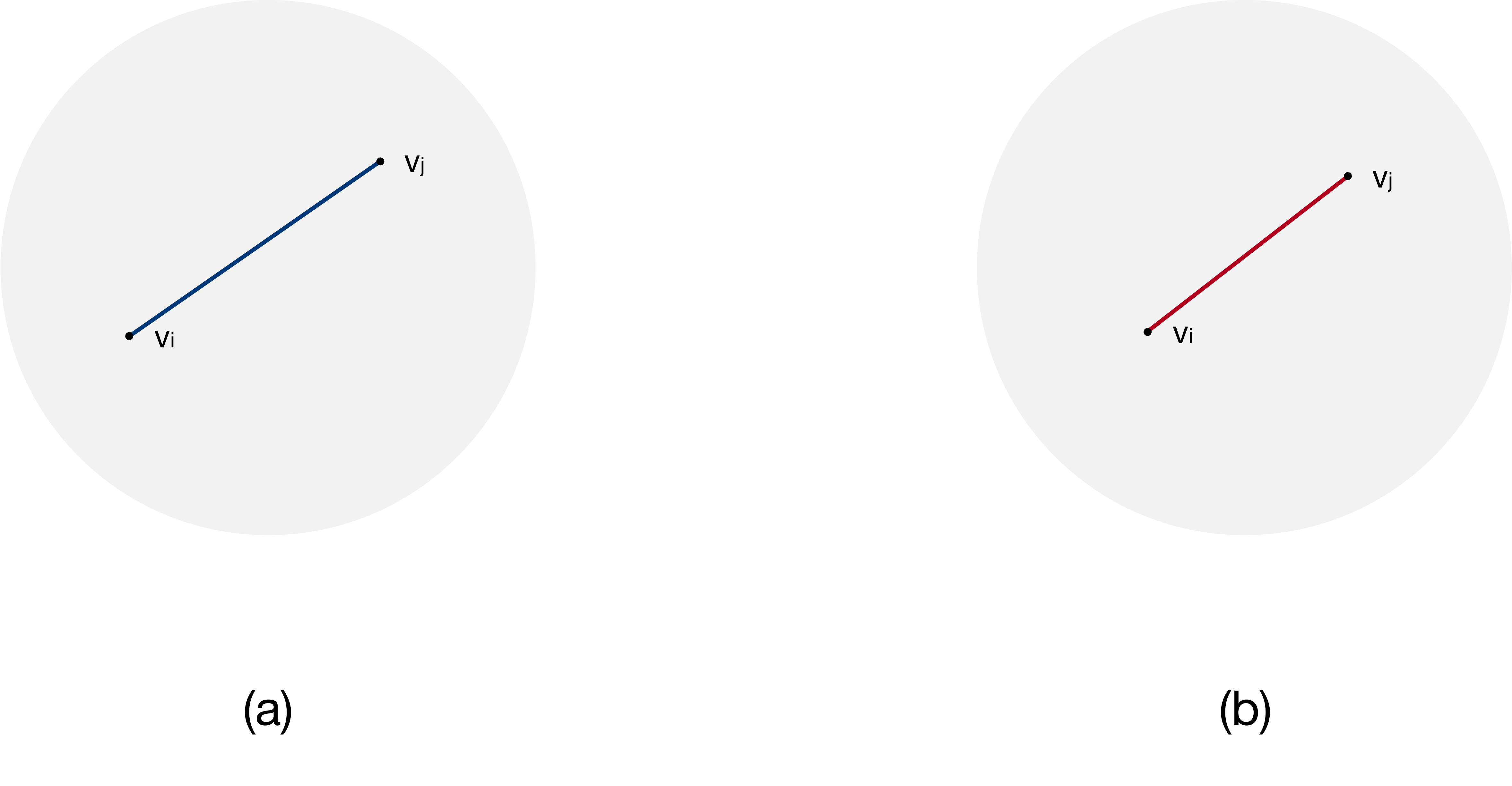}
\caption{The positive geometric edge in (a) is colored in blue, and the negative geometric edge in (b) is colored in red.}
\label{nonproper1}
\end{figure}

 \item Suppose $G_{ii}>0$ and $G_{jj}<0,$ or equivalently $\mathbf v_i\in \mathbb S(-1)=\mathbb H^3_+\cup \mathbb H^3_1$ and $\mathbf v_j\in\mathbb S(1).$  \begin{enumerate}[(a)]
 \item  If  $G_{ij}>0,$ then  $v_i$ and $v_j$ are on different sides of $\Pi_j.$ (See Figure \ref{nonproper2} (a).)
 \item  If $G_{ij}\leqslant 0,$ then $v_i$ and $v_j$ are on the same side of  $\Pi_j.$   (See Figure \ref{nonproper2} (b).)
  \end{enumerate}
  \begin{figure}[htbp]
\centering
\includegraphics[scale=0.08]{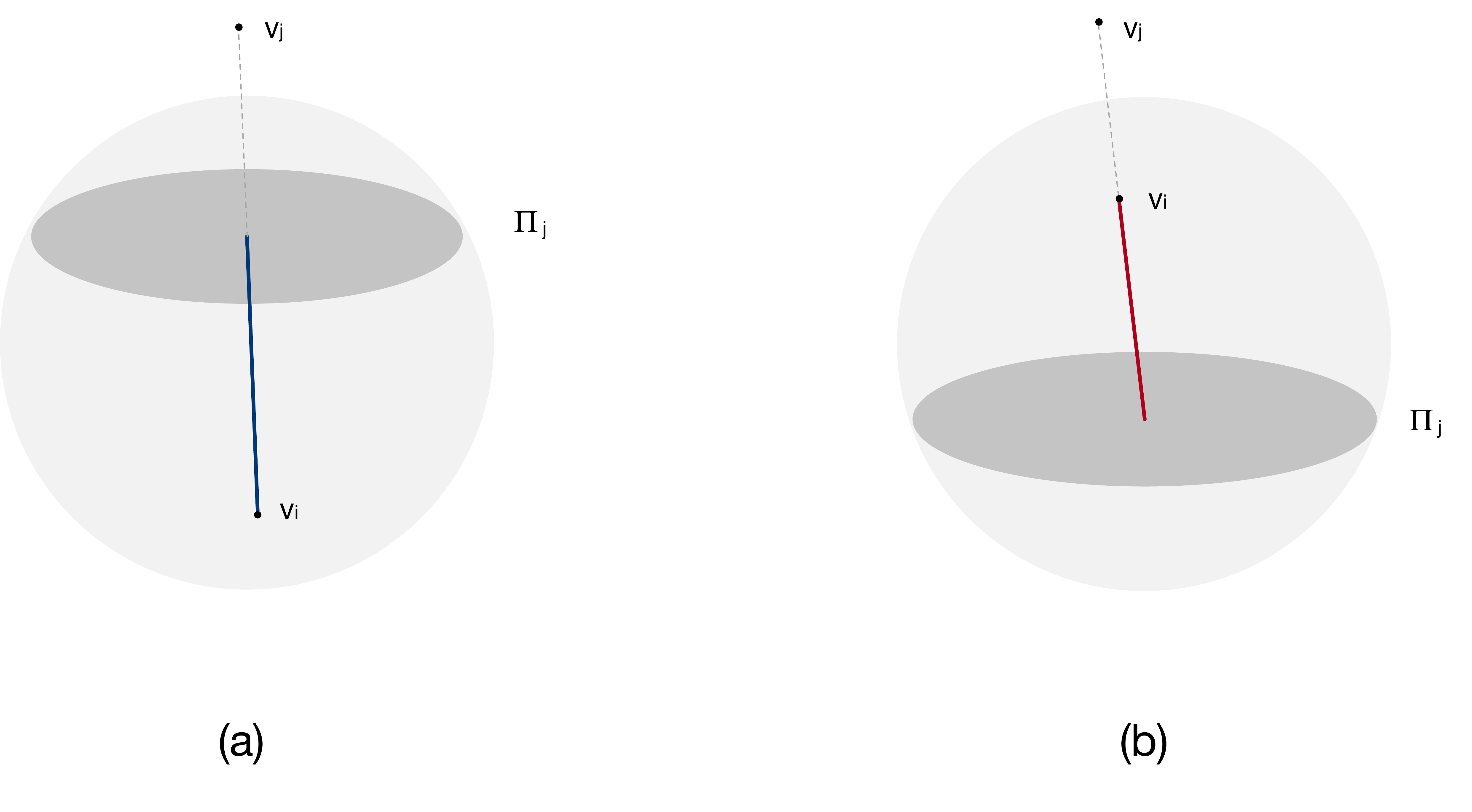}
\caption{The positive geometric edge in (a) is colored in blue, and the non-positive geometric edge in (b) is colored in red.}
\label{nonproper2}
\end{figure}

  \item If $G_{ii}<0$ and $G_{jj}<0,$ or equivalently $\{\mathbf v_i,\mathbf v_j\}\subset \mathbb S(1),$ then $G_{ij}\neq 0.$    
   \begin{enumerate}[(a)]
   \item If $G_{ij}>0,$ then $L_{ij}^+$ intersects $\Pi_i$ and $\Pi_j.$ (See Figure \ref{nonproper3} (a).)
 \item If $G_{ij}<0,$ then  $L_{ij}^-$ intersects $\Pi_i$ and $\Pi_j.$  (See Figure \ref{nonproper3} (b).)
   \end{enumerate}
     \begin{figure}[htbp]
\centering
\includegraphics[scale=0.08]{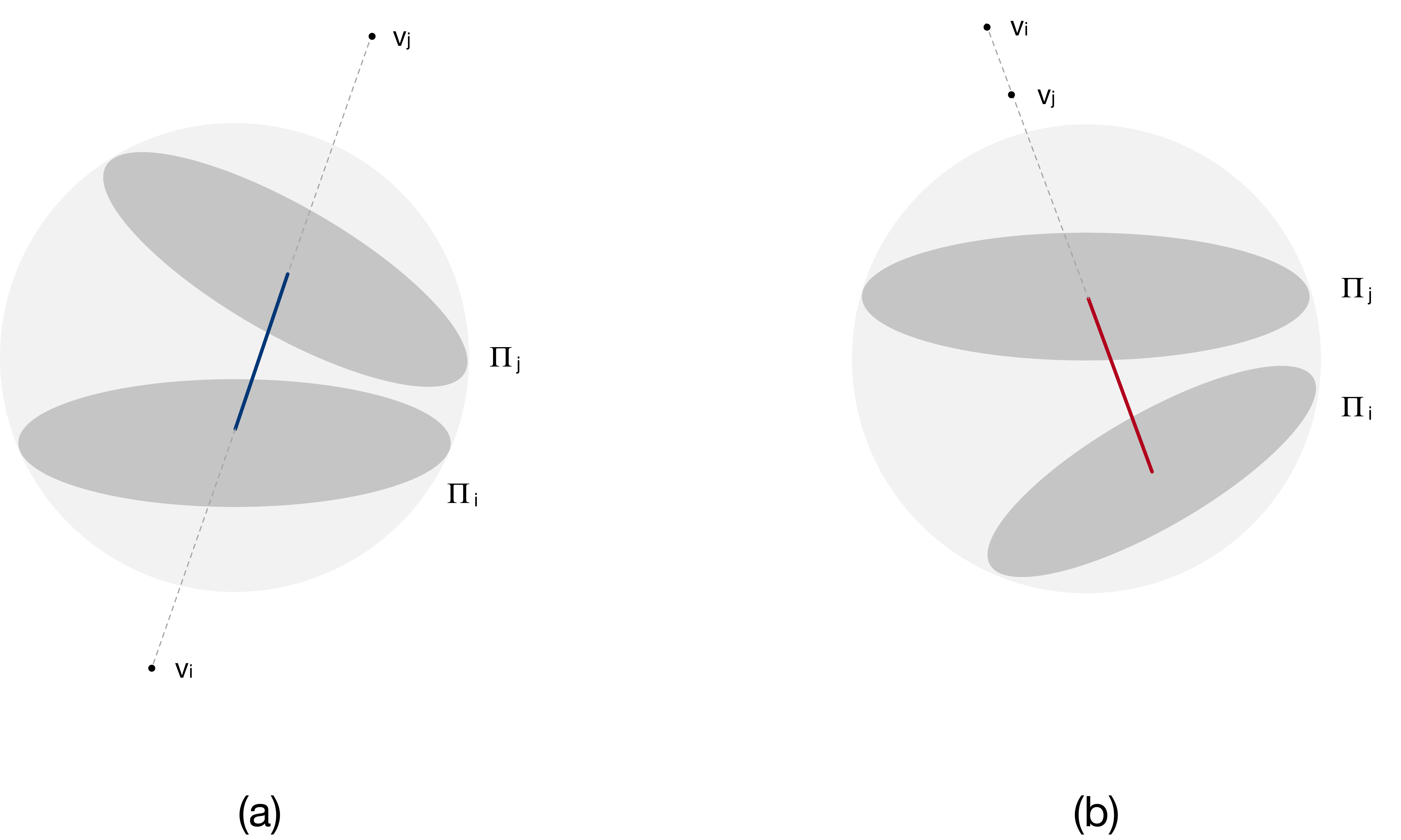}
\caption{The positive geometric edge in (a) is colored in blue, and the negative geometric edge in (b) is colored in red.}
\label{nonproper3}
\end{figure}
\end{enumerate}
\end{proposition}

\begin{proof} For (1), let $d_{ij}$ be the distance between $\mathbf v_i$ and $\mathbf v_j,$ i.e., the hyperbolic distance between $v_i$ and $v_j$ in $\mathbb H^3.$ If $\mathbf v_i, \mathbf v_j\in\mathbb H^3_+$ or $\mathbb H^3_-,$ then
$$\langle \mathbf v_i, \mathbf v_j\rangle = -\cosh d_{ij}<0;$$
and if 
$\mathbf v_i \in \mathbb H^3_+$ and $\mathbf v_j\in\mathbb H^3_-,$
then
$$\langle \mathbf v_i, \mathbf v_j\rangle = \cosh d_{ij}>0.$$
Then the result follows from the computation that 
$$\langle \mathbf v_i, \mathbf v_j\rangle=-\frac{G_{ij}}{\sqrt{G_{ii}G_{jj}}}.$$

For (2), we use the computation 
$$\langle \mathbf v_i, \mathbf v_j\rangle=-\frac{G_{ij}}{\sqrt{-G_{ii}G_{jj}}}.$$
If $G_{ij}>0,$ then $\langle \mathbf v_i, \mathbf v_j\rangle<0,$ and $\mathbf v_i$ and $\mathbf v_j$ are on different sides of the hyperplane $\mathbf \Pi_j$ in $\mathbb E^{3,1},$ hence $L^+_{ij}$ intersects $\mathbf \Pi_j.$ As a consequence, the radial projection $L^+_{ij}$ intersects the radial projection $\Pi_j.$ If $G_{ij}\leqslant 0,$ then $\langle \mathbf v_i, \mathbf v_j\rangle\geqslant 0,$ and $\mathbf v_i$ and $\mathbf v_j$ are on the same side of  $\mathbf \Pi_j,$ and hence $L^-_{ij}$ intersects it.  As a consequence, the radial projection $L^-_{ij}$ intersects the radial projection $\Pi_j.$

For (3), let $d_{ij}$ be the distance between $\mathbf v_i$ and $\mathbf v_j,$ i.e., the hyperbolic distance between the radial projections $\Pi_i$ and $\Pi_j$ in $\mathbb H^3.$ Then comparing the formulas 
$$|\langle \mathbf v_i, \mathbf v_j\rangle| = \cosh d_{ij}$$
and 
$$\langle \mathbf v_i, \mathbf v_j\rangle=-\frac{G_{ij}}{\sqrt{G_{ii}G_{jj}}},$$
we have $G_{ij}\neq 0.$ Now if $G_{ij}>0,$ then $\langle \mathbf v_i, \mathbf v_j\rangle<0,$ and $\mathbf v_i$ and $\mathbf v_j$ are on different sides of the hyperplanes $\mathbf \Pi_i$ and  $\mathbf \Pi_j$ in $\mathbb E^{3,1},$ hence $L^+_{ij}$ intersects $\mathbf \Pi_i$ and $\mathbf \Pi_j.$ As a consequence, the radial projection $L^+_{ij}$ intersects  the radial projections $\Pi_i$ and  $\Pi_j.$ If $G_{ij}<0,$ then $\langle \mathbf v_i, \mathbf v_j\rangle>0,$ and $\mathbf v_i$ and $\mathbf v_j$ are on the same side of $\mathbf \Pi_i$ and $\mathbf \Pi_j,$ and hence $L^-_{ij}$ intersects them.  As a consequence, the radial projection $L^-_{ij}$ intersects the radial projections $\Pi_i$ and $\Pi_j.$ 
\end{proof}


\subsection{A classification}

We give a classification of generalized hyperbolic tetrahedra $\Delta$ in terms of their Gram matrix. Let $\mathbf v_1,\dots,\mathbf v_4$ be the vertices of $\Delta.$ Recall that the change of angles operation opposite to the vertex $\mathbf v_i$ sends the set of dihedral angles  $\{\theta_{ij}, \theta_{ik},\theta_{il},\theta_{jk},\theta_{jl},\theta_{kl}\}$ to $\{\pi-\theta_{ij}, \pi-\theta_{ik},\pi-\theta_{il},\theta_{jk},\theta_{jl},\theta_{kl}\},$ $\{i,j,k,l\}=\{1,2,3,4\},$ which are the dihedral angles of the generalized hyperbolic tetrahedron with set of vertices $\{-\mathbf v_i,\mathbf v_j,\mathbf v_k,\mathbf v_l\}.$

\begin{proposition}\label{class} Let $G$ be the Gram matrix of a generalized hyperbolic tetrahedron.

\begin{enumerate}[(1)]
\item If $G_{ii}, G_{jj}, G_{kk},G_{ll}\geqslant 0,$ then there are the following cases. See Figure \ref{truncated}.
\begin{enumerate}[(a)]
\item $G_{ij}, G_{ik}, G_{il}, G_{jk}, G_{jl}, G_{kl}>0.$ 
\item  $G_{il}, G_{jl}, G_{kl}<0$ and $G_{ij}, G_{ik},G_{jk}>0.$
\item $G_{ik},G_{il},G_{jk},G_{jl}<0$ and $G_{ij},G_{kl}>0.$
\item  Otherwise, by doing a sequence of change of angles operations, we are in case (a), (b) or (c). 
\begin{figure}[htbp]
\centering
\includegraphics[scale=0.3]{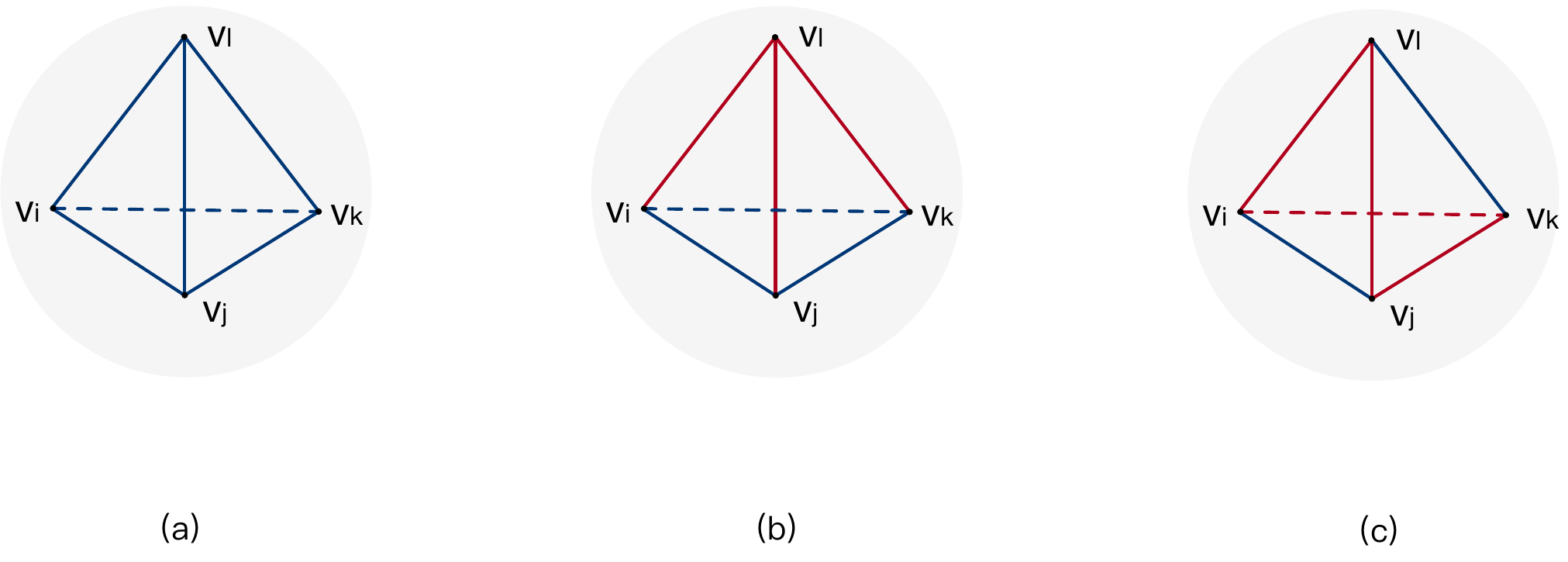}
\caption{In (a), all the vertices $\mathbf v_i,\mathbf v_j,\mathbf v_k,\mathbf v_l$ are positive. In (b), the vertices $\mathbf v_i,$ $\mathbf v_j$ and $\mathbf v_k$ are positive and $\mathbf v_l$ is negative. In (c), the vertices $\mathbf v_i$ and $\mathbf v_j$ are positive and $\mathbf v_k$ and $\mathbf v_l$ are negative. The positive geometric edges are colored in blue and the negative geometric edges are colored in red.}
\label{truncated}
\end{figure}
\end{enumerate}

\item If $G_{ii}, G_{jj}, G_{kk}\geqslant 0$ and $G_{ll}<0,$  then there are the following cases.  See Figure \ref{truncated2}.
\begin{enumerate}[(a)]
\item $G_{ij}, G_{ik}, G_{il}, G_{jk}, G_{jl}, G_{kl}>0.$ 
\item $G_{kl}\leqslant 0$ and $G_{ij}, G_{ik}, G_{il}, G_{jk}, G_{jl}>0.$ 
\item $G_{jl}, G_{kl}\leqslant 0$ and $G_{ij}, G_{ik}, G_{il}, G_{jk}>0.$ 
\item $G_{il}, G_{jl}, G_{kl}\leqslant 0$ and $G_{ij}, G_{ik}, G_{jk}>0.$ 
\item Otherwise, by doing a sequence of change of angles operations, we are in case (a) or (b).
\begin{figure}[htbp]
\centering 
\includegraphics[scale=0.3]{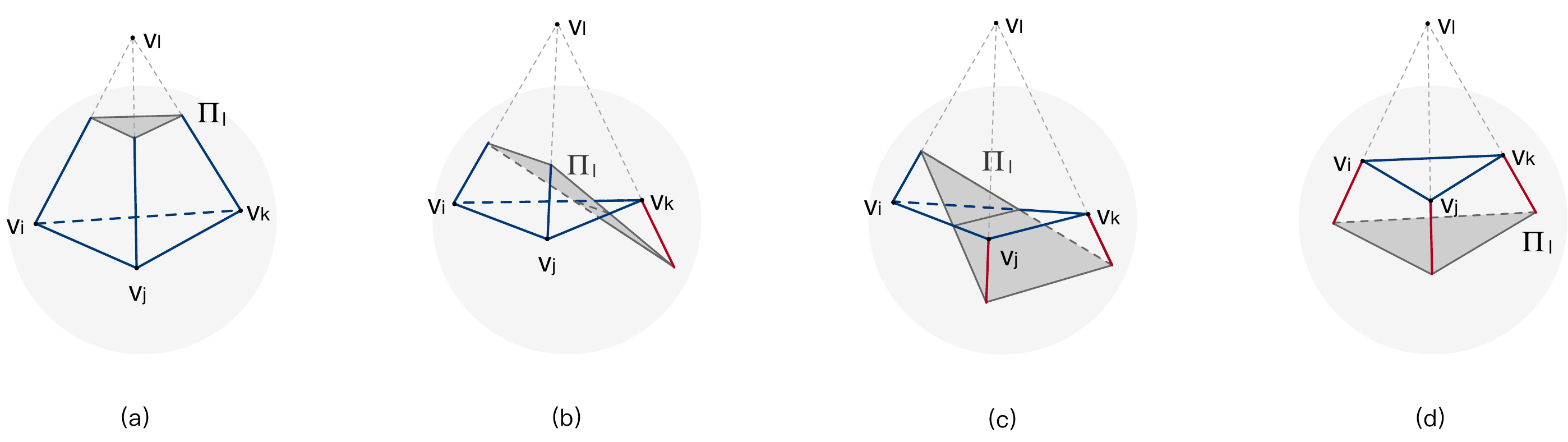}
\caption{All the regular or ideal vertices $\mathbf v_i,$ $\mathbf v_j$ and $\mathbf v_k$  are positive. The triangles of truncation are colored in grey,  the positive geometric edges are colored in blue and the non-positive geometric edges are colored in red.}
\label{truncated2}
\end{figure}
\end{enumerate}

\item  If $G_{ii}, G_{jj}\geqslant 0$ and $G_{kk}, G_{ll}<0,$  then there are the following cases. See Figure \ref{truncated4}.
\begin{enumerate}[(a)]
\item $G_{ij}, G_{ik}, G_{il}, G_{jk}, G_{jl}, G_{kl}>0.$ 
\item $G_{jk}\leqslant 0$ and $G_{ij}, G_{ik}, G_{il}, G_{jl}, G_{kl}>0.$ 
\item $G_{ik}, G_{jk}\leqslant 0$ and $G_{ij}, G_{il}, G_{jl}, G_{kl}>0.$
\item $G_{il}, G_{jk}\leqslant 0$ and $G_{ij}, G_{ik}, G_{jl}, G_{kl}>0.$
\item  Otherwise, by doing a sequence of change of angles operations,  we are in case (a), (b), (c) or (d).  
\begin{figure}[htbp]
\centering
\includegraphics[scale=0.3]{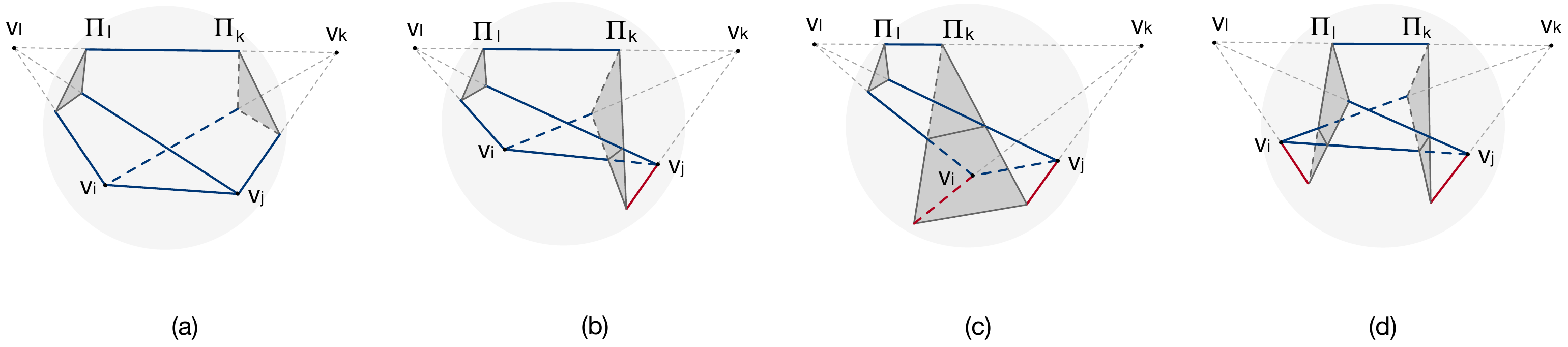}
\caption{All the regular or ideal vertices $\mathbf v_i$ and $\mathbf v_j$ are positive. The triangles of truncation are colored in grey,  the positive geometric edges are colored in blue and the negative geometric edges are colored in red.}
\label{truncated4}
\end{figure}
\end{enumerate}

\item If $G_{ii}\geqslant 0$ and $G_{jj}, G_{kk}, G_{ll}<0,$  then there are the following cases.  See Figure \ref{truncated5}.
\begin{enumerate}[(a)]
\item $G_{ij}, G_{ik}, G_{il}, G_{jk}, G_{jl}, G_{kl}>0.$ 
\item  $G_{ij}\leqslant 0$ and  $G_{ik}, G_{il}, G_{jk}, G_{jl}, G_{kl}>0.$ 
\item $G_{kl}<0$ and $G_{ij}, G_{ik}, G_{il}, G_{jk}, G_{jl}>0.$ 
\item $G_{ij}\leqslant 0,$  $G_{kl}<0$ and  $G_{ik}, G_{il}, G_{jk}, G_{jl}>0.$ 
\item  Otherwise, by doing a sequence of change of angles operations, we are in case (a), (b), (c) or (d). 
\begin{figure}[htbp] 
\centering
\includegraphics[scale=0.095]{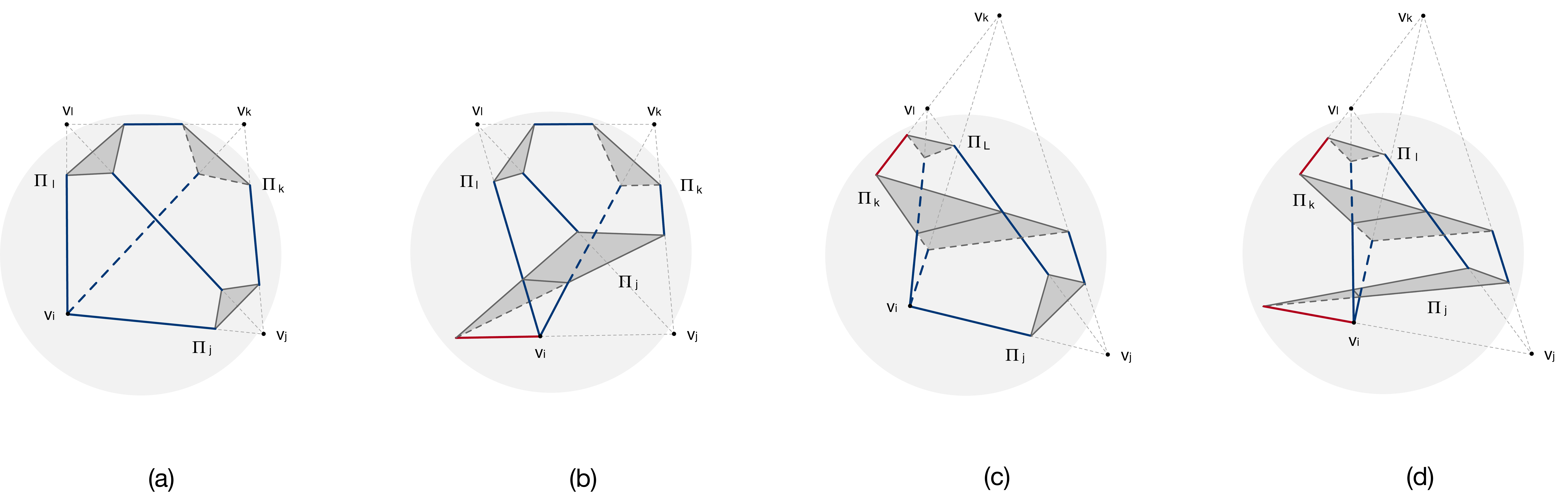}
\caption{The regular or ideal vertice $\mathbf v_i$ is positive. The triangles of truncation are colored in grey,  the positive geometric edges are colored in blue and the non-positive geometric edges are colored in red.}
\label{truncated5}
\end{figure}
\end{enumerate}

\item If $G_{ii}, G_{jj}, G_{kk}, G_{ll}<0,$   then there are the following cases. See Figure \ref{truncated3}.
\begin{enumerate}[(a)]
\item $G_{ij}, G_{ik}, G_{il}, G_{jk}, G_{jl}, G_{kl}>0.$ 
\item $G_{kl}<0$ and $G_{ij}, G_{ik}, G_{il}, G_{jk}, G_{jl}>0.$ 
\item $G_{ij}, G_{kl}<0$ and  $G_{ik}, G_{il}, G_{jk}, G_{jl}>0.$ 
\item   Otherwise, by doing a sequence of change of angles operations,   we are in case (a), (b) or (c).
\begin{figure}[htbp]
\centering
\includegraphics[scale=0.09]{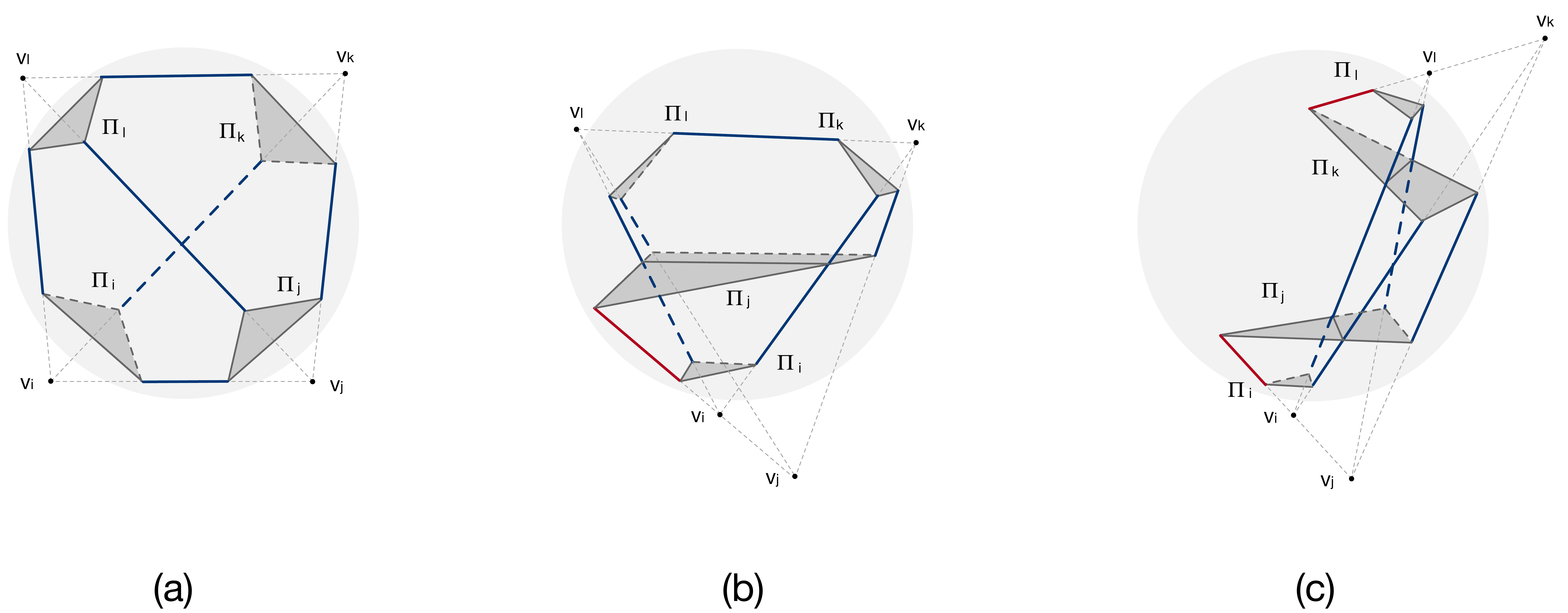}
\caption{The triangles of truncation are colored in grey,  the positive geometric edges are colored in blue and the negative geometric edges are colored in red.}
\label{truncated3}
\end{figure}
\end{enumerate}
\end{enumerate}
\end{proposition}

\begin{remark} In (2), (c) actually  can be obtained from (b) be doing a change of angles operation along the face $F_l,$ and (d)  can be obtained from (a) be doing a change of angles operation along the face $F_l.$ We still list these two redundant cases here for the purpose of proving Proposition \ref{connected} in Section \ref{caa}, which is a crucial step in the proof of Theorem \ref{asymp1} (1). 
\end{remark}

\begin{proof}[Proof of Proposition \ref{class}] In (1), all the vertices are regular or ideal. Then up to a sequence of  change of angles operations, either all the vertices are in $\mathbb H^3_+\cup\mathbb L^3_+$ as in (a); or exactly one vertex is  in $\mathbb H^3_-\cup\mathbb L^3_-$ as in (b); or exactly two vertices are  in $\mathbb H^3_-\cup\mathbb L^3_-$ as in (c). 

In (2), (3), (4), (5), up to a sequence of change of angles operations opposite to the negative regular vertices, we may assume that all the regular vertices are positive as in (a), (b), (c), (d) of each of the cases. Then up to the change of angles operations around the faces opposite to the hyperideal vertices, it is to look at  in the projective model the relative positions between the $v_i$'s for regular or ideal vertices and $ \Pi_j$'s for hyperideal vertices in $\overline{\mathbb H^3}.$

In (2), $\mathbf v_i, \mathbf v_j,\mathbf v_k$ are regular or ideal, and $\mathbf v_l$ is hyperideal. Then in the projective model, either all $v_i,  v_j, v_k$ are on one side of $\Pi_l$ as in (a); or one of them, say, $v_k$ and the other two $v_i, v_j$ are one different sides of $\Pi_l$ as in (b). 

In (3), $\mathbf v_i, \mathbf v_j$ are regular or ideal, and  $\mathbf v_k, \mathbf v_l$ are hyperideal. Then  in the projective model, $ \Pi_k$ and $ \Pi_l$ divide $\overline{\mathbb H^3}$ into three piece $D_k$ that is only adjacent to $\Pi_k,$ $D_l$ that is only adjacent to $\Pi_l,$ and $D_{kl}$ that is adjacent to both $\Pi_k$ and $\Pi_l.$ Then either $v_i,v_j$ are in $D_{kl}$ as in (a); or $v_i$ is in $D_{kl}$ and $v_j$ is in, say, $D_k$ as in (b);  or $v_i,v_j$ are in, say, $D_k$ as in (c); or $v_i$ is in, say, $D_k$ and $v_j$ is in $D_l$ as in (d).

In (4), $\mathbf v_i$ is regular or ideal, and $\mathbf v_j, \mathbf v_k, \mathbf v_l$ are hyperideal. Then in the projective model, $\Pi_j,$ $\Pi_k$ and $ \Pi_l$ divide $\overline{\mathbb H^3}$ into four piece in two different ways. In one way as in (a) and (b), there is a piece $D_{jkl}$ that is adjacent to all of $\Pi_j,\Pi_k,\Pi_l,$ and there are three pieces $D_j,$ $D_k,$ $D_l$ that are respective only adjacent to $\Pi_j,$ $\Pi_k,$ $\Pi_l.$ Then either $v_i$ is in $D_{jkl}$ as in (a); or $v_i$ is in, say, $D_j$ as in (b). in the other way as in (c) and (d), there are two pieces $D_{jk},$ $D_{kl}$ that are respectively adjacent to $\Pi_j,\Pi_k$ and $\Pi_k,\Pi_l,$ and to pieces $D_j,$ $D_l$ that are respectively only adjacent to $\Pi_j$ and $\Pi_l.$ Then either $v_i$ is in, say, $D_{jk}$ as in (c); or $v_i$ is in, say, $D_j$ as in (d). 

In (5), all the vertices are hyperideal, and $\Pi_1,\dots,\Pi_4$ divede $\overline{\mathbb H^3}$ in three different ways as in (a), (b) and (c). 
\end{proof}


\subsection{Volume and the Schl\"afli formula}

Finally, let us define the volume of a generalized hyperbolic tetrahedron. A \emph{geometric face} of $\Delta$ is a region in the radial projection of the faces $F_i$'s bounded by the geometric edges of $\Delta$ and the intersections $\{F_i\cap \Pi_j\}_{i,j\in\{1,2,3,4\}},$ and a \emph{geometric piece} of $\Delta$ is a region in $\mathbb H^3$ bounded by the geometric faces of $\Delta$ and the planes $\{\Pi_i\}_{i\in\{1,2,3,4\}}.$ Then each geometric piece of $\Delta$ is a polyhedron in $\mathbb H^3.$ We also observe that the intersection of each $\Pi_i$ with the union of the geometric pieces is a hyperbolic triangle which we call a \emph{triangle of truncation}, and that each geometric piece of $\Delta$ is contained in exactly one of $\Delta_0,$ $\Delta_i$'s and $\Delta_{ij}$'s. See Figures \ref{truncated}, \ref{truncated2},  \ref{truncated4}, \ref{truncated5}, \ref{truncated3}.

\begin{definition}[Volume]\label{v}  For a geometric piece $P$ of a generalized hyperbolic tetrahedron $\Delta,$ let $\mathrm{Vol}(P)$ be the hyperbolic volume of $P$ considered as a hyperbolic polyhedron. If $P$ is contained in $\Delta_0\cup\Delta_{12}\cup\Delta_{13}\cup\Delta_{14},$ then let
 $$\mathrm{Vol}^P(\Delta)=\mathrm{Vol}(P);$$
 and  if $P$ is contained in $\Delta_1\cup\Delta_2\cup\Delta_3\cup\Delta_4,$ then let 
 $$\mathrm{Vol}^P(\Delta)=-\mathrm{Vol}(P).$$
The \emph{volume} $\mathrm{Vol}(\Delta)$ of $\Delta$ is defined as
$$\mathrm {Vol}(\Delta)=\sum_{P}\mathrm{Vol}^P(\Delta),$$
where the sum if over all the geometric pieces of $\Delta.$
\end{definition}

\begin{example} In Figure \ref{truncated},  let $P$ be the only geometric piece. Then for (a) and (c), 
$$\mathrm{Vol}(\Delta)=\mathrm{Vol}(P),$$
and for (b),
$$\mathrm{Vol}(\Delta)=-\mathrm{Vol}(P).$$
\end{example}

\begin{example} In Figure \ref{truncated4} (a), let $P$ be the only geometric piece. Then 
$$\mathrm{Vol}(\Delta)=\mathrm{Vol}(P).$$
In Figure \ref{truncated4} (b), let $P_{ikl}$ be the geometric piece bounded by $v_i,$ $\Pi_k$ and $\Pi_l$ and let $P_{jk}$ be the geometric piece bounded by $v_j$ and $\Pi_k,$ then 
$$\mathrm{Vol}(\Delta)=\mathrm{Vol}(P_{ikl})-\mathrm{Vol}(P_{jk}).$$
In Figure \ref{truncated4} (c), let $P_{kl}$ be the geometric piece bounded by $\Pi_k$ and $\Pi_l$ and let $P_{ijk}$ be the geometric piece bounded by $v_i,$ $v_j$ and $\Pi_k,$ then 
$$\mathrm{Vol}(\Delta)=\mathrm{Vol}(P_{kl})-\mathrm{Vol}(P_{ijk}).$$
In Figure \ref{truncated4} (d),  let $P_{il}$ be the geometric piece bounded by $v_i$ and $\Pi_l,$ let $P_{kl}$ be the geometric piece bounded by $\Pi_k$ and $\Pi_l$ and let $P_{jk}$ be the geometric piece bounded by $v_j$ and $\Pi_k,$ then 
$$\mathrm{Vol}(\Delta)=-\mathrm{Vol}(P_{il})+\mathrm{Vol}(P_{kl})-\mathrm{Vol}(P_{jk}).$$
\end{example}

\begin{example} In each generalized hyperbolic tetrahedron in Figure \ref{eg3}, let $P_{ij}$ be the geometric piece bounded   by $\Pi_i$ and $\Pi_j.$ 
Then for (a) and (c), $$\mathrm{Vol}(\Delta)=-\mathrm{Vol}(P_{12})+\mathrm{Vol}(P_{23})-\mathrm{Vol}(P_{34}),$$ 
and for (b), $$\mathrm{Vol}(\Delta)=\mathrm{Vol}(P_{12})-\mathrm{Vol}(P_{23})+\mathrm{Vol}(P_{34}).$$ 
\begin{figure}[htbp]
\centering
\includegraphics[scale=0.09]{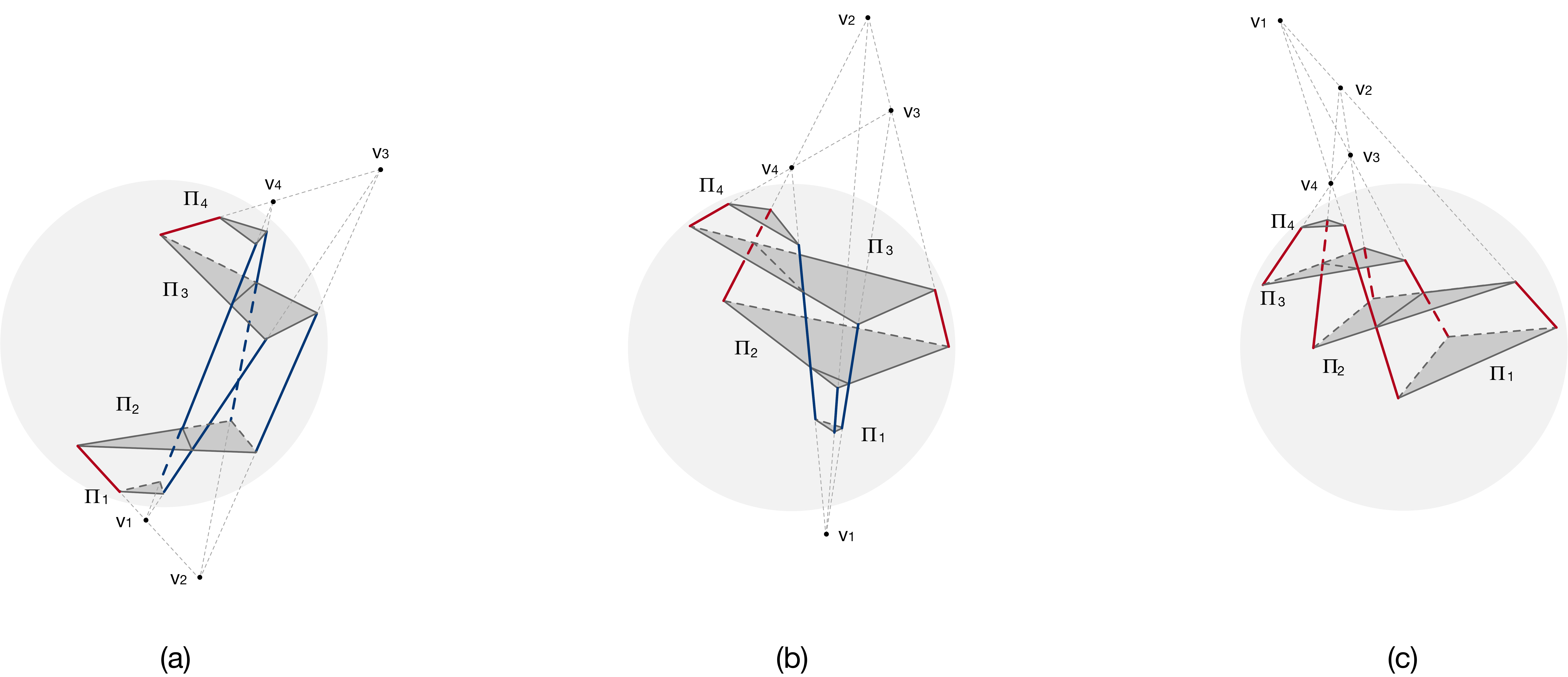}
\caption{The triangles of truncation are colored in grey,  the positive geometric edges are colored in blue and the non-positive geometric edges are colored in red.}
\label{eg3}
\end{figure}
\end{example}

\begin{proposition}[Schl\"afli Formula] \label{Schlafli} For a generalized hyperbolic tetrahedron $\Delta$ with either regular or hyperideal vertices, let $\mathrm{Vol}(\Delta)$ be the volume of $\Delta$ as defined in Definition \ref{v}, and for $\{i,j\}\subset\{1,\dots, 4\},$ let $\theta_{ij}$ and $l_{ij}$ be the dihedral angle at and the length of the edge $e_{ij}$ between the faces $F_i$ and $F_j$ respectively defined in Definition \ref{da} and Definition \ref{el}. Then 
\begin{equation}\label{Sch}
\frac{\partial\mathrm{Vol}(\Delta)}{\partial\theta_{ij}}=-\frac{l_{ij}}{2}.
\end{equation}
\end{proposition}

To prove Proposition \ref{Schlafli}, we need the following Lemma \ref{la} whose proof follows immediately from Definition \ref{el} and Proposition \ref{classification}. 

\begin{lemma}\label{la} Let $\Delta$ be a generalized hyperbolic tetrahedron with only positive regular vertices, and let $e_{st}$ be a geometric edge of $\Delta$ intersecting a geometric piece $P,$ $\{s,t\}\subset\{1,2,3,4\}.$  
\begin{enumerate}[(1)]
\item If $P$ is contained in $\Delta_0,$ then $e_{st}$ is positive.

\item If  $P$ is contained in $\Delta_i,$ $i=1,2,3,4,$ then $e_{st}$ is positive for $\{s,t\}=\{i,j\}$ with $j\in\{1,2,3,4\}\setminus\{i\};$ and $e_{st}$ is non-positive for $\{s,t\}=\{j,k\}$ with $\{j,k\}\subset\{1,2,3,4\}\setminus\{i\}.$

\item If  $P$ is contained in $\Delta_{ij},$ $\{i,j\}=\{1,2\},$ $\{1,3\}$ or $\{1,4\},$ then $e_{st}$ is positive for $\{s,t\}=\{i,j\}$ or $\{1,2,3,4\}\setminus\{i,j\};$ and $e_{st}$ is non-positive for $\{s,t\}\neq\{i,j\}$ nor $\{1,2,3,4\}\setminus\{i,j\}.$ 
\end{enumerate}
\end{lemma}

\begin{proof}[Proof of Proposition \ref{Schlafli}]  We first consider the special case that all the regular  vertices of $\Delta$ are in $\mathbb H^3_+\cup\mathbb L^3_+.$ Then
\begin{equation}\label{Vsum}
\mathrm {Vol}(\Delta)=\sum_{P}\mathrm{Vol}^P(\Delta),
\end{equation}
where the sum is over all the geometric pieces of $\Delta,$ and $\mathrm{Vol}^P(\Delta)$ is as defined in Definition \ref{v}. For each $P,$  let $e^P_{ij}$ be the intersection of $P$ with the geometric edge $e_{ij}$ of $\Delta,$ which is between the faces $F_i$ and $F_j.$ Let $\theta^P_{ij}$ be the dihedral angle of $P$ at $e^P_{ij}$ considered as a hyperbolic polyhedron, and defined the length $l^{P}_{ij}$ of $e^P_{ij}$ as follows. 
If $d^P_{ij}$ be the hyperbolic length of $e^P_{ij},$ then $l^{P}_{ij}=d^P_{ij}$ if $e_{ij}$ is a positive geometric edge, and $l^{P}_{ij}=-d^P_{ij}$ if $e_{ij}$ is a non-positive geometric edge. 

Then for each $\{i,j\}\subset\{1,2,3,4\},$ 
\begin{equation}\label{lsum}
l_{ij}=\sum_P l^P_{ij},
\end{equation}
where the sum is over all the geometric pieces of $\Delta.$ 

Next we prove that for each geometric piece $P$ of $\Delta,$ 
\begin{equation}\label{SchP}
\frac{\partial\mathrm{Vol}^P(\Delta)}{\partial\theta_{ij}}=-\frac{l^P_{ij}}{2}.
\end{equation}
Then the result follows from (\ref{Vsum}), (\ref{lsum}) and (\ref{SchP}).

To prove (\ref{SchP}) we have the following three cases:
\begin{enumerate}[(1)]
\item If $P$ is contained in $\Delta_0,$ then by Lemma \ref{la}, we have $$(\mathrm{Vol}^P(\Delta),\theta^P_{ij},l^P_{ij})=(\mathrm{Vol}(P), \theta_{ij}, d^P_{ij}),$$ and the Schl\"afli formula for $P$ considered as a hyperbolic polyhedron implies
$$\frac{\partial \mathrm{Vol}^P(\Delta)}{\partial \theta_{ij}}=\frac{\partial \mathrm{Vol}(P)}{\partial \theta^P_{ij}}=-\frac{d^P_{ij}}{2}=-\frac{l^P_{ij}}{2}.$$

\item If $P$ is contained in $\Delta_i,$  $i=1,2,3,4,$  then  for $j\in\{1,2,3,4\}\setminus\{ i\},$  by Proposition \ref{dac} and Lemma \ref{la} we have $$(\mathrm{Vol}^P(\Delta), \theta^P_{ij}, l^P_{ij})=(-\mathrm{Vol}(P), \pi-\theta_{ij}, d^P_{ij}),$$ 
and the Schl\"afli formula for $P$ considered as a hyperbolic polyhedron implies
$$\frac{\partial \mathrm{Vol}^P(\Delta)}{\partial \theta_{ij}}=\frac{\partial (-\mathrm{Vol}(P))}{\partial (\pi-\theta^P_{ij})}=\frac{\partial \mathrm{Vol}(P)}{\partial \theta^P_{ij}}=-\frac{d^P_{ij}}{2}=-\frac{l^P_{ij}}{2};$$
and for $\{j,k\}\subset\{1,2,3,4\}\setminus\{ i\},$  by Proposition \ref{dac} and Lemma \ref{la} we have
$$(\mathrm{Vol}^P(\Delta), \theta^P_{jk}, l^P_{jk})=(-\mathrm{Vol}(P), \theta_{jk}, -d^P_{jk}),$$ 
and the Schl\"afli formula for $P$ considered as a hyperbolic polyhedron implies
$$\frac{\partial \mathrm{Vol}^P(\Delta)}{\partial \theta_{jk}}=\frac{\partial (-\mathrm{Vol}(P))}{\partial \theta^P_{jk}}=\frac{d^P_{ij}}{2}=-\frac{l^P_{ij}}{2}.$$

\item If $P$ is contained in $\Delta_{ij},$  $\{i,j\}=\{1,2\},$ $\{1,3\}$ or $\{1,4\},$ then  for $\{s,t\}=\{i,j\}$ or $\{1,2,3,4\}\setminus\{i,j\},$  by Proposition \ref{dac} and Lemma \ref{la} we have
 $$(\mathrm{Vol}^P(\Delta), \theta^P_{st}, l^P_{st})=(\mathrm{Vol}(P), \theta_{ij}, d^P_{st}),$$ 
and the Schl\"afli formula for $P$ considered as a hyperbolic polyhedron implies
$$\frac{\partial \mathrm{Vol}^P(\Delta)}{\partial \theta_{st}}=\frac{\partial \mathrm{Vol}(P)}{\partial \theta^P_{st}}=-\frac{d^P_{st}}{2}=-\frac{l^P_{st}}{2};$$
and $\{s,t\}\neq \{i,j\}$ nor $\{1,2,3,4\}\setminus\{i,j\},$ by Proposition \ref{dac} and Lemma \ref{la} we have
$$(\mathrm{Vol}^P(\Delta), \theta^P_{st}, l^P_{st})=(\mathrm{Vol}(P), \pi-\theta_{st}, -d^P_{st}),$$ 
and the Schl\"afli formula for $P$ considered as a hyperbolic polyhedron implies
$$\frac{\partial \mathrm{Vol}^P(\Delta)}{\partial \theta_{st}}=\frac{\partial \mathrm{Vol}(P)}{\partial (\pi-\theta^P_{st})}=\frac{d^P_{ij}}{2}=-\frac{l^P_{ij}}{2}.$$
\end{enumerate}
This completes the proof under the assumption that all the regular vertices of $\Delta$ are  in $\mathbb H^3_+\cup\mathbb L^3_+.$ 

For the general case, we observe that replacing a regular  vertex $\mathbf v_i$ by its negative $-\mathbf v_i$ changes $\Delta_0$ to $\Delta_i,$ hence switches the roles of $\Delta_0\cup\Delta_{12}\cup\Delta_{13}\cup\Delta_{14}$ and $\Delta_1\cup\Delta_2\cup\Delta_3\cup\Delta_4.$ Then from Definition \ref{v}, the volume changes sign; and by Proposition \ref{dac} (1) and Lemma \ref{la} (2), for each edge $e_{st},$ exactly one of  $\partial \theta_{st}$ and $l_{st}$ changes sign and the other do not change sign. Then the result follows from the previous special case and an induction on the number of vertices in  in $\mathbb H^3_-\cup\mathbb L^3_-.$  
\end{proof}


\section{Classification of admissible $6$-tuples}\label{caa}

The goal of this section is to understand the geometry of  $6$-tuples $(\alpha_{12},\dots,\alpha_{34})$ satisfying the following \emph{admissibility conditions}, which  come from  sequences of the quantum $6j$-symbols.  The main result of this section is the following Theorem \ref{admclassification}, which is a refinement of Theorem \ref{main1}, where the term \emph{generalized Euclidean tetrahedron} will be explained in Section \ref{get}, after which we will prove Theorem \ref{admclassification} in Section \ref{pfs}.  As an application of Theorem \ref{admclassification}, we prove Theorem \ref{asymp1} (1) at the end of Section \ref{pfs}.

\begin{definition}[Admissibility conditions]\label{adm} A triple of real numbers $(\alpha_1,\alpha_2,\alpha_3)$  is \emph{admissible} if
\begin{enumerate}[(1)]
\item $\alpha_i\in[0,2\pi]$ for  $i\in\{1,2,3\},$ 
\item $\alpha_i+\alpha_j-\alpha_k\geqslant 0$ for $\{i,j,k\}=\{1,2,3\},$  and
\item $\alpha_1+\alpha_2+\alpha_3\leqslant 4\pi.$
\end{enumerate}
A $6$-tuple $(\alpha_{12},\dots,\alpha_{34})$ of real numbers is \emph{admissible} if for each $i\in\{1,2,3,4\},$ the triple $(\alpha_{jk},\alpha_{jl},\alpha_{kl})$ is admissible, where  $\{j,k,l\}=\{1,2,3,4\}\setminus\{i\}.$
\end{definition}

We notice that the admissibility conditions are a system of linear inequalities. To have a clearer picture, we also need to consider the following strict admissibility conditions where all the inequalities in the admissibility conditions are required  to be strict.

\begin{definition}[Strict admissibility conditions] A triple of real numbers $(\alpha_1,\alpha_2,\alpha_3)$  is \emph{strictly admissible} if
\begin{enumerate}[(1)]
\item $\alpha_i\in(0,\pi)\cup(\pi,2\pi)$ for $i\in\{1,2,3\},$ 
\item $\alpha_i+\alpha_j-\alpha_k>0$ for $\{i,j,k\}=\{1,2,3\},$  and
\item $\alpha_1+\alpha_2+\alpha_3<4\pi.$
\end{enumerate}
A $6$-tuple $(\alpha_{12},\dots,\alpha_{34})$ of real numbers is \emph{strictly admissible} if for each $i\in\{1,2,3,4\},$ the triple $(\alpha_{jk},\alpha_{jl},\alpha_{kl})$ is strictly admissible, where  $\{j,k,l\}=\{1,2,3,4\}\setminus\{i\}.$
\end{definition}

\begin{theorem} \label{admclassification} Let  $(\alpha_{12},\dots,\alpha_{34})$ be a $6$-tuple real numbers, and for $\{i,j\}\subset\{1,\dots,4\},$ let 
$$\theta_{ij}=|\pi-\alpha_{ij}|.$$ 
\begin{enumerate}[(1)]
\item If  $(\alpha_{12},\dots,\alpha_{34})$ is strictly admissible, then there are  the following three possibilities:
\begin{enumerate}[(a)]
\item  $(\theta_{12},\dots,\theta_{34})$ is the set of dihedral angles of a spherical tetrahedron, in which case its Gram matrix has signature $(4,0),$ i.e., is positive definite.
\item   $(\theta_{12},\dots,\theta_{34})$  is the set of  dihedral angles of a generalized Euclidean tetrahedron, in which case its Gram matrix has signature $(3,0).$
\item    $(\theta_{12},\dots,\theta_{34})$  is the set of  dihedral angles of a generalized hyperbolic tetrahedron, in which case its Gram matrix  has signature $(3,1).$ 

\end{enumerate}
\item  If  $(\alpha_{12},\dots,\alpha_{34})$ is admissible, then there is an extra possibility:   $(\theta_{12},\dots,\theta_{34})$  is the set of  angles between four oriented straight lines in the Euclidean plane $\mathbb E^2,$ in which case its Gram matrix   has signature $(2,0)$ or $(1,0).$

\end{enumerate}

\end{theorem}

The proof of Theorem \ref{admclassification} makes an intensive use of  the Cauchy Interlace Theorem, which we recall below. The proof can be found in e.g. \cite[p.411]{GV}, \cite[p.185]{HJ} or \cite[p.186]{P}.

\begin{theorem}[Cauchy Interlace Theorem]\label{CIT} Let $A$ be an $n\times n$ Hermitian matrix and let $B$ be an $(n-1)\times (n-1)$ principal submatrix of $A.$ If 
$\lambda_1\leqslant \cdots \leqslant \lambda_n$
 list the eigenvalues of $A$ and 
 $\mu_1\leqslant \cdots \leqslant \mu_{n-1}$
  list the eigenvalues of $B,$ then
$$\lambda_1\leqslant \mu_1\leqslant \lambda_2\leqslant \mu_2\leqslant \cdots \leqslant \lambda_{n-1}\leqslant \mu_{n-1}\leqslant \lambda_n.$$
\end{theorem}


\subsection{Generalized Euclidean tetrahedra}\label{get}

\begin{definition}[Generalized Euclidean tetrahedron] A \emph{generalized Euclidean tetrahedron} is a quadruple of vectors $\mathbf u_1,\mathbf u_2,\mathbf u_3, \mathbf u_4$ in the unit sphere $\mathbb S^2$ of the Euclidean space $\mathbb E^3$ that span $\mathbb E^3.$
\end{definition}

The \emph{face} orthogonal to $\mathbf u_i$ is the plane $F_i$  in $\mathbb E^3$ tangent to $\mathbb S^2$ at $\mathbf u_i,$ and $\mathbf u_i$ is the \emph{outward normal vector} of $F_i.$ See Figure \ref{eu} (a), (b), (c) for the generic cases and (d), (e), (f) for some non-generic cases. In particular, in  (a), the vectors $\mathbf u_1,\dots,\mathbf u_4$ are not contained in any half-space of $\mathbb E^3,$ then the faces $F_1,\dots,F_4$ bound a Euclidean tetrahedron in the usual sense (the convex hull of four points in $\mathbb E^3$ in a general position)  with  $\mathbf u_1,\dots,\mathbf u_4.$

\begin{figure}[htbp]
\centering
\includegraphics[scale=0.2]{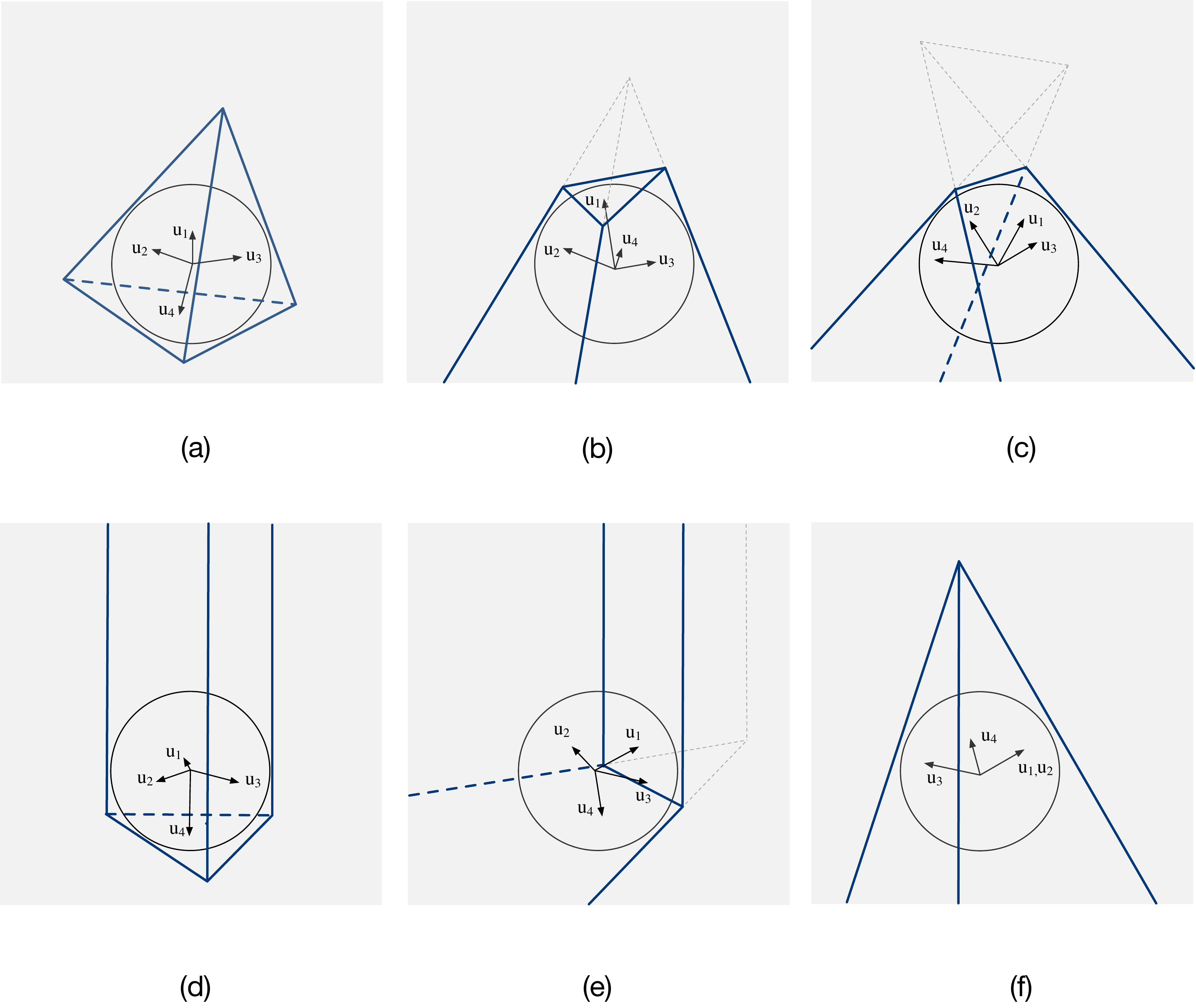}
\caption{In each of the figures, the disc represents the unit sphere $\mathbb S^2$ in $\mathbb E^3.$ In (a), all the vectors $\mathbf u_1,\dots,\mathbf u_4$ point outwards of the Euclidean tetrahedron, in (b), the vector $\mathbf u_1$ points inwards the ``tetrahedron" on the top and the vectors $\mathbf u_2,\mathbf u_3,\mathbf u_4$ point outwards, and in (c),  the vector $\mathbf u_1, \mathbf u_2$ point inwards the ``tetrahedron" on the top and the vectors $\mathbf u_3,\mathbf u_4$ point outwards. In (d) and (e), the vectors $\mathbf u_1,\mathbf u_2,\mathbf u_3$ lie in the same plan of $\mathbb E^3,$ and in (f), the vectors $\mathbf u_1$ and $\mathbf u_2$ coincide. In all these non-generic cases, $\mathbf u_1,\dots,\mathbf u_4$ still span $\mathbb E^3.$}
 \label{eu}
\end{figure}

The \emph{dihedral angle} $\theta_{ij}$ between the faces $F_i$ and $F_j$ is defined to be $\pi$ minus the angle between $\mathbf u_i$ and $\mathbf u_j,$ i.e.,
$$\theta_{ij}=\pi-\cos^{-1}\langle \mathbf u_i,\mathbf u_j\rangle,$$
where $\langle, \rangle$ here is the standard inner product on $\mathbb E^3.$

\begin{theorem}\label{ge}  Suppose  $(\theta_{12},\dots,\theta_{34})$ is a $6$-tuple of numbers in $[0,\pi].$ Then the following statements are equivalent.
\begin{enumerate}[(1)]
\item $(\theta_{12},\dots,\theta_{34})$ is the set of dihedral angles of a generalized Euclidean tetrahedron.
\item The Gram matrix $G$ of $(\theta_{12},\dots,\theta_{34})$ has signature $(3,0).$ 
\end{enumerate}
\end{theorem}

\begin{proof}The proof follows the idea of that of Luo\,\cite[Theorem]{L}.

Suppose (1) holds. Let $\mathbf u_1,\dots,\mathbf u_4$ be the vectors in $\mathbb S^2$ defining the generalized Euclidean tetrahedron. Since they span $\mathbb E^3,$ there are three of them, say,  $\mathbf u_1, \mathbf u_2, \mathbf u_3$ forming a basis of $\mathbb E^3.$ Let  $$U=[\mathbf u_1,\mathbf u_2,\mathbf u_3].$$ Then
$$G_1=U^T\cdot U,$$
where $G_1$ is the first $3\times 3$ principal submatrix $G.$ Since the tips of $\mathbf u_1, \mathbf u_2, \mathbf u_3$ are the vertices of a spherical triangle, $G_1$ is positive definite. As a consequence of the Cauchy Interlace Theorem, $G$ has at least three positive eigenvalues.  

Next we prove that $\det G=0.$  Indeed, since $\mathbf u_1, \mathbf u_2, \mathbf u_3$  form a basis of $\mathbb E^3,$ 
$$\mathbf u_4=a_1\mathbf u_1+a_2\mathbf u_2+a_3\mathbf u_3$$
for some real numbers $a_1,a_2,a_3.$ If we let $\mathbf g_i$ be the $i$-th column of $G,$ then
$$\mathbf g_4=a_1\mathbf g_1+a_2\mathbf g_2+a_3\mathbf g_3,$$
and as a consequence, $\det G=0.$ Now $G$ has at least three positive eigenvalues and  $\det G=0,$ hence the signature equals $(3,0),$ and (2) holds.

Suppose (2) holds, so that $G$ has signature $(3,0).$ Then by Sylvester's Law of Inertia,
$$G=W^T\cdot I_{3,0}\cdot W$$
for some $4\times 4$ matrix $W,$ where $I_{3,0}$ is  the matrix
$$I_{3,0}=\left[\begin{matrix}
1& 0& 0 & 0\\
0 & 1&0 &0\\
0& 0& 1&0 \\
0&0 &0 &  0\\
 \end{matrix}\right].$$
Let $\mathbf w_1,\dots,\mathbf w_4$ be the columns of $W,$ and for each $i\in\{1,\dots,4\},$ let $\mathbf u_i $ be the vector in $\mathbb E^3$ obtained from $\mathbf w_i$ by erasing  the last component. If $\langle, \rangle$ denotes the standard  inner product on $\mathbb E^{3},$ then we have
$$\langle \mathbf u_i,\mathbf u_i \rangle=\mathbf w_i^T\cdot I_{3,0}\cdot  \mathbf w_i=1$$
for $i\in\{1,2,3,4\},$ and
\begin{equation}\label{ed}
\langle \mathbf u_i,\mathbf u_j \rangle=  \mathbf w_i^T\cdot I_{3,0}\cdot  \mathbf w_j=-\cos\theta_{ij}
\end{equation}
for $\{i,j\}\subset\{1,2,3,4\}.$ 
In particular, 
$$\mathbf u_i\in\mathbb S^2$$
for each $i\in\{1,2,3,4\}.$

Next we will show that $\mathbf u_1,\dots, \mathbf u_4$ span $\mathbb E^3$; we proceed by contradiction. Suppose they do not span; then  $\mathbf u_1,\dots, \mathbf u_4$ lie in a subspace of $\mathbb E^3$ spanned by two vectors, say, $\mathbf u_1$ and $\mathbf u_2.$ Then 
$$\mathbf u_3=a_1\mathbf u_1+a_2\mathbf u_2\quad\text{and}\quad\mathbf u_4=b_1\mathbf u_1+b_2\mathbf u_2$$
 for some real numbers $a_1,a_2,b_1,b_2.$ As a consequence, the columns $\mathbf g_1,\dots,\mathbf g_4$ of $G$ satisfy 
$$\mathbf g_3=a_1\mathbf g_1+a_2\mathbf g_2\quad\text{and}\quad\mathbf g_4=b_1\mathbf g_1+b_2\mathbf g_2.$$
This implies that the rank of $G$ is at most $2,$ contradicting the assumption that $G$ has signature $(3,0).$ 
Therefore, $\mathbf u_1,\dots, \mathbf u_4$ are four vectors in $\mathbb S^2$ spanning $\mathbb E^3,$ hence define a generalized Euclidean tetrahedron $\Delta,$ and (\ref{ed}) shows that the dihedral angles of $\Delta$ are $\theta_{12},\dots,\theta_{34}.$
\end{proof}

\begin{theorem}\label{degenerate} Let $(\theta_{12},\dots,\theta_{34})$ be a $6$-tuple of numbers in $[0,\pi],$ and let $G$ be its Gram matrix. 
If the signature of $G$ is $(2,0)$ or $(1,0),$ then $(\theta_{12},\dots,\theta_{34})$  is the set of angles between four oriented straight lines in the Euclidean plane $\mathbb E^2.$ Moreover, the signature of $G$ is $(1,0),$ then the four straight lines are parallel.
\end{theorem}

\begin{remark}
Here the orientation of a straight line is defined by a specification of its normal vector, and the angle between two straight lines is $\pi$ minus the angle between the two normal vectors that define the orientation of the straight lines; See Figure \ref{4e}.
\end{remark}

\begin{figure}[htbp]
\centering
\includegraphics[scale=0.2]{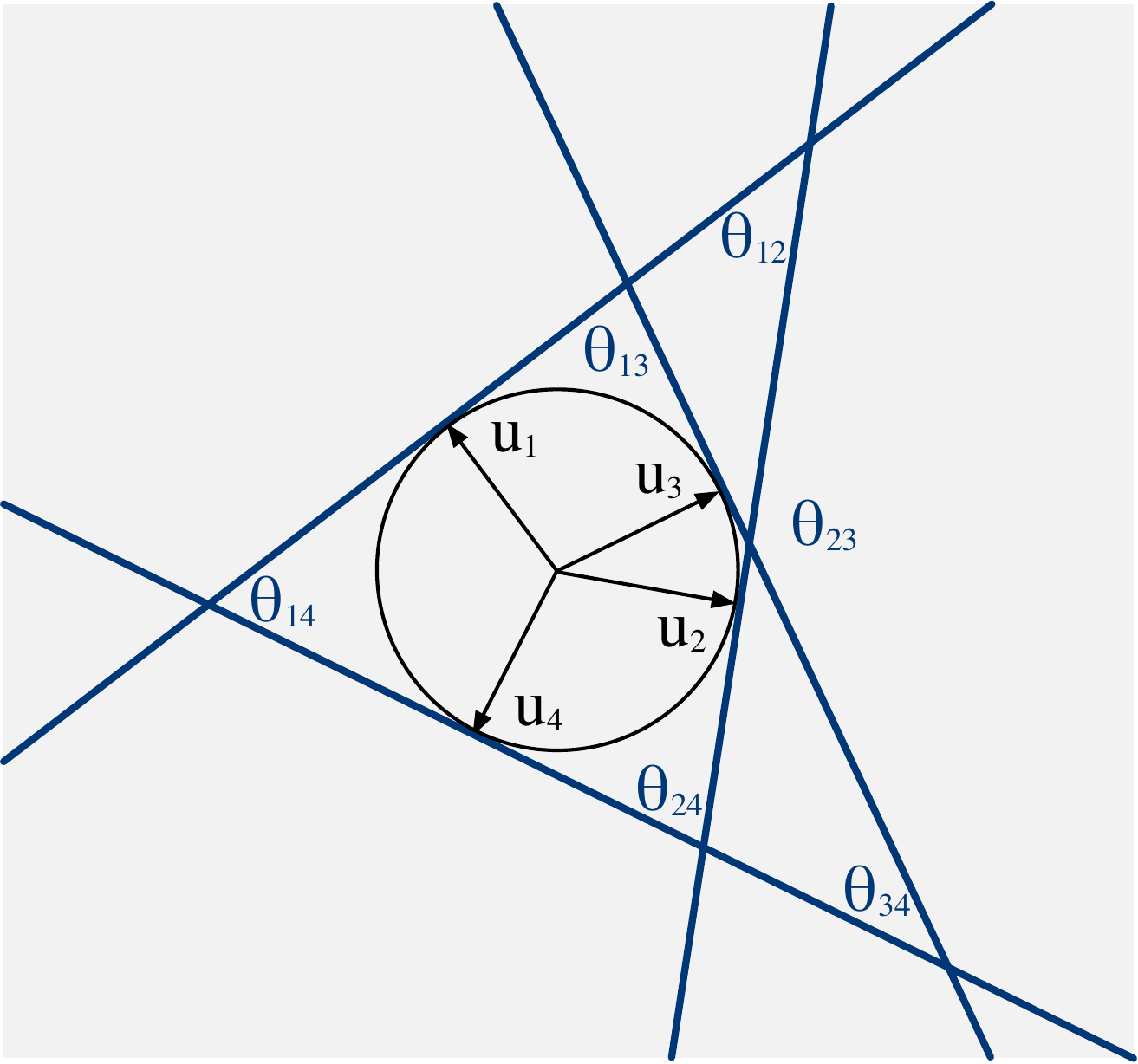}
\caption{In the figure, the circle represents the unit circle $\mathbb S^1$ in $\mathbb E^2.$ }
\label{4e}
\end{figure}

\begin{proof} Suppose $G$ has signature $(2,0),$ then by Sylvester's Law of Inertia,
$$G=W^T\cdot I_{2,0}\cdot W$$
for some $4\times 4$ matrix $W,$ where $I_{2,0}$ is  the matrix
$$I_{2,0}=\left[\begin{matrix}
1& 0& 0 & 0\\
0 & 1&0 &0\\
0& 0& 0&0 \\
0&0 &0 &  0\\
 \end{matrix}\right].$$
Let $\mathbf w_1,\dots,\mathbf w_4$ be the columns of $W,$ and for each $i\in\{1,\dots,4\},$ let $\mathbf u_i $ be the vector in $\mathbb E^2$ obtained from $\mathbf w_i$ by erasing  the last two components. If $\langle, \rangle$ denotes the standard  inner product on $\mathbb E^2,$ then we have
$$\langle \mathbf u_i,\mathbf u_j \rangle=  \mathbf w_i^T\cdot I_{3,0}\cdot  \mathbf w_j=-\cos\theta_{ij}$$
for $\{i,j\}\subset\{1,2,3,4\}.$ Let $L_i$ be the straight line in $\mathbb E^2$ tangent to the unit circle $\mathbb S^1$ at $\mathbf u_i,$ oriented as the direction of $\mathbf u_i.$  Then the angle between $L_i$ and $L_j$ is $\pi-\cos^{-1}\langle \mathbf u_i,\mathbf u_j\rangle=\theta_{ij},$ and $L_1\dots, L_4$ with this orientation are the desired oriented  straight  lines.

 Suppose $G$ has signature $(1,0),$ then by Sylvester's Law of Inertia,
$$G=W^T\cdot I_{2,0}\cdot W$$
for some $4\times 4$ matrix $W,$ where $I_{1,0}$ is  the matrix
$$I_{1,0}=\left[\begin{matrix}
1& 0& 0 & 0\\
0 & 0&0 &0\\
0& 0& 0&0 \\
0&0 &0 &  0\\
 \end{matrix}\right].$$
Let $\mathbf w_1,\dots,\mathbf w_4$ be the columns of $W,$ and for each $i\in\{1,\dots,4\},$ let $a_i$ be the first entry of $\mathbf w_i$ and let 
$$\mathbf u_i =\left[\begin{matrix}
a_i\\
0\\
 \end{matrix}\right].$$  
If $\langle, \rangle$ denotes the standard  inner product on $\mathbb E^2,$ then for each $i\in\{1,2,3,4\},$ we have
$$a_i^2=\langle \mathbf u_i,\mathbf u_i \rangle=  \mathbf w_i^T\cdot I_{3,0}\cdot  \mathbf w_i=1,$$
and for each $\{i,j\}\subset\{1,2,3,4\},$ we have
$$a_ia_j=\langle \mathbf u_i,\mathbf u_j \rangle=  \mathbf w_i^T\cdot I_{3,0}\cdot  \mathbf w_j=-\cos\theta_{ij}.$$ 
As a consequence, $a_i=\pm 1$ and $\mathbf u_i=\pm\left[\begin{matrix}
1\\
0\\
 \end{matrix}\right]$ for each $i.$ Then the desired parallel oriented straight lines are the vertical lines passing through the points $\pm(1,0).$ \end{proof}


\subsection{A proof of Theorem \ref{admclassification} and Theorem \ref{asymp1} (1)}\label{pfs}

The key ingredient in the proof of Theorem \ref{admclassification} is the following Lemma.

\begin{lemma}\label{hyper}  
Let $(\alpha_1, \alpha_2,\alpha_3)$ be triple of real numbers. For $i\in\{1,2,3\},$ let
$$\theta_i=|\pi-\alpha_i|,$$ and let 
 $$G=\left[\begin{matrix}
1& -\cos\theta_{1} & -\cos\theta_{2} \\
-\cos\theta_{1} & 1& -\cos\theta_{3}  \\
-\cos\theta_{2} & -\cos\theta_{3} & 1 \\
 \end{matrix}\right].$$ 
 \begin{enumerate}[(1)]
 \item If $(\alpha_1, \alpha_2,\alpha_3)$ is strictly admissible and $\det G\leqslant 0,$ then 
  $$\theta_1+\theta_2+\theta_3\leqslant \pi.$$ 
  \item If $(\alpha_1, \alpha_2,\alpha_3)$ is  admissible and $\det G<0,$ then 
  $$\theta_1+\theta_2+\theta_3<\pi.$$ 
 \end{enumerate}
\end{lemma}

\begin{proof} For (1), we first observe that  the principal minors are strictly positive due to the strict admissibility condition (1), thus $\det G$ is positive if and only if $G$ is positive definite. The latter is equivalent is equivalent to $(\theta_1,\theta_2,\theta_3)$ being the set of dihedral angles of a spherical triangle, which holds if and only if $(\theta_1,\theta_2,\theta_3)$  satisfies the system of the following four equations
\begin{equation}\label{system}
\left\{
\begin{array}{c}
\theta_1+\theta_2+\theta_3> \pi \\ 
(\pi-\theta_1)+(\pi-\theta_2)+\theta_3>\pi\\
(\pi-\theta_1)+\theta_2+(\pi-\theta_3)>\pi\\
\theta_1+(\pi-\theta_2)+(\pi-\theta_3)> \pi\\
\end{array}.\right.
\end{equation}
Therefore, if $\det G\leqslant 0,$ then one of the following four inequalities  is satisfied.
\begin{equation}\label{8}
\theta_1+\theta_2+\theta_3\leqslant \pi,
\end{equation}
\begin{equation}\label{9}
(\pi-\theta_1)+(\pi-\theta_2)+\theta_3\leqslant \pi,
\end{equation}
\begin{equation}\label{10}
(\pi-\theta_1)+\theta_2+(\pi-\theta_3)\leqslant \pi,
\end{equation}
\begin{equation}\label{11}
\theta_1+(\pi-\theta_2)+(\pi-\theta_3)\leqslant \pi.
\end{equation}
In the rest of the proof, we will show that (\ref{9}) is impossible under the admissibility conditions. Then by similar reasons, we can rule out (\ref{10}) and (\ref{11}), and leaves (\ref{8}) as the only possibility.  We consider the following cases. 
\begin{enumerate}[(a)]
\item $\alpha_1< \pi,$ $\alpha_2<\pi$ and $\alpha_3<\pi.$ Then $\theta_1=\pi-\alpha_1,$ $\theta_2=\pi-\alpha_2$ and $\theta_3=\pi-\alpha_3,$ and (\ref{9}) is equivalent to
$$\alpha_1+\alpha_2-\alpha_3\leqslant 0,$$
which contradicts the strict admissibility condition $\alpha_1+\alpha_2-\alpha_3>0.$

\item $\alpha_1>\pi,$ $\alpha_2<\pi$ and $\alpha_3<\pi.$ (The case  $\alpha_2>\pi,$ $\alpha_1<\pi$ and $\alpha_3<\pi$ is similar by symmetry.) Then $\theta_1=\alpha_1-\pi,$ $\theta_2=\pi-\alpha_2$ and $\theta_3=\pi-\alpha_3,$ and (\ref{9}) is equivalent to
$$\alpha_1-\alpha_2+\alpha_3\geqslant 2\pi.$$
Then we have
$$\alpha_1-\alpha_2-\alpha_3=(\alpha_1-\alpha_2+\alpha_3)-2\alpha_3\geqslant 2\pi-2\pi=0,$$
which contradicts the strict admissibility condition $\alpha_2+\alpha_3-\alpha_1>0.$

\item  $\alpha_1<\pi,$ $\alpha_2<\pi$ and $\alpha_3>\pi.$ Then $\theta_1=\pi-\alpha_1,$ $\theta_2=\pi-\alpha_2$ and $\theta_3=\alpha_3-\pi,$ and (\ref{9}) is equivalent to
$$\alpha_1+\alpha_2+\alpha_3\leqslant 2\pi.$$
Then we have 
$$\alpha_1+\alpha_2-\alpha_3=(\alpha_1+\alpha_2+\alpha_3)-2\alpha_3\leqslant 2\pi-2\pi=0,$$
which contradicts the strict admissibility condition $\alpha_1+\alpha_2-\alpha_3>0.$

\item $\alpha_1>\pi,$ $\alpha_2>\pi$ and $\alpha_3<\pi.$ Then $\theta_1=\alpha_1-\pi,$ $\theta_2=\alpha_2-\pi$ and $\theta_3=\pi-\alpha_3,$ and (\ref{9}) is equivalent to
$$\alpha_1+\alpha_2+\alpha_3\geqslant 4\pi,$$
which contradicts the strict admissibility condition $\alpha_1+\alpha_2+\alpha_3<4\pi.$

\item  $\alpha_1>\pi,$ $\alpha_2<\pi$ and $\alpha_3>\pi.$  (The case  $\alpha_2>\pi,$ $\alpha_1<\pi$ and $\alpha_3>\pi$ is similar by symmetry.) Then $\theta_1=\alpha_1-\pi,$ $\theta_2=\pi-\alpha_2$ and $\theta_3=\alpha_3-\pi,$ and (\ref{9}) is equivalent to
$$\alpha_2+\alpha_3-\alpha_1\leqslant 0,$$
which contradicts the strict admissibility condition $\alpha_2+\alpha_3-\alpha_1>0.$

\item $\alpha_1>\pi,$ $\alpha_2>\pi$ and $\alpha_3>\pi.$ Then $\theta_1=\alpha_1-\pi,$ $\theta_2=\alpha_2-\pi$ and $\theta_3=\alpha_3-\pi,$ and (\ref{9}) is equivalent to
$$\alpha_1+\alpha_2-\alpha_3\geqslant 2\pi.$$
Then we have 
$$\alpha_1+\alpha_2+\alpha_3=(\alpha_1+\alpha_2-\alpha_3)+2\alpha_3\geqslant 2\pi+2\pi=4\pi,$$
which contradicts the strict admissibility condition $\alpha_1+\alpha_2+\alpha_3<4\pi.$
\end{enumerate}

For (2), by the discussion in the beginning of the proof of Part (1), we see that if $\theta_1,$ $\theta_2,$ $\theta_3$ are numbers in $(0,\pi),$ then $\det G>0$ if and only if the system of equations (\ref{system}) hold. Therefore, for $\theta_1,$ $\theta_2,$ $\theta_3$  in $[0,\pi],$ if the following system of equalities 
\begin{equation}\label{system}
\left\{
\begin{array}{c}
\theta_1+\theta_2+\theta_3\geqslant  \pi \\ 
(\pi-\theta_1)+(\pi-\theta_2)+\theta_3\geqslant \pi\\
(\pi-\theta_1)+\theta_2+(\pi-\theta_3)\geqslant \pi\\
\theta_1+(\pi-\theta_2)+(\pi-\theta_3)\geqslant  \pi\\
\end{array}\right.
\end{equation}
hold, then $\det G\geqslant 0.$ 
As a consequence, if $\det G<0,$ then one of the following four mutually exclusive  strict  inequalities is satisfied.
\begin{equation}\label{18}
\theta_1+\theta_2+\theta_3< \pi,
\end{equation}
\begin{equation}\label{19}
(\pi-\theta_1)+(\pi-\theta_2)+\theta_3<\pi,
\end{equation}
\begin{equation}\label{20}
(\pi-\theta_1)+\theta_2+(\pi-\theta_3)<\pi,
\end{equation}
\begin{equation}\label{21}
\theta_1+(\pi-\theta_2)+(\pi-\theta_3)< \pi.
\end{equation}
Similar to the proof of Part (1), we will show that (\ref{19}) is impossible under the admissibility conditions. Then by similar reasons, we can also rule out (\ref{20}) and (\ref{21}), and leaves (\ref{18}) as the only possibility. The argument is very similar to that of Part (1), and we include the details for the readers' convenience. We consider the following cases. 
\begin{enumerate}[(a)]
\item $\alpha_1\leqslant  \pi,$ $\alpha_2\leqslant \pi$ and $\alpha_3\leqslant \pi.$ Then $\theta_1=\pi-\alpha_1,$ $\theta_2=\pi-\alpha_2$ and $\theta_3=\pi-\alpha_3,$ and (\ref{19}) is equivalent to
$$\alpha_1+\alpha_2-\alpha_3<0,$$
which contradicts the admissibility condition $\alpha_1+\alpha_2-\alpha_3\geqslant 0.$

\item $\alpha_1>\pi,$ $\alpha_2\leqslant \pi$ and $\alpha_3\leqslant \pi.$ (The case  $\alpha_1\leqslant \pi,$ $\alpha_2>\pi$ and $\alpha_3\leqslant \pi$ is similar by symmetry.) Then $\theta_1=\alpha_1-\pi,$ $\theta_2=\pi-\alpha_2$ and $\theta_3=\pi-\alpha_3,$ and (\ref{19}) is equivalent to
$$\alpha_1-\alpha_2+\alpha_3>2\pi.$$
Then we have
$$\alpha_1-\alpha_2-\alpha_3=(\alpha_1-\alpha_2+\alpha_3)-2\alpha_3>2\pi-2\pi=0,$$
which contradicts the  admissibility condition $\alpha_2+\alpha_3-\alpha_1\geqslant 0.$

\item  $\alpha_1\leqslant \pi,$ $\alpha_2\leqslant \pi$ and $\alpha_3>\pi.$ Then $\theta_1=\pi-\alpha_1,$ $\theta_2=\pi-\alpha_2$ and $\theta_3=\alpha_3-\pi,$ and (\ref{19}) is equivalent to
$$\alpha_1+\alpha_2+\alpha_3<2\pi.$$
Then we have 
$$\alpha_1+\alpha_2-\alpha_3=(\alpha_1+\alpha_2+\alpha_3)-2\alpha_3<2\pi-2\pi=0,$$
which contradicts the admissibility condition $\alpha_1+\alpha_2-\alpha_3\geqslant 0.$

\item $\alpha_1>\pi,$ $\alpha_2>\pi$ and $\alpha_3\leqslant \pi.$ Then $\theta_1=\alpha_1-\pi,$ $\theta_2=\alpha_2-\pi$ and $\theta_3=\pi-\alpha_3,$ and (\ref{19}) is equivalent to
$$\alpha_1+\alpha_2+\alpha_3>4\pi,$$
which contradicts the  admissibility condition $\alpha_1+\alpha_2+\alpha_3\leqslant 4\pi.$

\item  $\alpha_1>\pi,$ $\alpha_2\leqslant \pi$ and $\alpha_3>\pi.$  (The case  $\alpha_1\leqslant \pi,$ $\alpha_2>\pi$ and $\alpha_3>\pi$ is similar by symmetry.) Then $\theta_1=\alpha_1-\pi,$ $\theta_2=\pi-\alpha_2$ and $\theta_3=\alpha_3-\pi,$ and (\ref{19}) is equivalent to
$$\alpha_2+\alpha_3-\alpha_1< 0,$$
which contradicts the   admissibility condition $\alpha_2+\alpha_3-\alpha_1\geqslant 0.$

\item $\alpha_1>\pi,$ $\alpha_2>\pi$ and $\alpha_3>\pi.$ Then $\theta_1=\alpha_1-\pi,$ $\theta_2=\alpha_2-\pi$ and $\theta_3=\alpha_3-\pi,$ and (\ref{19}) is equivalent to
$$\alpha_1+\alpha_2-\alpha_3>2\pi.$$
Then we have 
$$\alpha_1+\alpha_2+\alpha_3=(\alpha_1+\alpha_2-\alpha_3)+2\alpha_3>2\pi+2\pi=4\pi,$$
which contradicts the admissibility condition $\alpha_1+\alpha_2+\alpha_3\leqslant 4\pi.$
\end{enumerate}
\end{proof}

\begin{proof}[Proof of Theorem \ref{admclassification}] For (1), due to the strict admissibility condition (1), $\cos\theta_{ij}\neq \pm 1$ and hence  all the $2\times 2$ principal submatrices of $G$ are positive definite. Then by the Cauchy Interlace Theorem, $G$ has at least two positive eigenvalues.  We consider the following cases: 
\begin{enumerate}[(a)]
\item If for at least one $i\in\{1,2,3,4\},$ $G_{ii}>0,$ then due to the strict admissibility condition that $\cos\theta_{jk}\neq \pm 1,$ the $i$-th $3\times 3$ principal submatrix is positive definite, and hence all its eigenvalues are positive. As a consequence of the Cauchy Interlace Theorem, $G$ has at least three positive eigenvalues hence the signature is  $(4,0),$ $(3,0)$ or $(3,1).$ 

\item If for all $i\in\{1,2,3,4\},$ $G_{ii}<0,$ then by Lemma \ref{hyper} (2) and Bonahon-Bao\,\cite{BB},  $(\theta_{12},\dots,\theta_{34})$ is the set of dihedral angles of a hyperideal tetrahedron, and $G$ has signature $(3,1).$ 

\item If some $G_{ii}=0,$ then $(\theta_{12},\dots,\theta_{34})$ lies in the closure of the region given by cases of (a) and (b), hence besides $(4,0),$ $(3,0)$ or $(3,1),$ the signature of $G$ can have extra possibilities $(2,0)$ and $(2,1).$ 

 \end{enumerate}

Next we rule out the extra possibilities as follows. By Theorem \ref{degenerate} and Theorem \ref{hyperdegenerate}, $(\theta_{12},\dots,\theta_{34})$ is the set of angles  between four straight lines in $\mathbb E^2$ or four intersecting geodesics in $\overline{\mathbb H^2}.$
This implies that $G_{ii}\leqslant 0$ for all $i\in\{1,2,3,4\},$ and from Figure \ref{4e} and Figure \ref{4h} we see that there is always a triple of angles (e.g. $(\theta_{12},\theta_{13},\theta_{23})$ ) satisfying one of the inequalities (\ref{9}), (\ref{10}) and (\ref{11}), which by Lemma \ref{hyper} (1) is impossible.

For (2), $(\theta_{12},\dots,\theta_{34})$ lies in the closure of the region given by cases of (a) and (b), hence besides $(4,0),$ $(3,0)$ or $(3,1),$ the signature of $G$ can have extra possibilities $(2,0),$ $(1,0),$ $(2,1)$ and $(1,1).$

 By Theorem \ref{degenerate}, if the signature of $G$ is  $(2,0)$ or $(1,0),$ then $(\theta_{12},\dots,\theta_{34})$ is the set of angles  between four straight lines in $\mathbb E^2,$ which is what Part (2) claims.

 Next we rule out the other two possibilities that $G$ has signature $(2,1)$ and $(1,1).$ 
 
By Theorem \ref{hyperdegenerate} and Lemma \ref{hyper} (2), if $G$ has signature $(2,1),$ then  $(\theta_{12},\dots,\theta_{34})$ is the set of angles  between four  intersecting geodesics in $\overline{\mathbb H^2}.$ This implies that $G_{ii}\leqslant 0$ for all $i\in\{1,2,3,4\},$ and from Figure \ref{4h} we see there is always a triple of angles satisfying one of the inequalities (\ref{19}), (\ref{20}) and (\ref{21}), which by Lemma \ref{hyper} (2) is impossible.

If $G$ has signature $(1,1),$ then $\theta_{ij}=0$ or $\pi$ for all $\{i,j\}\subset \{1,2,3,4\},$ because otherwise some $2\times 2$ submatrix of $G$ would be positive definite and by the Cauchy Interlace Theorem $G$ would have at least two positive eigenvalues, whose signature cannot be $(1,1).$ Then a case by case computation shows that the  signature of $G$ can only be  $(3,1),$ $(2,1)$ or $(1,0).$ Indeed, by Lemma \ref{change}, by a sequence of change of angles operations we only need to compute the following three cases:
\begin{enumerate}[(1)]
\item $(0,0,0,0,0,0)$ whose Gram matrix has signature $(3,1).$ This is the set of dihedral angles of the hyperbolic tetrahedron whose geometric piece is the regular ideal octahedron.
\item $(\pi,0,0,0,0,0)$ whose Gram matrix has signature $(2,1).$ This is the set of angles between four geodesics $L_1,$ $ L_2$ $L_3$  and $L_4$ in $\overline{\mathbb H^2}$ such that $L_1$ and $L_2$ coincide and $L_2,$ $L_3$ and $L_4$ are the edges of an ideal hyperbolic triangle. 
\item $(\pi, 0,0,0,0,\pi)$ whose Gram matrix has signature $(1,0).$ This is the set of angles between four parallel straight lines $L_1,$ $ L_2$ $L_3$  and $L_4$ in $\mathbb E^2,$ where $L_1$ and $L_2$ have the same  orientation which is opposite to that of $L_3$ and $L_4.$
\end{enumerate}
Therefore, $G$ cannot have signature $(1,1).$
 \end{proof}

\begin{proof}[Proof of Theorem \ref{asymp1} (1)] By Theorems \ref{admclassification}, \ref{characterization}, \ref{ge} and \ref{degenerate},  the signature of $G$ can only  be  $(4,0),$ $(3,0),$ $(2,0),$ $(1,0)$ or $(3,1).$ By Cauchy Interlace Theorem, if $G_{ii}<0$ for some $i\in\{1,2,3,4\},$ then there must be at least one negative eigenvalue. As a consequence, the only possibility is $(3,1),$ and by Theorem \ref{characterization}, $(\theta_1,\dots,\theta_6)$ is the set of dihedral angles of a generalized hyperbolic tetrahedron. 
\end{proof}

\begin{proposition}\label{connected} The space of $(\alpha_1,\dots,\alpha_6)$ satisfying the conditions of Theorem \ref{asymp1} is connected. 
\end{proposition}

\begin{proof} By Theorem \ref{asymp1}, $(\theta_1,\dots,\theta_6)$ is the set of dihedral angles of a generalized hyperbolic tetrahedron $\Delta$ with at least one hyperideal vertex. Then by
Lemma \ref{hyper} (2), the only possibilities are those in Proposition \ref{class} (2a), (2b), (2c), (2d), (3a), (3b), (3c), (3d), (4a), (4b) and (5a). We observe that in any of these cases, the negative edges always connect a regular or ideal vertex with a hyperideal vertex. Then along each negative edge, by pushing the regular or ideal vertex away from the hyperideal one, one deforms $\Delta$ into cases Proposition \ref{class}  (2a), (3a) (4a) and (5a). By further pushing all the regular and ideal vertices out of $\overline{\mathbb H^3},$ we deform $\Delta$ into case (5a), a hyperbolic tetrahedron with four hyperideal vertices. By Bonahon-Bao\,\cite{BB}, the space of such hyperbolic tetrahedra is connected. This completes the proof.
\end{proof}


\section{A volume formula}

The main results of this section are Theorem \ref{volume}, where we obtain a formula for the volume of a generalized hyperbolic tetrahedron in terms of the dihedral angles. The formula is in the same spirit of Murakami-Yano\,\cite{MY} and Ushijima\,\cite{U}, and essentially coincides with \cite[Theorem 2]{MY} in the case of hyperbolic tetrahedra with only regular vertices. In Section \ref{growth}, by studying the asymptotics of quantum $6j$-symbols, we obtain a simpler presentation of the volume formula when at least one vertex of the generalized hyperbolic tetrahedron is hyperideal. See Theorem \ref{volume2}. It worth mentioning that a different volume formula was also obtained in Sohn\,\cite{S} and Bonahon-Sohn\,\cite{BS} for generalized hyperbolic tetrahedra, which also works for the cases with deep truncations. 

Our formula is closely related to the critical values of a function defined using the dilogarithm function. Let $\log:\mathbb C\setminus (-\infty, 0]\to\mathbb C$ be the standard logarithm function defined by
$$\log z=\log|z|+\sqrt{-1}\cdot\arg z$$
with $-\pi<\arg z<\pi.$ The dilogarithm function $\mathrm{Li}_2: \mathbb C\setminus (1,\infty)\to\mathbb C$ is defined by
$$\mathrm{Li}_2(z)=-\int_0^z\frac{\log (1-u)}{u}du$$
where the integral is along any path in $\mathbb C\setminus (1,\infty)$ connecting $0$ and $z,$ which is holomorphic in $\mathbb C\setminus [1,\infty)$ and continuous in $\mathbb C\setminus (1,\infty).$ The dilogarithm function satisfies the following property (see eg. Zagier\,\cite{Z}).
On the unit circle $\big\{ z=e^{2\sqrt{-1}\theta}\,\big|\,0 \leqslant \theta\leqslant\pi\big\},$ 
\begin{equation}\label{dilogLob}
\mathrm{Li}_2(e^{2\sqrt{-1}\theta})=\frac{\pi^2}{6}+\theta(\theta-\pi)+2\sqrt{-1}\cdot\Lambda(\theta).
\end{equation}
Here $\Lambda:\mathbb R\to\mathbb R$  is the Lobachevsky function defined by
$$\Lambda(\theta)=-\int_0^\theta\log|2\sin t|dt,$$
which is an odd function of period $\pi$ (see eg. Thurston's notes\,\cite[Chapter 7]{T}).

Suppose $(\theta_1,\dots,\theta_6)$ is the set of dihedral angles of a generalized truncated hyperbolic tetrahedron $\Delta,$ $\alpha_i=\pi\pm\theta_i$ for $i\in\{1,\dots,6\}$ and $\boldsymbol\alpha=(\alpha_1,\dots,\alpha_6).$ Let 
\begin{equation}\label{term}
U(\boldsymbol\alpha,\xi)=-\frac{1}{2}\sum_{i=1}^4\sum_{j=1}^3\mathrm L(\eta_j-\tau_i)+\frac{1}{2}\sum_{i=1}^4\mathrm L(\tau_i-\pi)-\mathrm L(\xi-\pi) + \sum_{i=1}^4\mathrm L(\xi-\tau_i) + \sum_{j=1}^3\mathrm L(\eta_j-\xi),
\end{equation}
where $\mathrm L:\mathbb R\to\mathbb C$ is the function  defined by
$$\mathrm L(x)=\mathit{Li}\big(e^{2\sqrt{-1}x}\big)+x^2,$$
and
$$\tau_1=\frac{\alpha_1+\alpha_2+\alpha_3}{2}, \quad \tau_2=\frac{\alpha_1+\alpha_5+\alpha_6}{2},\quad \tau_3=\frac{\alpha_2+\alpha_4+\alpha_6}{2},\quad \tau_4=\frac{\alpha_3+\alpha_4+\alpha_5}{2},$$
$$\eta_1=\frac{\alpha_1+\alpha_2+\alpha_4+\alpha_5}{2},\quad\eta_2=\frac{\alpha_1+\alpha_3+\alpha_4+\alpha_6}{2},\quad \eta_3=\frac{\alpha_2+\alpha_3+\alpha_5+\alpha_6}{2}.$$
Then a direct computation shows that
$$\frac{\partial U(\boldsymbol\alpha,\xi)}{\partial \xi}=2\sqrt{-1}\cdot\log\frac{(1-z)(1-zu_1u_2u_4u_5)(1-zu_1u_3u_4u_6)(1-zu_2u_3u_5u_6)}{(1-zu_1u_2u_3)(1-zu_1u_5u_6)(1-zu_2u_4u_6)(1-zu_3u_4u_5)}\quad\quad(\mathrm{mod}\ 4\pi)$$
where $u_i=e^{\sqrt{-1}\alpha_i}$ for $i\in\{1,\dots,6\}$ and $z=e^{-2\sqrt{-1}\xi};$ 
and the equation
$$\frac{\partial U(\boldsymbol\alpha,\xi)}{\partial \xi}=0\quad\quad(\mathrm{mod}\ 4\pi)$$
is equivalent to
$$\frac{(1-z)(1-zu_1u_2u_4u_5)(1-zu_1u_3u_4u_6)(1-zu_2u_3u_5u_6)}{(1-zu_1u_2u_3)(1-zu_1u_5u_6)(1-zu_2u_4u_6)(1-zu_3u_4u_5)}=1,$$
which simplifies to the following quadratic equation
\begin{equation}\label{quadratic}
Az^2+Bz+C=0,
\end{equation}
where
 \begin{equation*}
\begin{split}
A=&u_1u_4+u_2u_5+u_3u_6-u_1u_2u_6-u_1u_3u_5-u_2u_3u_4-u_4u_5u_6+u_1u_2u_3u_4u_5u_6,\\
B=&-\Big(u_1-\frac{1}{u_1}\Big)\Big(u_4-\frac{1}{u_4}\Big)-\Big(u_2-\frac{1}{u_2}\Big)\Big(u_5-\frac{1}{u_5}\Big)-\Big(u_3-\frac{1}{u_3}\Big)\Big(u_6-\frac{1}{u_6}\Big),\\
C=&\frac{1}{u_1u_4}+\frac{1}{u_2u_5}+\frac{1}{u_3u_6}-\frac{1}{u_1u_2u_6}-\frac{1}{u_1u_3u_5}-\frac{1}{u_2u_3u_4}-\frac{1}{u_4u_5u_6}+\frac{1}{u_1u_2u_3u_4u_5u_6}.
\end{split}
\end{equation*}
Let 
\begin{equation}\label{z}
z=\frac{-B+\sqrt{B^2-4AC}}{2A}\quad\text{and}\quad z^*=\frac{-B-\sqrt{B^2-4AC}}{2A}
\end{equation}
be the two solutions of (\ref{quadratic}).  Here as a convention, we let $\sqrt{x}=\sqrt{-1}\sqrt{|x|}$ if $x$ is a negative real number.
Then by a direct computation (see also \cite{MY, U}), we have
\begin{equation}\label{detG}
 B^2-4AC=16\det G,
 \end{equation}
 where $G$ is the Gram matrix of $\Delta.$  By Theorem \ref{characterization}, $G$ has signature $(3,1)$ and hence $\det G<0.$  As a consequence, $B^2-4AC<0.$ Due to the fact that  $B$ is real and $A$ and $C$ are complex conjugate,  we have
$$|z|=|z^*|=1;$$
and as a consequence the equations 
$$e^{-2\sqrt{-1}\xi}=z\quad\text{and}\quad e^{-2\sqrt{-1}\xi}=z^*$$ 
respectively have a unique real solution 
$$\xi=\xi(\boldsymbol\alpha)\quad\text{and}\quad\xi=\xi^*(\boldsymbol\alpha)$$
 in the interval $[\pi, 2\pi).$ Then we have
\begin{equation}\label{4k}
\frac{\partial U}{\partial \xi}\Big|_{\xi=\xi(\boldsymbol\alpha)}=4k\pi\quad\text{and}\quad\frac{\partial U}{\partial \xi}\Big|_{\xi=\xi^*(\boldsymbol\alpha)}=4k^*\pi
\end{equation}
for some integers $k$ and $k^*.$

For  a fixed $\boldsymbol\alpha=(\alpha_1,\dots,\alpha_6),$ let
$$V(\xi)=\frac{1}{2}\mathrm{Im}U(\boldsymbol\alpha,\xi).$$ 
Then by the relationship between the dilogarithm function and the Lobachevsky function, 
\begin{equation}\label{f}
\begin{split}
V(\xi)=&\delta(\alpha_1,\alpha_2,\alpha_3)+\delta(\alpha_1,\alpha_5,\alpha_6)+\delta(\alpha_2,\alpha_4,\alpha_6)+\delta(\alpha_3,\alpha_4,\alpha_5)\\
&+\Lambda(2\pi-\xi)+\sum_{i=1}^4\Lambda(\xi-\tau_i)+\sum_{j=1}^3\Lambda(\eta_j-\xi),
\end{split}
\end{equation}
where
$$\delta(\alpha,\beta,\gamma)=-\frac{1}{2}\Lambda\Big(\frac{-\alpha+\beta+\gamma}{2}\Big)-\frac{1}{2}\Lambda\Big(\frac{\alpha-\beta+\gamma}{2}\Big)-\frac{1}{2}\Lambda\Big(\frac{\alpha+\beta-\gamma}{2}\Big)+\frac{1}{2}\Lambda\Big(\frac{\alpha+\beta+\gamma}{2}\Big).$$

\begin{theorem}\label{volume} Suppose  $\Delta$ is a generalized hyperbolic tetrahedron with dihedral angles $(\theta_1,\dots,\theta_6)$  and $\boldsymbol\alpha=(\pi\pm\theta_1,\pi\pm\theta_2,\pi\pm\theta_3,\pi\pm\theta_4,\pi\pm\theta_5,\pi\pm\theta_6).$  Then
$$\mathrm{Vol}(\Delta)=V(\xi(\boldsymbol\alpha)).$$
\end{theorem}

\begin{proof} The proof follows the same idea of Ushijima\,\cite[Theorem 1.1]{U}, which is to verify the Schl\"afli formula (\ref{Sch}). To this end, let 
$$W(\boldsymbol\alpha)=U(\boldsymbol\alpha,\xi(\boldsymbol\alpha))\quad\text{and}\quad W^*(\boldsymbol\alpha)=U(\boldsymbol\alpha,\xi^*(\boldsymbol\alpha)),$$
and let 
$$F(\boldsymbol\alpha)=\frac{1}{2}(W(\boldsymbol\alpha)-W^*(\boldsymbol\alpha))$$
and
$$F^*(\boldsymbol\alpha)=\frac{1}{2}(W(\boldsymbol\alpha)+W^*(\boldsymbol\alpha)).$$

First, we will prove that when $G_{ii}\neq 0$ for all $i\in\{1,2,3,4\},$ 
\begin{equation}\label{F}
\frac{\partial \mathrm{Im}F}{\partial \theta_k}=-l_k
\end{equation}
and 
\begin{equation}\label{G}
\frac{\partial\mathrm{Im} F^*}{\partial \theta_k}=0
\end{equation}
for each $k\in\{1,\dots,6\}.$
If these identities hold, then
$$\frac{\partial \mathrm{Im}W}{\partial \theta_k}=\frac{\partial \mathrm{Im}F}{\partial \theta_k}+\frac{\partial \mathrm{Im}F^*}{\partial \theta_k}=-l_k;$$
and since
$$V(\xi(\boldsymbol\alpha))=\frac{1}{2}W(\boldsymbol\alpha),$$
we have 
$$\frac{\partial V(\xi(\boldsymbol\alpha))}{\partial \theta_k}=-\frac{l_k}{2},$$
which satisfies the Schl\"afli formula in Proposition \ref{Schlafli}.  Hence
\begin{equation}\label{sch}
V(\xi(\boldsymbol\alpha))=\mathrm{Vol}(\Delta)+C
\end{equation}
for some constant $C.$ Then by continuity, this identity extends to $\boldsymbol\alpha$ with some $G_{ii}=0,$ and extends to the closure of the space of all $\boldsymbol\alpha$s coming from the dihedral angles of generalized hyperbolic tetrahedra. At the end of the proof, we will show that $C=0$ by doing a direct computation for a certain choice of degenerate $\boldsymbol\alpha$ that liesagaion the boundary of the space.

Now to prove (\ref{F}) and (\ref{G}), we have
$$\frac{\partial W}{\partial \alpha_k}=\frac{\partial U}{\partial \alpha_k}\Big|_{\xi=\xi(\boldsymbol\alpha)}+\frac{\partial U}{\partial \xi}\Big|_{\xi=\xi(\boldsymbol\alpha)}\cdot \frac{\partial \xi(\boldsymbol\alpha)}{\partial \alpha_k}=\frac{\partial U}{\partial \alpha_k}\Big|_{\xi=\xi(\boldsymbol\alpha)}+4k\pi\cdot \frac{\partial \xi(\boldsymbol\alpha)}{\partial \alpha_k},$$
and 
$$\frac{\partial W^*}{\partial \alpha_k}=\frac{\partial U}{\partial \alpha_k}\Big|_{\xi=\xi^*(\boldsymbol\alpha)}+\frac{\partial U}{\partial \xi}\Big|_{\xi=\xi^*(\boldsymbol\alpha)}\cdot \frac{\partial \xi(\boldsymbol\alpha)}{\partial \alpha_k}=\frac{\partial U}{\partial \alpha_k}\Big|_{\xi=\xi^*(\boldsymbol\alpha)}+4k^*\pi\cdot \frac{\partial \xi(\boldsymbol\alpha)}{\partial \alpha_k}.$$
Then
\begin{equation*}
\frac{\partial F}{\partial \alpha_k}=\frac{1}{2}\Big(\frac{\partial U}{\partial \alpha_k}\Big|_{\xi=\xi(\boldsymbol\alpha)}-\frac{\partial U}{\partial \alpha_k}\Big|_{\xi=\xi^*(\boldsymbol\alpha)}\Big)+4k\pi\cdot \frac{\partial \xi(\boldsymbol\alpha)}{\partial \alpha_k}-4k^*\pi\cdot \frac{\partial \xi(\boldsymbol\alpha)}{\partial \alpha_k},
\end{equation*}
and
\begin{equation*}
\frac{\partial F^*}{\partial \alpha_k}=\frac{1}{2}\Big(\frac{\partial U}{\partial \alpha_k}\Big|_{\xi=\xi(\boldsymbol\alpha)}+\frac{\partial U}{\partial \alpha_k}\Big|_{\xi=\xi^*(\boldsymbol\alpha)}\Big)+4k\pi\cdot \frac{\partial \xi(\boldsymbol\alpha)}{\partial \alpha_k}+4k^*\pi\cdot \frac{\partial \xi(\boldsymbol\alpha)}{\partial \alpha_k};
\end{equation*}
and hence
\begin{equation}\label{dF}
\frac{\partial \mathrm{Im}F}{\partial \alpha_k}=\frac{1}{2}\mathrm{Im}\Big(\frac{\partial U}{\partial \alpha_k}\Big|_{\xi=\xi(\boldsymbol\alpha)}-\frac{\partial U}{\partial \alpha_k}\Big|_{\xi=\xi^*(\boldsymbol\alpha)}\Big)
\end{equation}
and
\begin{equation}\label{dG}
\frac{\partial \mathrm{Im}
F^*}{\partial \alpha_k}=\frac{1}{2}\mathrm{Im}\Big(\frac{\partial U}{\partial \alpha_k}\Big|_{\xi=\xi(\boldsymbol\alpha)}+\frac{\partial U}{\partial \alpha_k}\Big|_{\xi=\xi^*(\boldsymbol\alpha)}\Big).
\end{equation}

In the rest of the proof, we look at $\alpha_1,$ and the argument for the other $\alpha_k$'s follows verbatim.  By a direct computation, we have
\begin{equation}\label{Ua}
\begin{split}
\frac{\partial U}{\partial \alpha_1}=&\frac{\sqrt{-1}}{2}\cdot\log\frac{(1-u_1u_2u_3^{-1})(1-u_1u_2^{-1}u_3)(1-u_1u_5u_6^{-1})(1-u_1u_5^{-1}u_6)}{u_1^4(1-u_1^{-1}u_2u_3)(1-u_1^{-1}u_2^{-1}u_3^{-1})(1-u_1^{-1}u_5u_6)(1-u_1^{-1}u_5^{-1}u_6^{-1})}\\
&+\sqrt{-1}\cdot\log\frac{u_4(1-zu_1u_2u_3)(1-zu_1u_5u_6)}{(1-zu_1u_2u_4u_5)(1-zu_1u_3u_4u_6)}\quad\quad(\mathrm{mod}\ \pi),
\end{split}
\end{equation}
and as a consequence of (\ref{dF}) and (\ref{Ua}), we have
\begin{equation*}
\frac{\partial \mathrm{Im}F}{\partial \alpha_1}=\mathrm{Im}\bigg(\frac{\sqrt{-1}}{2}\cdot\log\frac{(1-z u_1u_2u_3)(1-zu_1u_5u_6)(1-z^*u_1u_2u_4u_5)(1-z^* u_1u_3u_4u_6)}{(1-z^*u_1u_2u_3)(1-z^*u_1u_5u_6)(1-zu_1u_2u_4u_5)(1-zu_1u_3u_4u_6)}\bigg).
\end{equation*}
Let $\mathrm R$ and $\mathrm S$ respectively be the terms in $(1-zu_1u_2u_3)(1-zu_1u_5u_6)(1-z^*u_1u_2u_4u_5)(1-z^*u_1u_3u_4u_6)$ not containing and containing $\sqrt{B^2-4AC}.$   Then by a direct computation (see also Murakami-Yano\,\cite{MY} and Ushijima\,\cite{U}),
 $$\mathrm R=8\mathrm QG_{34},$$
 where  $$\mathrm Q= \frac{1}{4}A^{-2}u_1^2u_4^{-1}(u_4u_5-u_3)(u_3u_4-u_5)(u_2u_4-u_6)(u_4u_6-u_2)$$
and $G_{ij}$ is the $ij$-th cofactor of the Gram matrix $G;$
 and
 $$\mathrm S =\mathrm Q\big(u_1^{-1}-u_1\big)\sqrt{B^2-4AC}=4\mathrm Q\big(u_1^{-1}-u_1\big)\sqrt{\det G},$$
 where the last equality comes from (\ref{detG}.)  
 Here recall that the placement of the entries of $G$ follows the rule that if $-\cos\theta_k$ is in row $i$ and column $j,$ then $\theta_k$ is the dihedral angle between the faces $F_i$ and $F_j.$  For example, $\theta_1$ is at the edge between $F_1$ and $F_2,$ which is also the edge connecting the vertices $\mathbf v_3$ and $\mathbf v_4.$

By Jacobi's Theorem (see \cite[2.5.1. Theorem]{P}),  
$$G_{34}^2-G_{33}G_{44}=(\cos^2\theta_1-1)\det G=-\sin^2\theta_1\det G.$$

\begin{enumerate}[(1)]
 
\item If $\alpha_1=\pi+\theta_1,$ then  $$\mathrm S=4\mathrm Q\cdot2\sqrt{-1}\sin\theta_1\cdot\sqrt{\det G}=-8\mathrm Q\sqrt{G_{34}^2-G_{33}G_{44}}.$$
Therefore,  we have
\begin{equation*}
\begin{split}
\frac{\partial \mathrm{Im}F}{\partial \alpha_1}=&\mathrm{Im}\bigg(\frac{\sqrt{-1}}{2}\cdot\log\frac{G_{34}-\sqrt{G_{34}^2-G_{33}c_{44}}}{G_{34}+\sqrt{G_{34}^2-G_{33}G_{44}}} \bigg).
\end{split}
\end{equation*}

Let $d_{34}$ be the distance between the vertices $\mathbf v_3$ and $\mathbf v_4.$ Then we consider the following cases. 

\begin{enumerate}[(a)]
\item If $G_{33}G_{44}>0$ and $G_{34}>0,$  then $l_1=d_{34}$ and
$$0<\frac{G_{34}-\sqrt{{G'}_{34}^2-G_{33}G_{44}}}{G_{34}+\sqrt{{G}_{34}^2-G_{33}G_{44}}}<1.$$
Then by (\ref{rr}) and (\ref{hh})
$$\frac{G_{34}-\sqrt{{G'}_{34}^2-G_{33}G_{44}}}{G_{34}+\sqrt{{G}_{34}^2-G_{33}G_{44}}}=e^{-2 d_{34}}=e^{-2l_1},$$
and
\begin{equation*}
\frac{\partial \mathrm{Im} F}{\partial \theta_1}=\frac{\partial \mathrm{Im} F}{\partial \alpha_1}=\mathrm{Im}\bigg(\frac{\sqrt{-1}}{2}\cdot\log e^{-2l_1}\bigg)=-l_1.
\end{equation*}

\item If $G_{33}G_{44}>0$ and $G_{34}<0,$  then $l_1=-d_{34}$ and
$$\frac{G_{34}-\sqrt{{G'}_{34}^2-G_{33}G_{44}}}{G_{34}+\sqrt{{G}_{34}^2-G_{33}G_{44}}}>1.$$
Then by (\ref{rr}) and (\ref{hh})
$$\frac{G_{34}-\sqrt{{G'}_{34}^2-G_{33}G_{44}}}{G_{34}+\sqrt{{G}_{34}^2-G_{33}G_{44}}}=e^{2 d_{34}}=e^{-2l_1},$$
and
\begin{equation*}
\frac{\partial \mathrm{Im} F}{\partial \theta_1}=\frac{\partial \mathrm{Im} F}{\partial \alpha_1}=\mathrm{Im}\bigg(\frac{\sqrt{-1}}{2}\cdot\log e^{-2l_1}\bigg)=-l_1.
\end{equation*}

 \item  If $G_{33}G_{44}<0$ and $G_{34}>0,$  then $l_1=d_{34}$ 
and
$$-1<\frac{G_{34}-\sqrt{{G'}_{34}^2-G_{33}G_{44}}}{G_{34}+\sqrt{{G}_{34}^2-G_{33}G_{44}}}<0.$$
Then by (\ref{rh})
$$\frac{G_{34}-\sqrt{{G'}_{34}^2-G_{33}G_{44}}}{G_{34}+\sqrt{{G}_{34}^2-G_{33}G_{44}}}=-e^{-2 d_{34}}=-e^{-2l_1},$$
and
\begin{equation*}
\frac{\partial \mathrm{Im} F}{\partial \theta_1}=\frac{\partial \mathrm{Im} F}{\partial \alpha_1}=\mathrm{Im}\bigg(\frac{\sqrt{-1}}{2}\cdot\log (-e^{-2l_1})\bigg)=-l_1.
\end{equation*}

\item If $G_{33}G_{44}<0$ and $G_{34}\leqslant 0,$ then $l_1=-d_{34}$ and
$$\frac{G_{34}-\sqrt{{G'}_{34}^2-G_{33}G_{44}}}{G_{34}+\sqrt{{G}_{34}^2-G_{33}G_{44}}}\leqslant -1.$$
Then by (\ref{rh})
$$\frac{G_{34}-\sqrt{{G'}_{34}^2-G_{33}G_{44}}}{G_{34}+\sqrt{{G}_{34}^2-G_{33}G_{44}}}=-e^{2 d_{34}}=-e^{-2l_1},$$
and
\begin{equation*}
\frac{\partial \mathrm{Im} F}{\partial \theta_1}=\frac{\partial \mathrm{Im} F}{\partial \alpha_1}=\mathrm{Im}\bigg(\frac{\sqrt{-1}}{2}\cdot\log (-e^{-2l_1})\bigg)=-l_1.
\end{equation*}
\end{enumerate}

\item If $\alpha_1=\pi-\theta_1,$ then  $$\mathrm S=-4\mathrm Q\cdot2\sqrt{-1}\sin\theta_1\cdot\sqrt{\det G}=8\mathrm Q\sqrt{G_{34}^2-G_{33}G_{44}}.$$
Therefore,  we have
\begin{equation*}
\begin{split}
\frac{\partial \mathrm{Im}F}{\partial \alpha_1}=&\mathrm{Im}\bigg(\frac{\sqrt{-1}}{2}\cdot\log\frac{G_{34}+\sqrt{G_{34}^2-G_{33}c_{44}}}{G_{34}-\sqrt{G_{34}^2-G_{33}G_{44}}} \bigg).
\end{split}
\end{equation*}

Let $d_{34}$ be the distance between the vertices $\mathbf v_3$ and $\mathbf v_4.$ Then we consider the following cases.

\begin{enumerate}[(a)]
\item If $G_{33}G_{44}>0$ and $G_{34}>0,$ then $l_1=d_{34}$ and 
$$\frac{G_{34}+\sqrt{{G'}_{34}^2-G_{33}G_{44}}}{G_{34}-\sqrt{{G}_{34}^2-G_{33}G_{44}}}>1.$$
Then by (\ref{rr}) and (\ref{hh}),
$$\frac{G_{34}+\sqrt{{G'}_{34}^2-G_{33}G_{44}}}{G_{34}-\sqrt{{G}_{34}^2-G_{33}G_{44}}}=e^{2 d_{34}}=e^{2 l_1},$$
and
\begin{equation*}
\frac{\partial \mathrm{Im} F}{\partial \theta_1}=-\frac{\partial \mathrm{Im} F}{\partial \alpha_1}=-\mathrm{Im}\bigg(\frac{\sqrt{-1}}{2}\cdot\log e^{2l_1}\bigg)=-l_1.
\end{equation*}

\item If $G_{33}G_{44}>0$ and $G_{34}<0,$ then $l_1=-d_{34}$ and 
$$0<\frac{G_{34}+\sqrt{{G'}_{34}^2-G_{33}G_{44}}}{G_{34}-\sqrt{{G}_{34}^2-G_{33}G_{44}}}<1.$$
Then by (\ref{rr}) and (\ref{hh}),
$$\frac{G_{34}+\sqrt{{G'}_{34}^2-G_{33}G_{44}}}{G_{34}-\sqrt{{G}_{34}^2-G_{33}G_{44}}}=e^{-2 d_{34}}=e^{2l_1},$$
and
\begin{equation*}
\frac{\partial \mathrm{Im} F}{\partial \theta_1}=-\frac{\partial \mathrm{Im} F}{\partial \alpha_1}=-\mathrm{Im}\bigg(\frac{\sqrt{-1}}{2}\cdot\log e^{2l_1}\bigg)=-l_1.
\end{equation*}

\item If $G_{33}G_{44}<0$ and $G_{34}>0,$ then $l_1=d_{34}$  and 
$$\frac{G_{34}+\sqrt{{G'}_{34}^2-G_{33}G_{44}}}{G_{34}-\sqrt{{G}_{34}^2-G_{33}G_{44}}}<-1.$$
Then by (\ref{rh}),
$$\frac{G_{34}+\sqrt{{G'}_{34}^2-G_{33}G_{44}}}{G_{34}-\sqrt{{G}_{34}^2-G_{33}G_{44}}}=-e^{2 d_{34}}=-e^{2 l_1},$$
and
\begin{equation*}
\frac{\partial \mathrm{Im} F}{\partial \theta_1}=-\frac{\partial \mathrm{Im} F}{\partial \alpha_1}=-\mathrm{Im}\bigg(\frac{\sqrt{-1}}{2}\cdot\log (-e^{2l_1})\bigg)=-l_1.
\end{equation*}

\item If $G_{33}G_{44}<0$ and $G_{34}\leqslant 0,$  then $l_1=-d_{34}$ and 
$$-1\leqslant \frac{G_{34}+\sqrt{{G'}_{34}^2-G_{33}G_{44}}}{G_{34}-\sqrt{{G}_{34}^2-G_{33}G_{44}}}<0.$$
Then by (\ref{rh}),
$$\frac{G_{34}+\sqrt{{G'}_{34}^2-G_{33}G_{44}}}{G_{34}-\sqrt{{G}_{34}^2-G_{33}G_{44}}}=-e^{-2 d_{34}}=-e^{2l_1},$$
and
\begin{equation*}
\frac{\partial \mathrm{Im} F}{\partial \theta_1}=-\frac{\partial \mathrm{Im} F}{\partial \alpha_1}=-\mathrm{Im}\bigg(\frac{\sqrt{-1}}{2}\cdot\log (-e^{2l_1})\bigg)=-l_1.
\end{equation*}
\end{enumerate}
\end{enumerate}
This proves (\ref{F}).

Next we prove (\ref{G}). By (\ref{Ua}) and (\ref{G}), we have 
\begin{equation*}
\begin{split}
\frac{\partial U}{\partial \alpha_1}=\mathrm{Im}\bigg(&\sqrt{-1}\cdot\log\frac{(1-u_1u_2u_3^{-1})(1-u_1u_2^{-1}u_3)(1-u_1u_5u_6^{-1})(1-u_1u_5^{-1}u_6)}{u_1^4(1-u_1^{-1}u_2u_3)(1-u_1^{-1}u_2^{-1}u_3^{-1})(1-u_1^{-1}u_5u_6)(1-u_1^{-1}u_5^{-1}u_6^{-1})}\\
&+\sqrt{-1}\cdot\log\frac{u_4^2(1-zu_1u_2u_3)(1-z^*  u_1u_2u_3)(1-zu_1u_5u_6)(1-z^*  u_1u_5u_6)}{(1-zu_1u_2u_4u_5)(1-z^*u_1u_2u_4u_5)(1-z  u_1u_3u_4u_6)(1-z^*  u_1u_3u_4u_6)}\bigg).
\end{split}
\end{equation*}
Then (\ref{G}) follows from the  following direct computations
$$(1-z u_1u_2u _3)(1-z^*  u_1u_2u_3)=\frac{1}{A }\frac{(u_1u_2u_3)^2}{u_4u_5u_6}\Big(1-\frac{u_4u_5}{u_3}\Big)\Big(1-\frac{u_4u_6}{u_2}\Big)\Big(1-\frac{u_5u_6}{u_1}\Big)\Big(1-\frac{1}{u_1u_2u_3}\Big),$$
$$(1-z u_1u_5u_6)(1-z^*  u_1u_5u_6)=\frac{1}{A }\frac{(u_1u_5u_6)^2}{u_2u_3u_4}\Big(1-\frac{u_2u_4}{u_6}\Big)\Big(1-\frac{u_3u_4}{u_5}\Big)\Big(1-\frac{u_2u_3}{u_1}\Big)\Big(1-\frac{1}{u_1u_5u_6}\Big),$$
$$(1-z u_1u_2u_4u_5)(1-z^*u_1u_2u_4u_5)=\frac{1}{A}\frac{(u_1u_2u_4u_5)^2}{u_3u_6}\Big(1-\frac{u_3}{u_4u_5}\Big)\Big(1-\frac{u_6}{u_2u_4}\Big)\Big(1-\frac{u_6}{u_1u_5}\Big)\Big(1-\frac{u_3}{u_1u_2}\Big),$$ 
$$(1-z  u_1u_3u_4u_6)(1-z^*  u_1u_3u_4u_6)=\frac{1}{A}\frac{(u_1u_3u_4u_6)^2}{u_2u_5}\Big(1-\frac{u_2}{u_4u_6}\Big)\Big(1-\frac{u_5}{u_3u_4}\Big)\Big(1-\frac{u_2}{u_1u_3}\Big)\Big(1-\frac{u_5}{u_1u_6}\Big).$$

We are left to show that  the constant $C$ in (\ref{sch}) equals $0.$ First by a direct computation we have  for $(\theta_1,\dots,\theta_6)=(\pi,0,0,\pi,0,0)$ or $(\pi,\pi,\pi,0,0,0)$ that $\xi(\boldsymbol\alpha)=\frac{3\pi}{2}$ and $V(\xi(\boldsymbol\alpha))=0,$ which equals the volume of a ``flat tetrahedron". Hence $C=0$ in these degenerate cases. 

Then we claim that each generalized hyperbolic tetrahedron $\Delta$ without ideal vertices can be continuously deformed into one of the previous two cases without changing the type of the vertices along the way.  We consider the following three cases.
\begin{enumerate}[(1)]
\item If $\Delta$ has at least one regular vertex, say, $\mathbf v_1,$ then project $\mathbf v_1$ into $F_1$ along the shortest geodesic connecting the two provides the desired deformation (with a small perturbation in the non-generic case that the projection of $\mathbf v_1$ lines on an edge of $\Delta$).

\item  If all the vertices of $\Delta$ are hyperideal and all the edges are positive,  then by Bonahon-Bao\,\cite{BB} that the space of such hyperbolic tetrahedra is connected, $\Delta$ can be continuously deformed into a flat one.

\item If all the vertices of $\Delta$ are hyperideal and at least one edge, say, $e_{34}$ (that connects $\mathbf v_1$ and $\mathbf v_2$) is negative, then moving $\mathbf v_1$ sufficiently close to  $\mathbf v_2$ along the line segment  $L^+_{ij}$ connecting them and then projecting $\mathbf v_1$ to the plane to $F_1$ provides the desired deformation. 
\end{enumerate}

Finally, by the continuity of both the volume function and $V(\xi(\boldsymbol\alpha)),$ $C=0$ for all generalized hyperbolic tetrahedron, and 
$$\mathrm{Vol}(\Delta)=V(\xi(\boldsymbol\alpha)).$$
\end{proof}


\section{Growth of quantum $6j$-symbols}\label{growth}

The main results of this section is Theorem \ref{asymp1}, which consider the growth rate of quantum $6j$-symbols.
For each $k\in\{1,\dots,6\},$ let 
$$\alpha_k=\lim_{r\to\infty} \frac{2\pi a_k^{(r)}}{r};$$
and let 
$$\theta_k=|\pi -\alpha_k|,$$
or equivalently,
$$\alpha_k=\pi\pm\theta_k.$$
We observe that if  the $6$-tuple $(a_1^{(r)},\dots,a_6^{(r)})$ is $r$-admissible for each $r\geqslant 3$, then the $6$-tuple $(\alpha_1,\dots,\alpha_6)$ is admissible. 
Let 
 $$T^{(r)}_1=\frac{a^{(r)}_1+a^{(r)}_2+a^{(r)}_3}{2},\ T^{(r)}_2=\frac{a^{(r)}_1+a^{(r)}_5+a^{(r)}_6}{2},\ T^{(r)}_3=\frac{a^{(r)}_2+a^{(r)}_4+a^{(r)}_6}{2}, \ T^{(r)}_4=\frac{a^{(r)}_3+a^{(r)}_4+a^{(r)}_5}{2},$$ 
 $$Q^{(r)}_1=\frac{a^{(r)}_1+a^{(r)}_2+a^{(r)}_4+a^{(r)}_5}{2}, \ \ Q^{(r)}_2=\frac{a^{(r)}_1+a^{(r)}_3+a^{(r)}_4+a^{(r)}_6}{2}, \ \ Q^{(r)}_3=\frac{a^{(r)}_2+a^{(r)}_3+a^{(r)}_5+a^{(r)}_6}{2}.$$
 For $i\in\{1,2,3,4\},$ let  
$$\tau_i=\lim_{r\to\infty}\frac{2\pi T^{(r)}_i}{r}$$
and for $j\in\{1,2,3\},$ let   
$$\eta_j=\lim_{r\to\infty}\frac{2\pi Q^{(r)}_j}{r}.$$
Then
$$\tau_1=\frac{\alpha_1+\alpha_2+\alpha_3}{2}, \quad \tau_2=\frac{\alpha_1+\alpha_5+\alpha_6}{2},\quad \tau_3=\frac{\alpha_2+\alpha_4+\alpha_6}{2},\quad \tau_4=\frac{\alpha_3+\alpha_4+\alpha_5}{2},$$
$$\eta_1=\frac{\alpha_1+\alpha_2+\alpha_4+\alpha_5}{2},\quad\eta_2=\frac{\alpha_1+\alpha_3+\alpha_4+\alpha_6}{2},\quad \eta_3=\frac{\alpha_2+\alpha_3+\alpha_5+\alpha_6}{2}.$$

One of the main ingredients used to estimate the quantum $6j$-symbol is the following lemma, first appeared in Garoufalidis-Le\,\cite[Proposition 8.2]{gale} for $q=e^{\frac{\pi\sqrt{-1}}{r}},$ and then in the other roots of unity in Detcherry-Kalfagianni\,\cite[Proposition 4.1]{DK}.

\begin{lemma}\label{stima}
 For any integer $0<n<r,$ let $\{n\}=q^n-q^{-n}$ and $\{n\}!=\prod_{k=1}^n\{k\}.$ Then at $q=e^{\frac{2\pi \sqrt{-1}}{r}},$
$$
  \log\left|\{n\}!\right|=-\frac{r}{2\pi}\Lambda\bigg(\frac{2n\pi}{r}\bigg)+O(\log r),
$$
where the term $O(\log r)$ is such that there exist constants $C,r_0$ independent of $n$ and $r$ such that 
$O(\log r)\leqslant C \log r$ whenever $r>r_0.$
\end{lemma}

The other main ingredient is the following symmetry of quantum $6j$-symbols established in Detcherry-Kalfagianni-Yang\,\cite[Lemma A.3]{DKY}.

\begin{lemma}\label{sym}
For   $a\in\{0,...,r-2\},$ let $a'=r-2-a.$ Then at $q=e^{\frac{2\pi \sqrt{-1}}{r}},$
$$\bigg|\begin{matrix}
a_1 & a_2 & a_3 \\
   a_4 & a_5  & a_6
  \end{matrix}\bigg|=\bigg|\begin{matrix}
a_1 & a_2 & a_3 \\
   a_4' & a_5'  & a_6'
  \end{matrix}\bigg|=\bigg|\begin{matrix}
a_1 & a_2' & a_3' \\
   a_4 & a_5'  & a_6'
  \end{matrix}\bigg|.$$
\end{lemma}

We call  the operation that changes the three $a_i$'s at the edges around a face to $a_i'$s a \emph{change of colors operation around a face}, and  the operation that changes the four $a_i$'s at the edges around a quadrilateral to $a_i'$s a \emph{change of colors  operation around a quadrilateral}.
 Then Lemma \ref{sym} says that a quantum $6j$-symbol remains unchanged under a change of colors operation around a face or a quadrilateral. 

\begin{lemma}\label{3poss} Every quantum $6j$-symbol is equal to one that is in one of the following three cases.
\begin{enumerate}[(1)]
\item $a_i < \frac{r-2}{2}$ for all $i\in\{1,\dots,6\}.$ 

\item $a_i> \frac{r-2}{2}$ for exactly one $i\in\{1,\dots,6\}.$ 

\item  $a_i> \frac{r-2}{2}$  and $a_j> \frac{r-2}{2}$ for exactly one opposite pair $\{i,j\}\in\big\{ \{1,4\}, \{2,5\} , \{3,6\}\big\}.$  
\end{enumerate}
\end{lemma}

\begin{proof}  In the proof, we will call an $a_i$ \emph{big} if it is greater than $\frac{r-2}{2}.$ 

If the quantum $6j$-symbol contains exactly zero, one, or two opposite big $a_i$'s, then the result holds automatically. 

If the quantum $6j$-symbol contains exactly two adjacent  big $a_i$'s, then they must belong to a common face. Doing a change of colors operation around that  face will reduce the number of big $a_i$'s to one. 

If the quantum $6j$-symbol contains exactly three big $a_i$'s, then either they are around a face in which case a change of colors operation around that  face will reduce the number of big $a_i$'s to zero, or they are contained in a quadrilateral in which case a change of colors operation around that quadrilateral will reduce the number of big $a_i$'s to one, or they are around a vertex in which case two of them are in a common face and a change of colors operation around that  face will change the big $a_i$'s into an opposite pair.  

If the quantum $6j$-symbol contains exactly four  big $a_i$'s, then either they belong to a quadrilateral in which case a change of colors operation around that quadrilateral will reduce the number of big $a_i$'s to zero, or they contain a face in which case a change of colors operation around that  face will reduce the number of big $a_i$'s to one,

If the quantum $6j$-symbol contains exactly five or six  big $a_i$'s, then they must contain a quadrilateral in which case a change of colors operation around that quadrilateral will reduce the number of big $a_i$'s to one in the former case, and change the big $a_i$'s into an opposite pair in the latter case. 
\end{proof}

\begin{proof}[Proof of Theorem \ref{asymp1} (2)]  The proof follows the same idea of Costantino\,\cite[Theorem 1.2]{C}. For the simplicity of the notations, we will write $T_i$ for $T^{(r)}_i$ for $i\in\{1,2,3,4\},$  and write $Q_j$ for $Q^{(r)}_j$ for $j\in\{1,2,3\}.$

 By Lemma \ref{stima},  we have 
\begin{equation}\label{del}
\lim_{r\to\infty}\frac{2\pi}{r}\log\Delta(a^{(r)}_i,a^{(r)}_j,a^{(r)}_k)|=\delta(\alpha_i,\alpha_j,\alpha_k).
\end{equation}

Next, we study the asymptotics of 
$$S=\sum_{z=\max\{T_1,T_2,T_3,T_4\}}^{\min_j\{Q_1,Q_2,Q_3\}}\frac{(-1)^z[z+1]!}{\prod_{i=1}^4[z-T_i]!\prod_{j=1}^3[Q_j-z]!}.$$
Let 
$$S_z=\frac{(-1)^z[z+1]!}{\prod_{i=1}^4[z-T_i]!\prod_{j=1}^3[Q_j-z]!}.$$

The proof contains the following three steps.

\begin{enumerate}[Step 1.]
\item Since $[z+1]!=0$ when $z>r-2,$ $S_z=0$ for those $z.$ Hence we only need to consider $S_z$ for $z$ in between $\max\{T_1, T_2, T_3, T_4\}$ and $\min\{Q_1,Q_2, Q_3, r-2\}.$ We will show that for those $z,$ all $S_z$ have the same sign so the growth rate of the sum is determined by that of the largest term. 
\item If $\lim_{r\to\infty}\frac{2\pi z^{(r)}}{r}=\xi,$ then by Lemma \ref{stima} we have
$$\lim_{r\to\infty}\frac{2\pi}{r}\log |S_{z^{(r)}}|=\sum_{i=1}^4\Lambda(\xi-\tau_i)+\sum_{j=1}^3\Lambda(\eta_j-\xi)-\Lambda(\xi).$$
We will show that the function $s$  defined by
$$s(\xi)=\sum_{i=1}^4\Lambda(\xi-\tau_i)+\sum_{j=1}^3\Lambda(\eta_j-\xi)-\Lambda(\xi)$$
has a unique maximum point $\xi_0$ on the interval $I=[\max\{\tau_1,\tau_2,\tau_3,\tau_4\}, \min\{\eta_1,\eta_2,\eta_3,2\pi\}]$ so that the  growth rate of $S$ equals $s(\xi_0),$ and hence the growth rate of the quantum $6j$-symbol equals $V(\xi_0),$ where $V$ is the function defined in (\ref{f}).  

\item We will show that $\xi_0=\xi(\boldsymbol\alpha)$ so that  $V(\xi_0)=V(\xi(\boldsymbol\alpha)),$ which by Theorem \ref{volume} equals  $\mathrm{Vol}(\Delta).$ 
\end{enumerate}

We accomplish Step 1 by showing that the ratio between to consecutive summands $S_z$ and $S_{z-1}$ is positive. By a direct computation, we have
$$\frac{S_{z}}{S_{z-1}}=-\frac{\sin{\frac{2\pi(z+1)}{r}}\sin{\frac{2\pi(Q_1-z+1)}{r}}\sin{\frac{2\pi(Q_2-z+1)}{r}}\sin{\frac{2\pi(Q_3-z+1)}{r}}}{\sin{\frac{2\pi(z-T_1)}{r}}\sin{\frac{2\pi(z-T_2)}{r}}\sin{\frac{2\pi(z-T_3)}{r}}\sin{\frac{2\pi(z-T_4)}{r}}}$$
for $z$ satisfying $\max\{T_1,T_2,T_3,T_4\}+1\leqslant z\leqslant \min\{Q_1,Q_2,Q_3,r-2\}.$ The heuristic idea here 
is to look at the sign of the following function 
 $$h(\xi)=-\frac{\sin(\xi)\sin(\eta_1-\xi)\sin(\eta_2-\xi)\sin(\eta_3-\xi)}{\sin(\xi-\tau_1)\sin(\xi-\tau_2)\sin(\xi-\tau_3)\sin(\xi-\tau_4)}.$$
If $h(\xi)>0$ for every $\xi,$ than  for $r$ sufficiently large, $\frac{S_z}{S_{z-1}}>0$ for every $z.$ To this end, by Lemma \ref{3poss}, we only need to consider the  three possibilities listed there.

\begin{enumerate}[(1)]
\item In this case, we have   $\alpha_1,\dots,\alpha_6\leqslant  \pi.$ We mention here that this case is the only place where we need condition (2) that $G_{ii}<0$ for some $i\in\{1,2,3,4\}.$

By reindexing $\tau_i$'s and $\eta_j$'s if necessary, we assume that the vertex $\mathbf v_4$ is a hyperideal vertex. Then by Lemma \ref{hyper} (2),
\begin{equation}\label{hyperideal}
\theta_3+\theta_4+\theta_5<\pi.
\end{equation}

Since $\alpha_i\leqslant \pi$ for each $i\in\{1,\dots,6\},$  we have  $\theta_i=\pi-\alpha_i$ and (\ref{hyperideal}) is equivalent to 
\begin{equation}\label{2.2}
\tau_4>\pi.
\end{equation}

As a consequence of (\ref{2.2}), we have $I\subset [\pi, 2\pi],$ and in particular,
\begin{equation}\label{2.4}
\xi\in[\pi,2\pi].
\end{equation}
By the  condition that all $\alpha_i\leqslant \pi$, we have
that for all $i\in\{1,2,3,4\}$ and $j\in\{1,2,3\},$
\begin{equation}\label{2.5}
\eta_j-\tau_i=\frac{\alpha_k+\alpha_l-\alpha_m}{2}\leqslant \frac{\alpha_k+\alpha_l}{2}\leqslant \pi
\end{equation}
for some triple $(k,l,m)$ around a vertex. As a consequence, we have
\begin{equation}\label{2.6}
0\leqslant \xi-\tau_i\leqslant \eta_1-\tau_i\leqslant \pi
\end{equation}
 for  all $i\in\{1,2,3,4\},$ and
\begin{equation}\label{2.7}
0\leqslant \eta_j-\xi\leqslant \eta_j-\tau_1\leqslant \pi
\end{equation}
 for all $j\in\{1,2,3\}.$

\item In this case, we assume, say, $\alpha_1\geqslant \pi$ and $\alpha_2,\dots,\alpha_6\leqslant  \pi.$ Then we have
\begin{equation}\label{3.1}
\tau_1=\frac{\alpha_1+\alpha_2+\alpha_3}{2}=\alpha_1+\frac{\alpha_2+\alpha_3-\alpha_1}{2}\geqslant \alpha_1\geqslant \pi,
\end{equation} 
where the penultimate inequality comes from the admissibility conditions. As a consequence of (\ref{3.1}), we have $I\subset [\pi, 2\pi],$ and in particular, 
\begin{equation}\label{3.2}
\xi\in [\pi,2\pi].
\end{equation}
By bounding $\xi$ from above by $\eta_3,$ we have
\begin{equation}\label{3.3}
 \xi-\tau_i\in[0,\pi]
\end{equation}
for all $i\in\{1,2,3,4\}.$ Indeed, 
$$0\leqslant\xi-\tau_1\leqslant\eta_3-\tau_1=\frac{\alpha_5+\alpha_6-\alpha_1}{2}\leqslant\frac{\alpha_5+\alpha_6}{2}\leqslant  \pi.$$
$$0\leqslant \xi-\tau_2\leqslant \eta_3-\tau_2=\frac{\alpha_2+\alpha_3-\alpha_1}{2}\leqslant \frac{\alpha_2+\alpha_3}{2}\leqslant\pi.$$
$$0\leqslant \xi-\tau_3\leqslant \eta_3-\tau_3=\frac{\alpha_3+\alpha_5-\alpha_4}{2}\leqslant \frac{\alpha_3+\alpha_5}{2}\leqslant \pi.$$
$$0\leqslant \xi-\tau_4\leqslant \eta_3-\tau_4=\frac{\alpha_2+\alpha_6-\alpha_4}{2}\leqslant \frac{\alpha_2+\alpha_6}{2}\leqslant \pi.$$
Also, by bounding $\xi$ from below by $\tau_1,$ we have
\begin{equation}\label{3.4}
\eta_j-\xi \in[0,\pi]
\end{equation}
for all $j\in\{1,2,3\}.$ Indeed, 
$$0\leqslant \eta_1-\xi\leqslant \eta_1-\tau_1=\frac{\alpha_4+\alpha_5-\alpha_3}{2}\leqslant \frac{\alpha_4+\alpha_5}{2}\leqslant\pi.$$
$$0\leqslant \eta_2-\xi\leqslant  \eta_2-\tau_1=\frac{\alpha_4+\alpha_6-\alpha_2}{2}\leqslant \frac{\alpha_4+\alpha_6}{2}\leqslant\pi.$$
$$0\leqslant \eta_3-\xi\leqslant  \eta_2-\tau_1=\frac{\alpha_2+\alpha_3-\alpha_1}{2}\leqslant \frac{\alpha_2+\alpha_3}{2}\leqslant\pi.$$

\item In this case, we have, say, $\alpha_1,\alpha_4 \geqslant \pi$ and $\alpha_2, \alpha_3,\alpha_5, \alpha_6 \leqslant \pi.$  We claim that 
\begin{equation}\label{4.1}
\pi \leqslant \tau_i\leqslant 2\pi
\end{equation}
for each $i\in\{1,2,3,4\},$ and 
\begin{equation}\label{4.2}
0 \leqslant \eta_j-\tau_i\leqslant \pi
\end{equation}
for  each $i\in\{1,2,3,4\}$ and $j\in\{1,2,3\}.$ Indeed, the second half of (\ref{4.1}) and the first half of (\ref{4.2}) come from the admissibility conditions of $(\alpha_1,\dots,\alpha_6).$ 
For first half of (\ref{4.1}), we have
$$\tau_i=\frac{\alpha_j+\alpha_k+\alpha_l}{2}$$
for some triple $(j,k,l)$ around a vertex. Since $\alpha_1$ and $\alpha_4$ are angles of opposite edges, without loss of generality, we may assume that $(j,k,l)=(1,2,3).$ Then
\begin{equation}\label{4.15}
\frac{\alpha_1+\alpha_2+\alpha_3}{2}=\alpha_1+\frac{\alpha_2+\alpha_3-\alpha_1}{2}\geqslant \pi,
\end{equation}
where the last inequality comes from the admissibility conditions. 
For  the second half of (\ref{4.2}),
 we have
 $$\eta_j-\tau_i=\frac{\alpha_k+\alpha_l-\alpha_m}{2}$$
for some triple $(k,l,m)$ around a vertex. Again,  without of generality, assume that $\{j,k,l\}=\{1,2,3\}.$ Then we have
  $$\frac{\alpha_2+\alpha_3-\alpha_1}{2}\leqslant \frac{\alpha_2+\alpha_3}{2}\leqslant\pi,$$ 
  $$\frac{\alpha_1+\alpha_2-\alpha_3}{2}\leqslant \frac{(\alpha_2+\alpha_3)+(\alpha_2-\alpha_3)}{2}=\alpha_2\leqslant\pi,$$ and
  $$\frac{\alpha_1+\alpha_3-\alpha_2}{2}\leqslant \frac{(\alpha_2+\alpha_3)+(\alpha_3-\alpha_3)}{2}=\alpha_3\leqslant\pi,$$
  where the first inequalities in the last two cases come from the admissibility conditions. As a consequence of (\ref{4.1}), we have  $I\subset[\pi, 2\pi],$ and in particular, 
\begin{equation}\label{4.3}
\xi\in[\pi, 2\pi].
\end{equation}
As a consequence of (\ref{4.2}), we have
\begin{equation}\label{4.6}
0\leqslant \xi-\tau_i\leqslant \eta_1-\tau_i\leqslant \pi
\end{equation}
 for  all $i\in\{1,2,3,4\},$ and
\begin{equation}\label{4.7}
0\leqslant \eta_j-\xi\leqslant \eta_j-\tau_1\leqslant \pi
\end{equation}
 for all $j\in\{1,2,3\}.$ 
\end{enumerate}

From (\ref{2.4}), (\ref{2.6}) and (\ref{2.7}) in Case (1), (\ref{3.2}), (\ref{3.3}) and (\ref{3.4}) in Case (2), and (\ref{4.3}), (\ref{4.6}) and (\ref{4.7}) in Case (3), we have for sufficiently large $r$ that 
\begin{equation*} 
\frac{r-2}{2}<z<r-2,
\end{equation*}
 \begin{equation*} 0 < z-T_i< \frac{r-2}{2}
\end{equation*}
 for  all $i\in\{1,2,3,4\},$ and
\begin{equation*} 
0<Q_j-z<\frac{r-2}{2}
\end{equation*}
 for all $j\in\{1,2,3\}.$ 
As a consequence, we have  $\frac{S_z}{S_{z-1}}>0$ for all $z$ in the range, and all the $S_z$ have the same sign. This completes Step 1.
 \\

We accomplish Step 2 as follows. If $\max\{\tau_1,\tau_2,\tau_3,\tau_4\}=\min\{\eta_1,\eta_2,\eta_3, 2\pi\},$ then the interval $I$ is a single point and the result holds automatically.  If $\max\{\tau_1,\tau_2,\tau_3,\tau_4\}<\min\{\eta_1,\eta_2,\eta_3, 2\pi\},$ then we show that the function  $s(\xi)$ is strictly concave down on the interval 
$I$  and that the derivative $s'(\xi)$ has different signs at the two end points of $I.$ In this case, we first have 
\begin{equation}\label{5.1}
2\pi-\tau_i>0
\end{equation}
for each $i\in\{1,2,3,4\},$
and
\begin{equation}\label{5.2}
\eta_j-\tau_i>0
\end{equation}
for each $i\in\{1,2,3,4\}$ and $j\in\{1,2,3\}.$  Then we compute
\begin{equation}\label{f'}
s'(\xi)=\log\bigg(\frac{\sin(2\pi-\xi)\sin(\eta_1-\xi)\sin(\eta_2-\xi)\sin(\eta_3-\xi)}{\sin(\xi-\tau_1)\sin(\xi-\tau_2)\sin(\xi-\tau_3)\sin(\xi-\tau_4)}\bigg),
\end{equation}
and 
\begin{equation}\label{f''}
s''(\xi)=-\sum_{i=1}^4\cot(\xi-\tau_i)-\sum_{j=1}^3\cot(\eta_j-\xi)-\cot(2\pi-\xi).
\end{equation}
As a consequence, we have
\begin{equation}\label{lim}
\lim _{\xi\to \max\{\tau_1,\tau_2,\tau_3,\tau_4\}^+} s'(\xi)=+\infty\quad\text{and}\quad\lim _{\xi\to \min\{\eta_1,\eta_2,\eta_3,2\pi\}^-} s'(\xi)=-\infty.
\end{equation}

We still consider the three cases of Lemma \ref{3poss}. 
 
\begin{enumerate}[(1)]

\item Recall in this case, we have   $\alpha_1,\dots,\alpha_6\leqslant  \pi.$  Then by (\ref{2.2}) and (\ref{5.1}), we have
\begin{equation}\label{5.3}
0<(\xi-\tau_4)+(2\pi-\xi)=2\pi-\tau_4< \pi,
\end{equation}
 and by (\ref{2.5}) and (\ref{5.2}), we have for each $i\in\{1,2,3\}$ that
\begin{equation}\label{5.4}
0<(\xi-\tau_i)+(\eta_i-\xi)=\eta_i-\tau_i\leqslant \pi.
\end{equation}
In particular, both inequalities in (\ref{5.3}) are strict. Therefore, by  (\ref{f''}),  (\ref{2.4}), (\ref{2.6}), (\ref{2.7}), (\ref{5.3}), (\ref{5.4}) and Lemma \ref{lemma} below, we have
\begin{equation}\label{concave1}
s''(\xi)=-\big(\cot(\xi-\tau_4)+\cot(2\pi-\xi)\big)-\sum_{i=1}^3\big(\cot(\xi-\tau_i)+\cot(\eta_i-\xi)\big)<0.
\end{equation}

\item Recall in this case, we have $\alpha_1\geqslant \pi$ and $\alpha_2,\dots,\alpha_6\leqslant  \pi.$ Then we have the following two sub-cases.

\begin{enumerate}[(\text{2}.1)]
\item $\tau_4>\pi.$ In this case, by (\ref{5.1}) and (\ref{5.2}) we have
\begin{equation}\label{3.5}
0<  (\xi-\tau_4)+(2\pi-\xi)=2\pi-\tau_4< \pi,
\end{equation}

\begin{equation}\label{3.6}
0< (\xi-\tau_1)+(\eta_1-\xi)=\eta_1-\tau_1=\frac{\alpha_4+\alpha_5-\alpha_3}{2}\leqslant \frac{\alpha_4+\alpha_5}{2}\leqslant \pi,
\end{equation}

\begin{equation}\label{3.7}
0<(\xi-\tau_2)+(\eta_2-\xi)=\eta_2-\tau_2=\frac{\alpha_3+\alpha_4-\alpha_5}{2}\leqslant \frac{\alpha_3+\alpha_4}{2}\leqslant \pi,
\end{equation}

\begin{equation}\label{3.8}
0<  (\xi-\tau_3)+(\eta_3-\xi)=\eta_3-\tau_3=\frac{\alpha_3+\alpha_5-\alpha_4}{2}\leqslant \frac{\alpha_3+\alpha_5}{2}\leqslant \pi.
\end{equation}
In particular, both inequalities in (\ref{3.5}) are strict. Therefore, by  (\ref{f''}),  (\ref{3.2}), (\ref{3.3}), (\ref{3.4}), (\ref{3.5}), (\ref{3.6}), (\ref{3.7}), (\ref{3.8}) and Lemma \ref{lemma} below, we have
\begin{equation}\label{concave2}
s''(\xi)=-\big(\cot(\xi-\tau_4)+\cot(2\pi-\xi)\big)-\sum_{i=1}^3\big(\cot(\xi-\tau_i)+\cot(\eta_i-\xi)\big)<0.
\end{equation}

\item $\tau_4\leqslant \pi.$ In this case the key observation is that
\begin{equation}\label{7.2}
\tau_3+\tau_4> \eta_1.
\end{equation}
Indeed,  by (\ref{5.2}), 
\begin{equation*}
\tau_3+\tau_4-\eta_1=\frac{\alpha_3+\alpha_4+\alpha_6-\alpha_1}{2}\geqslant \frac{\alpha_5+\alpha_6-\alpha_1}{2}=\eta_3-\tau_1>0,
\end{equation*}
where the first inequalities come from the admissibility conditions. From (\ref{5.1}), (\ref{5.2}) and (\ref{7.2}),  we have
\begin{equation}\label{3.9}
0< (\xi-\tau_1)+(2\pi-\xi)=2\pi-\tau_1\leqslant \pi,
\end{equation}

\begin{equation}\label{3.10}
0<(\xi-\tau_2)+(\eta_2-\xi)=\eta_2-\tau_2=\frac{\alpha_3+\alpha_4-\alpha_5}{2}\leqslant\frac{\alpha_3+\alpha_4}{2}\leqslant  \pi,
\end{equation}

\begin{equation}\label{3.11}
0<(\xi-\tau_4)+(\eta_3-\xi)=\eta_3-\tau_4=\frac{\alpha_2+\alpha_6-\alpha_4}{2}\leqslant \frac{\alpha_2+\alpha_6}{2}\leqslant \pi,
\end{equation}

\begin{equation}\label{3.12}
0< (\xi-\tau_3)+(\eta_1-\xi)=\eta_1-\tau_3<  \tau_4 \leqslant \pi.
\end{equation}
In particular, the first two inequalities in (\ref{3.12}) are strict. Therefore, by  (\ref{f''}), (\ref{3.2}), (\ref{3.3}), (\ref{3.4}), (\ref{3.9}), (\ref{3.10}), (\ref{3.11}), (\ref{3.12}) and Lemma \ref{lemma} below, we have
\begin{equation}\label{concave3}
\begin{split}
s''(\xi)=&-\big(\cot(\xi-\tau_1)+\cot(2\pi-\xi)\big)-\big(\cot(\xi-\tau_2)+\cot(\eta_2-\xi)\big)\\
&-\big(\cot(\xi-\tau_4)+\cot(\eta_3-\xi)\big)-\big(\cot(\xi-\tau_3)+\cot(\eta_1-\xi)\big)<0.
\end{split}
\end{equation} 
\end{enumerate}

\item Recall in this case, we have $\alpha_1, \alpha_4\geqslant \pi$ and $\alpha_2,\alpha_3,\alpha_5,\alpha_6\leqslant  \pi.$  In this case, we first observe that 
\begin{equation}\label{12.2}
\tau_4>\pi
\end{equation}
Indeed, by (\ref{5.2}),
$$\tau_4=\frac{\alpha_3+\alpha_4+\alpha_5}{2}=\alpha_4+\frac{\alpha_3+\alpha_5-\alpha_4}{2}=\alpha_4+\eta_3-\tau_3>\pi.$$
 Then by (\ref{12.2}) and (\ref{5.1}), we have
\begin{equation}\label{8.3}
0<(\xi-\tau_4)+(2\pi-\xi)=2\pi-\tau_4< \pi,
\end{equation}
 and by (\ref{4.2}) and (\ref{5.2}), we have for each $i\in\{1,2,3\}$ that
\begin{equation}\label{8.4}
0<(\xi-\tau_i)+(\eta_i-\xi)=\eta_i-\tau_i\leqslant \pi.
\end{equation}
In particular, both inequalities in (\ref{8.3}) are strict. Therefore, by  (\ref{f''}),  (\ref{4.3}), (\ref{4.6}), (\ref{4.7}), (\ref{8.3}), (\ref{8.4}) and Lemma \ref{lemma} below, we have
\begin{equation}\label{concave4}
s''(\xi)=-\big(\cot(\xi-\tau_4)+\cot(2\pi-\xi)\big)-\sum_{i=1}^3\big(\cot(\xi-\tau_i)+\cot(\eta_i-\xi)\big)<0.
\end{equation}
\end{enumerate}

Then by (\ref{concave1}),  (\ref{concave2}), (\ref{concave3}) and  (\ref{concave4}), $s(\xi)$ is strictly concave on $I,$ and by (\ref{lim}), $s(\xi)$ achieves a unique maximum point $\xi_0$ in the interior of $I.$

Now for each sequence $z^{(r)}$ with $\lim_{r\to\infty}\frac{2\pi z^{(r)}}{r}=\xi,$ by Lemma \ref{stima} one has 
$$|S_{z^{(r)}}|=\exp\Big(\frac{r}{2\pi}s(\xi)+O(\log r)\Big)\leqslant \exp\Big(\frac{r}{2\pi}s(\xi_0)+C\log r\Big).$$
Since all the $S_z$'s have the same sign, we have
$$\bigg|\sum_{z=\max\{T_i\}}^{\min\{Q_j\}}S_z\bigg|\leqslant \big(\min\{Q_j,r-2\}-\max\{T_i\}\big)\exp\Big(\frac{r}{2\pi}s(\xi_0)+C\log r\Big),$$
and hence
\begin{equation*}
\begin{split}
\limsup_{r\to\infty}&\frac{1}{r}\log\bigg|\sum_{z=\max\{T_i\}}^{\min\{Q_j\}}S_z\bigg|\\
 \leqslant &\lim_{r\to\infty}\frac{1}{r}\log\bigg(\big(\min\{Q_j,r-2\}-\max\{T_i\}\big)\exp\Big(\frac{r}{2\pi}s(\xi_0)+C\log r\Big)\bigg)=\frac{s(\xi_0)}{2\pi}.
\end{split}
\end{equation*}
On the other hand, let $z^{(r)}$ be a sequence such that 
$$\lim_{r\to\infty}\frac{2\pi z^{(r)}}{r}=\xi_0.$$ Then by Lemma \ref{stima}
$$\lim_{r\to\infty}\frac{2\pi}{r}\log S_{z^{(r)}}=s(\xi_0).$$
Again since all the $S_z$'s have the same sign, we have
$$\bigg|\sum_{z=\max\{T_i\}}^{\min\{Q_j\}}S_z\bigg|>S_{z^{(r)}},$$ and hence
$$\liminf_{r\to\infty}\frac{2\pi}{r}\log\bigg|\sum_{z=\max\{T_i\}}^{\min\{Q_j\}}S_z\bigg|\geqslant \lim_{r\to\infty}\frac{2\pi}{r}\log|S_{z^{(r)}}|=s(\xi_0).$$
Therefore, we have 
$$\lim_{r\to\infty}\frac{2\pi}{r}\log\bigg(\sum_{z=\max\{T_i\}}^{\min\{Q_j\}}S_z\bigg)=\lim_{r\to\infty}\frac{2\pi}{r}\log\bigg|\sum_{z=\max\{T_i\}}^{\min\{Q_j\}}S_z\bigg|=s(\xi_0),$$
and together with (\ref{del}), 
\begin{equation*}
\begin{split}
\lim_{r\to\infty}\frac{2\pi}{r}&\log \bigg|\begin{array}{ccc}a_1^{(r)} & a_2^{(r)} & a_3^{(r)} \\a_4^{(r)} & a_5^{(r)} & a_6^{(r)} \\\end{array} \bigg|_{q=e^{\frac{2\pi \sqrt{-1}}{r}}}\\
=&\delta(\alpha_1,\alpha_2,\alpha_3)+\delta(\alpha_1,\alpha_5,\alpha_6)+\delta(\alpha_2,\alpha_4,\alpha_6)+\delta(\alpha_3,\alpha_4,\alpha_5)+s(\xi_0)=V(\xi_0).
\end{split}
\end{equation*}
This completes Step 2.
\\

We accomplish Step 3 as follows. Recall that  $U$ is the function defined in (\ref{term}). Then by (\ref{dilogLob}), together with (\ref{2.4}), (\ref{2.6}), (\ref{2.7}) in Case (1), (\ref{3.2}), (\ref{3.3}), (\ref{3.4}) in Case (2),  (\ref{4.3}), (\ref{4.6}), (\ref{4.7}) in Case (3) and a direct computation, we have that the real part of $U(\boldsymbol\alpha,\xi)$ is independent of $\xi.$ As a consequence, 
$$\frac{\partial \mathrm{Re}U(\boldsymbol \alpha,\xi)}{\partial \xi}=0$$
for every $\xi$ in $I.$ Since $V(\xi)=\frac{1}{2}\mathrm{Im}U(\boldsymbol\alpha,\xi)$ and $\xi_0$ is the maximum of $V,$ 
$$\frac{\partial \mathrm{Im}U(\boldsymbol \alpha,\xi)}{\partial \xi}\bigg|_{\xi=\xi_0}=0.$$ As a consequence, 
$$\frac{\partial U(\boldsymbol \alpha,\xi)}{\partial \xi}\bigg|_{\xi=\xi_0}=0,$$
and by (\ref{4k}), either $\xi_0=\xi(\boldsymbol\alpha)$ or  $\xi_0=\xi^*(\boldsymbol\alpha).$ A direct computation at $\boldsymbol\alpha=\boldsymbol \pi=(\pi,\pi,\pi,\pi,\pi,\pi)$  shows that 
$$\xi_0=\xi(\boldsymbol\pi)=\frac{7\pi}{4},$$
and $\xi^*(\boldsymbol\pi)=\frac{5\pi}{4}$ which dose not lie in $I=[\frac{3\pi}{2},2\pi].$
Now by Proposition \ref{connected}, in each of the Cases (1), (2), (3) of Lemma \ref{3poss}, the space of $\boldsymbol \alpha$ is connected, and 
$\boldsymbol \pi=(\pi,\pi,\pi,\pi,\pi,\pi)$ belongs to all of these three Cases. As a consequence, we have 
$$\xi_0=\xi(\boldsymbol\alpha)$$
for $\boldsymbol \alpha$  in each of the Cases (1), (2), (3), and by Theorem \ref{volume},
$$V(\xi_0)=V(\xi(\boldsymbol\alpha))=\mathrm{Vol}(\Delta).$$
This completes Step 3.

Putting Steps 1, 2, and 3 together, we complete the proof. 
\end{proof}

\begin{lemma}\label{lemma} For $\alpha, \beta\in[0,\pi],$ if $0<\alpha+\beta<\pi,$ then 
$\cot\alpha+\cot\beta>0.$
\end{lemma}

\begin{proof} Under the given conditions, there is an Euclidean triangle $ABC$ with $\angle A= \alpha$ and $\angle B=\beta.$ Let $|AB|$ be the lengths of the edge $AB$ and let $h_C$ be the hight at $AB,$ then 
$$\cot\alpha+\cot\beta=\frac{|AB|}{h_C}>0.$$
\end{proof}

\begin{theorem}\label{volume2} Suppose  $\Delta$ is a generalized hyperbolic tetrahedron $\Delta$ with dihedral angles $(\theta_1,\dots,\theta_6)$ and with  $G_{ii}<0$ for at least one  $i\in\{1,2,3,4\},$  and $\boldsymbol\alpha=(\pi\pm\theta_1,\pi\pm\theta_2,\pi\pm\theta_3,\pi\pm\theta_4,\pi\pm\theta_5,\pi\pm\theta_6).$  Then
$$\mathrm{Vol}(\Delta)=V(\xi_0)$$
where $\xi_0$ is the unique maximum point of $V$ on the interval  $[\max\{\tau_1,\tau_2,\tau_3,\tau_4\}, \min\{\eta_1,\eta_2,\eta_3,2\pi\}].$
\end{theorem}

\begin{proof} From the proof of Theorem \ref{asymp1} (2) above, we see that the result holds for the three cases in Lemma \ref{3poss}. For the general case, we observe that the change of colors operation around a face changes the three limiting $\alpha$'s around the face to $2\pi-\alpha.$ Let us by abuse of terminology still call this operation a change of colors operation. Notice that a change of colors operation does not change the dihedral angle $(\theta_1,\dots,\theta_6).$ Then by Lemma \ref{3poss}, it suffices to show that $V(\xi_0)$ is unchanged under the change of colors operation. 

Now suppose that for the $6$-tuple $\boldsymbol\alpha=(\alpha_i)_{i\in\{1,\dots,6\}}$ the function $V$ has a unique maximum point $\xi_0$ on $[\max\{\tau'_i\},\min\{\eta'_j,2\pi\}]$ and $V(\xi_0)=\mathrm{Vol}(\Delta).$ Without loss of generality, let $\boldsymbol\alpha'=(\alpha'_i)_{i\in\{1,\dots,6\}}=(\alpha_1,\alpha_2,\alpha_3,2\pi-\alpha_4,2\pi-\alpha_5,2\pi-\alpha_6)$ be the $6$-tuple obtained from $\boldsymbol \alpha$ by doing a change of colors operation around the face opposite to the vertex $\mathbf v_1.$ Then a direct computation show that for any triple $(i,j,k)$ around a vertex, 
\begin{equation}\label{d=}
\delta(\alpha_i,\alpha_j,\alpha_k)=\delta(\alpha'_i,\alpha'_j,\alpha'_k).
\end{equation}
Let 
$$\tau'_1=\frac{\alpha'_1+\alpha'_2+\alpha'_3}{2}, \quad \tau'_2=\frac{\alpha'_1+\alpha'_5+\alpha'_6}{2},\quad \tau'_3=\frac{\alpha'_2+\alpha'_4+\alpha'_6}{2},\quad \tau'_4=\frac{\alpha'_3+\alpha'_4+\alpha'_5}{2},$$
$$\eta'_1=\frac{\alpha'_1+\alpha'_2+\alpha'_4+\alpha'_5}{2},\quad\eta'_2=\frac{\alpha'_1+\alpha'_3+\alpha'_4+\alpha'_6}{2},\quad \eta'_3=\frac{\alpha'_2+\alpha'_3+\alpha'_5+\alpha'_6}{2}.$$
  Then we have
\begin{equation}\label{=}
\tau_1=\tau'_1,
\end{equation}
and a direct computation shows that 
\begin{equation}\label{===}
\left\{
\begin{array}{rcl}
2\pi-\eta_3&=&\tau'_2-\tau'_1,\\
2\pi-\eta_2&=&\tau'_3-\tau'_1,\\
2\pi-\eta_1&=&\tau'_4-\tau'_1,\\
\eta_3-\eta_2&=&\tau'_3-\tau'_2,\\
\eta_3-\eta_1&=&\tau'_4-\tau'_2,\\
\eta_2-\eta_1&=&\tau'_4-\tau'_3,\\
\end{array}\right.
\quad\text{and}\quad 
\left\{
\begin{array}{rcl}
\tau_4-\tau_3&=&\eta'_2-\eta'_1,\\
\tau_4-\tau_2&=&\eta'_3-\eta'_1,\\ 
\tau_4-\tau_1&=& 2\pi-\eta'_1.\\
\tau_3-\tau_2&=&\eta'_3-\eta'_2,\\
\tau_3-\tau_1&=&2\pi-\eta'_2,\\ 
\tau_2-\tau_1&=& 2\pi-\eta'_3.\\
\end{array}\right.
\end{equation}
As a consequence,  if 
$$\tau_i=\max\{\tau_2,\tau_3,\tau_4\}\quad\text{and}\quad\eta_j=\min\{\eta_1,\eta_2,\eta_3\}$$
for some $i,j\in\{1,2,3,4\},$ 
then $$(i,j)\in\{(4,1), (3,2), (2,3)\},$$
and
$$\tau'_i=\max\{\tau'_2,\tau'_3,\tau'_4\}\quad\text{and}\quad\eta'_j=\min\{\eta'_1,\eta'_2,\eta'_3\}$$
for the same pair $(i,j).$ 

For $(i,j)\in\{(4,1), (3,2), (2,3)\},$ a direct computation show that 
\begin{equation}\label{==}
\tau_i-\eta_j=\tau'_i-\eta'_j
\end{equation}
and 
\begin{equation}\label{6=}
2\pi-\eta_j=\tau'_i-\tau'_1\quad\text{and}\quad 2\pi-\eta'_j=\tau_i-\tau_1.
\end{equation}

Let 
$$s_{\boldsymbol\alpha}(\xi)=\sum_{i=1}^4\Lambda(\xi-\tau_i)+\sum_{j=1}^3\Lambda(\eta_j-\xi)-\Lambda(\xi),$$
and let 
$$s_{\boldsymbol\alpha'}(\xi)=\sum_{i=1}^4\Lambda(\xi-\tau'_i)+\sum_{j=1}^3\Lambda(\eta'_j-\xi)-\Lambda(\xi).$$
Then by (\ref{=}), (\ref{===}), (\ref{==}) and a direct computation, we have for any $\xi,$
\begin{equation}\label{4=}
s_{\boldsymbol\alpha'}(\eta'_j-\xi)=s_{\boldsymbol\alpha}(\tau_i+\xi).
\end{equation}

We consider the following three cases.
\begin{enumerate}[(1)]

\item $\tau_1\leqslant \tau_i$ and $\tau_1'\leqslant \tau'_i.$ In this case, we have
$$[\max\{\tau_1,\tau_2,\tau_3,\tau_4\},\min\{\eta_1,\eta_2,\eta_3,2\pi\}]=[\tau_i,\eta_j],$$
and 
$$ [\max\{\tau'_1,\tau'_2,\tau'_3,\tau'_4\}, \min\{\eta'_1,\eta'_2,\eta'_3,2\pi\}]=[\tau'_i,\eta'_j].$$
By \ref{==}, the two intervals $[\tau_i,\eta_j]$ and $[\tau'_i,\eta'_j]$ have the same length. If $\xi_0$ is the unique maximum of $s_{\boldsymbol\alpha}(\xi)$ on $[\tau_i,\eta_j],$ then by (\ref{==}) again, 
$$\xi'_0\doteq\eta'_j+\tau_i-\xi_0$$ 
lies in $[\tau'_i,\eta'_j],$ and by (\ref{4=}), $\xi'_0$ is the unique maximum of $s_{\boldsymbol\alpha'}(\xi)$ on $[\tau'_i,\eta'_j]$ with 
\begin{equation}\label{5=}
s_{\boldsymbol\alpha'}(\xi'_0)=s_{\boldsymbol\alpha}(\xi_0).
\end{equation}

\item  $\tau_1> \tau_i$ and $\tau_1'\leqslant \tau'_i,$ or $\tau_1\leqslant \tau_i$ and $\tau_1'>\tau'_i.$ By symmetry, we only need to consider the former case. In this case, by (\ref{6=}),  we have
$$[\max\{\tau_1,\tau_2,\tau_3,\tau_4\}, \min\{\eta_1,\eta_2,\eta_3,2\pi\}]=[\tau_1,\eta_j],$$
and 
$$ [\max\{\tau'_1,\tau'_2,\tau'_3,\tau'_4\},\min\{\eta'_1,\eta'_2,\eta'_3,2\pi\}]=[\tau'_i,2\pi],$$
and the two intervals $[\tau_1,\eta_j]$ and $[\tau'_i,2\pi]$ have the same length.  If $\xi_0$ is the unique maximum of $s_{\boldsymbol\alpha}(\xi)$ on $[\tau_i,\eta_j],$ then by (\ref{6=}) again, 
$$\xi'_0\doteq\eta'_j+\tau_i-\xi_0$$ 
lies in $[\tau'_i,2\pi],$  and by (\ref{4=}), $\xi'_0$ is the unique maximum of $s_{\boldsymbol\alpha'}(\xi)$ on $[\tau'_i,2\pi]$ with 
\begin{equation}\label{7=}
s_{\boldsymbol\alpha'}(\xi'_0)=s_{\boldsymbol\alpha}(\xi_0).
\end{equation}

\item $\tau_1> \tau_i$ and $\tau_1'>\tau'_i.$  In this case, by (\ref{6=}), we have
$$[\max\{\tau_1,\tau_2,\tau_3,\tau_4\},\min\{\eta_1,\eta_2,\eta_3,2\pi\}]=[\tau_1,2\pi],$$
and 
$$ [\max\{\tau'_1,\tau'_2,\tau'_3,\tau'_4\},\min\{\eta'_1,\eta'_2,\eta'_3,2\pi\}]=[\tau'_1,2\pi].$$
By \ref{=}, the two intervals $[\tau_1,2\pi]$ and $[\tau'_1,2\pi]$ have the same length.  If $\xi_0$ is the unique maximum of $s_{\boldsymbol\alpha}(\xi)$ on $[\tau_1,2\pi],$ then by (\ref{6=}), 
$$\xi'_0\doteq\eta'_j+\tau_i-\xi_0$$ 
lies in $[\tau'_1,2\pi],$  and by (\ref{4=}), $\xi'_0$ is the unique maximum of $s_{\boldsymbol\alpha'}(\xi)$ on $[\tau'_1,2\pi]$ with 
\begin{equation}\label{8=}
s_{\boldsymbol\alpha'}(\xi'_0)=s_{\boldsymbol\alpha}(\xi_0).
\end{equation}

\end{enumerate}

Putting (\ref{d=}) and (\ref{5=}) in Case (1), (\ref{7=}) in Case (2) and (\ref{8=}) in Case (3) together,  we have that $\xi'_0$ is the unique maximum of $V(\xi)$ on  $[\max\{\tau'_i\},\min\{\eta'_j,2\pi\}],$ and 
$$V(\xi'_0)=V(\xi_0)=\mathrm{Vol}(\Delta).$$

\end{proof}


\section{Application to the volume conjecture for trivalent graphs}

This section is devoted to an application of the previous results to the volume conjecture for polyhedra. The first subsection recalls the \emph{Kauffman bracket}, an invariant of graphs of which the quantum $6j$-symbol can be seen as a special case of; the second section states the volume conjecture for polyhedra and gives a proof in the case of prisms and of a more general family of polyhedra.

\subsection{The Kauffman bracket for graphs}

We start with some notations. For $i\in \mathbb{N}$ define
\begin{equation}
 \Delta_i=(-1)^{i+1}[i+1].
\end{equation}

Recall the function $\Delta(a,b,c)$ defined in Section \ref{6j}, taking as input an $r$-admissible triple. 
If $v$ is a trivalent vertex of a graph whose incident edges are colored by an admissible triple $a,b,c$ we write for short $\Delta(v)$ instead of $\Delta(a,b,c)$.

 The \emph{Kauffman bracket} is an invariant of \emph{trivalent framed graphs}; because we will only be interested in the case of planar trivalent graphs, we will only list the properties needed to calculate the Kauffman bracket in this case.

\begin{definition}\label{def:kauf}
 The \emph{Kauffman bracket} at the root of unity $q=e^{\frac{2\pi \sqrt{-1}}{r}}$ is the unique map $$\langle\ \rangle_q:\{\textrm{colored trivalent planar graphs in } S^3\}\rightarrow \mathbb{C}$$ with the following properties:
 \begin{enumerate}[(i)]
 \item\label{prop:prima} If $\Gamma$ is the circle colored with $i\in \mathbb{N}$ then $\langle\Gamma\rangle_q=\Delta_i$;
  \item If $\Theta$ is a theta graph $\cp{\includegraphics[width=0.5cm]{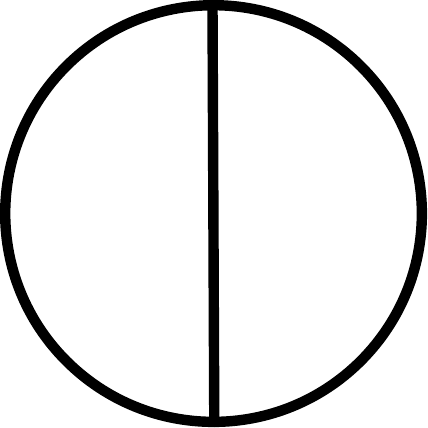}}$ colored with the admissible triple $\left(a,b,c\right)\in \mathbb{N}^3$ then $\langle\Gamma\rangle_q=1$;
  \item If $\Gamma$ is a tetrahedron graph $\cp{\includegraphics[width=0.5cm]{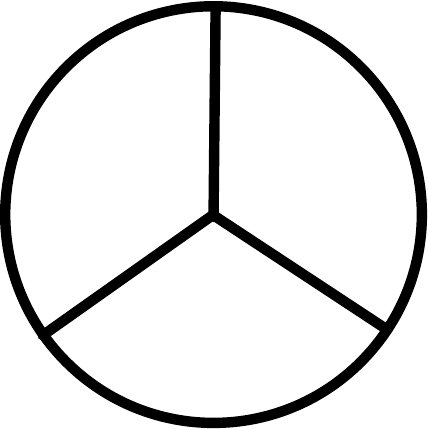}}$ colored with the $r$-admissible $6$-tuple $(a_1,\dots,a_6)\in I_r^6$ then $$\langle\Gamma\rangle_q=\begin{vmatrix}
   a_1 &a_2&a_3\\
   a_4&a_5&a_6
  \end{vmatrix}_q;$$
  \item\label{prop:fusion} The \emph{fusion rule}:
\begin{equation}\label{eq:fusion} \left\langle\vcenter{\hbox{\includegraphics[width=1.5cm]{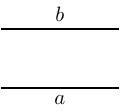}}}\right\rangle_q=\sum_{i\in I_r}\Delta_i\left\langle\vcenter{\hbox{\includegraphics[width=2cm]{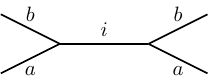}}}\right\rangle_q\end{equation}
  \item \label{prop:bridge} If $\Gamma$ has a bridge (that is to say, an edge that disconnects the graph if removed) colored with $i\neq 0$, then $\langle\Gamma\rangle_q=0$;
  \item If at some vertex of $\Gamma$ the colors do not form an $r$-admissible triple, then $\langle\Gamma\rangle_q=0$;
  \item \label{prop:zero}If $\Gamma$ is colored with an $r$-admissible coloring such that the color of an edge $e$ is equal to $0$, then \begin{equation} \left\langle\vcenter{\hbox{\includegraphics[width=2cm]{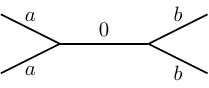}}}\right\rangle_q=\frac{1}{\sqrt{\Delta_a\Delta_b}}\left\langle\vcenter{\hbox{\includegraphics[height=0.8cm]{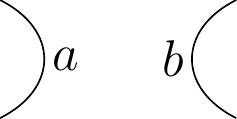}}}\right\rangle_q\end{equation}
  
  $\langle\Gamma\rangle_q=\frac{1}{\sqrt{\Delta_a\Delta_b}}\langle\Gamma'\rangle_q$ where $\Gamma'$ is $\Gamma$ with $e$ removed, and $a,b$ are the colors of the edges that share a vertex with $e$ (notice that since the coloring is $r$-admissible, two edges sharing the same vertex with $e$ will have the same color);
  \item \label{prop:ultima} If $\Gamma$ is the disjoint union of $\Gamma_1$ and $\Gamma_2$, then $\langle\Gamma\rangle_q=\langle\Gamma_1\rangle_q\langle\Gamma_2\rangle_q$.
 \end{enumerate}

\end{definition}

It is immediately clear that from the properties listed it is possible to calculate the Kauffman bracket of any trivalent planar graph:  repeated applications of the Fusion rule create a bridge, and Properties (\ref{prop:bridge}), (\ref{prop:zero}) and (\ref{prop:ultima}) allow to compute $\langle\Gamma\rangle$ by induction. This immediately implies uniqueness of the Kauffman bracket map. The existence is not at all obvious; its proof can be found  in \cite[Chapter 9]{KL} (in the general case of framed trivalent graphs in $S^3$). Notice that \cite{KL} uses a different normalization of the invariants (here we use the \emph{unitary normalization}); nevertheless the proof of the existance is unaffected by the different normalization.

\subsection{The volume conjecture for polyhedra}
The volume conjecture for polyhedra has been proposed in various forms in several papers. The first such conjecture, for $q=e^{\frac{\pi \sqrt{-1}}{r}}$ and simple hyperideal polyhedra (i.e. polyhedra with only trivalent hyperideal vertices), appeared in \cite{CGvdV}; a volume conjecture for the root of unity $q=e^{\frac{2\pi \sqrt{-1}}{r}}$ appeared in \cite{KM} for simple compact polyhedra. An all-encompassing version of the conjecture, stated for any generalized hyperbolic polyhedron (i.e. polyhedra with any combination of regular, ideal or hyperideal vertices), has been proposed in \cite{BEL}. The statement in general involves graphs with vertices of any valence and their Yokota invariant; in the case of simple polyhedra it involves the Kauffman bracket and the statement is the following.

 \begin{conjecture}[The Volume Conjecture for polyhedra]\label{conjmk}
Let $P$ be a simple generalized hyperbolic polyhedron with dihedral angles
  $\theta_1,\dots,\theta_m$ at the edges $e_1,\dots,e_m$, and $1$-skeleton $\Gamma$. Let $col_r$ be a sequence of $r$-admissible colorings of the edges $e_1,\dots,e_m$ of $\Gamma$ such that 
  \begin{displaymath}
   2\pi\lim_{r\rightarrow+\infty}\frac{col_r(e_k)}{r}=\pi\pm\theta_k.
  \end{displaymath}
Then at $q=e^{\frac{2\pi\sqrt{-1}}{r}}$ and as $r$ runs over all the odd integers,
\begin{displaymath}
 \lim_{r\rightarrow+\infty}\frac{2\pi}{r}\ln\big| \langle\Gamma,col_r\rangle_q\big|=\mathrm{Vol}(P).
\end{displaymath}

 \end{conjecture}

This conjecture has been proven for some subset of angles of some families of polyhedra; so far every known example had at least one hyperideal vertex.
Using the machinery developed in this paper we are able to prove the volume conjecture for triangular prisms with one condition on the angles of the vertical faces; part of these have no ideal or hyperideal vertices.
 
\begin{theorem}\label{thm:volconjprism} Suppose that $P$ is a triangular prism such that the sum of the dihedral angles of the vertical faces is less than $\pi$. Then Conjecture \ref{conjmk} is true for $P$.
\end{theorem}
\begin{proof}

Call $\Gamma$ the $1$-skeleton of $P$; denote with $\beta_1,\beta_2,\beta_3$ the dihedral angles of one of the bases, with $\gamma_1,\gamma_2,\gamma_3$ the angles of the other base and with $\alpha_1,\alpha_2,\alpha_3$ the vertical angles.
 Consider the planes $\Pi_1,\Pi_2,\Pi_3$ supporting the vertical faces of $P$. Because of the condition of the angles, $\Pi_1\cap \Pi_2\cap \Pi_3$ is a hyperideal point $v$. Consider the plane $\Pi_v$ dual to $v$; we want to ``cut'' $P$ along $\Pi_v$. If $\Pi_v$ intersects the interior of all vertical edges of $P$ (see Figure \ref{fig:cut1}), then cutting $P$ along $\Pi_v$ simply results in two truncated hyperbolic tetrahedra; each of them has, as vertices, three vertices that make up one of the bases and $v$. In other words, $P$ is obtained by taking two generalized hyperbolic polyhedra with a hyperideal vertex, truncating the hyperideal vertex and gluing the truncation faces with an isometry. The angles of one of the tetrahedra is given by the vertical angles of $P$ (for the edges around $v$) and by the angles of the corresponding base of $P$.
 
 \begin{figure}
 \begin{center}
 
 \makebox[\textwidth]{\includegraphics[width=0.7\textwidth]{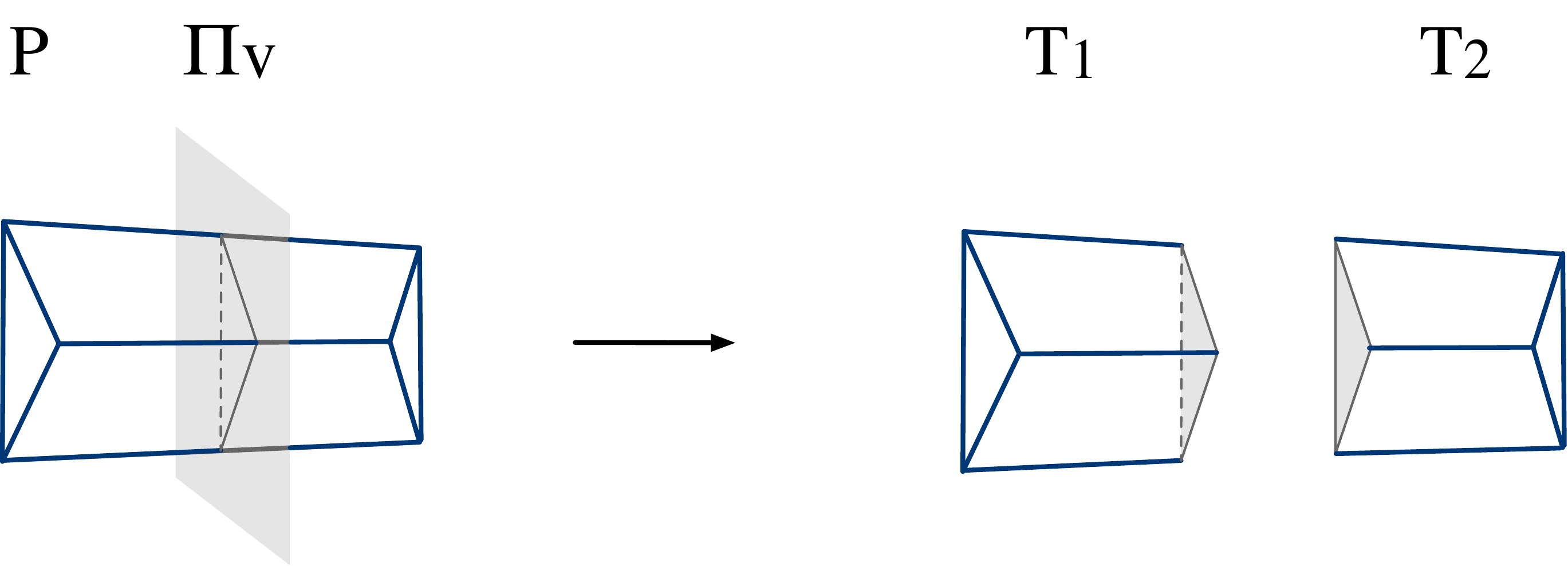}}
 \end{center}
  \caption{A splitting of a prism, case 1: the splitting plane intersects all vertical edges. The result are two truncated hyperideal tetrahedra.}
  \label{fig:cut1}
 \end{figure}
 
 In general, though, $\Pi_v$ might not intersect some (or any) of the interiors of the vertical edges of $P$ (see Figure \ref{fig:cut}). However, with the machinery developed in this paper, the procedure carries out in the exact same way. If we denote the vertices of $P$ with $u_1,u_2,u_3,w_1,w_2,w_3$ (with $u_1,u_2,u_3$ the vertices of one of the bases) and we consider the two generalized hyperbolic tetrahedra $T_1$ and $T_2$ with vertices $u_1,u_2,u_3,v$ and $w_1,w_2,w_3,v$ respectively, we can see that $P$ is obtained by gluing together the truncations of $T_1$ and $T_2$ along $v$. Therefore, 
 \begin{equation}\label{additivity}
 \textrm{Vol}(P)=\textrm{Vol}(T_1)+\textrm{Vol}(T_2),
 \end{equation}
where  $\textrm{Vol}(T_i)$ is the volume of the generalized hyperbolic tetrahedron $T_i$ defined in Definition \ref{v}. Furthermore, notice that as before the angles of $T_1$ are given by $\alpha_1,\alpha_2,\alpha_3$ at the edges around $v$ and by $\beta_1,\beta_2,\beta_3$ at the base; the similarly the angles of $T_2$ are given by $\alpha_1,\alpha_2,\alpha_3$ at the edges around $v$ and by $\gamma_1,\gamma_2,\gamma_3$ at the base;
 
 \begin{figure}
 \begin{center}
 
 \makebox[\textwidth]{\includegraphics[width=0.8\textwidth]{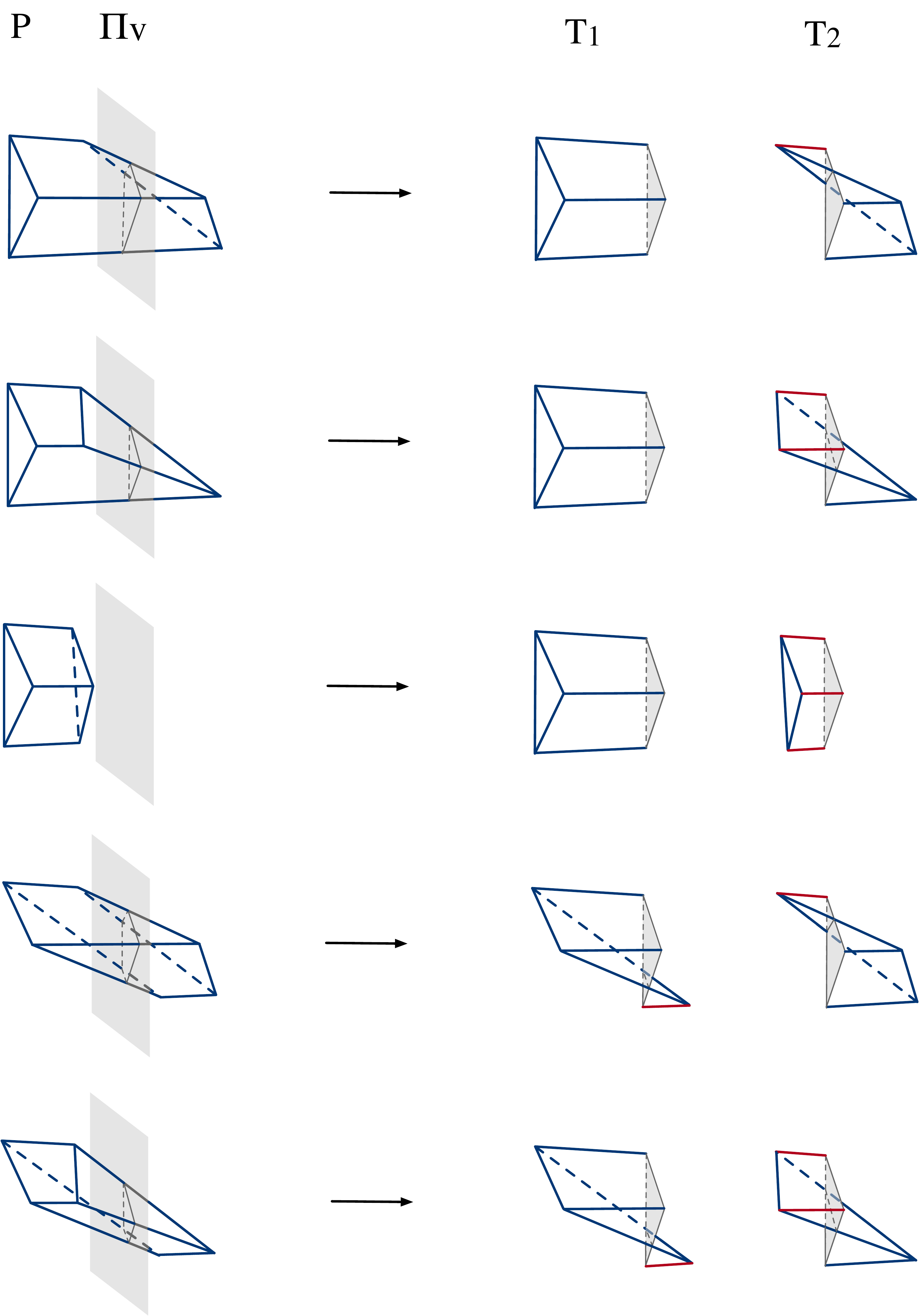}}
 \end{center}
  \caption{A splitting of a prism, case 2: the splitting plane does not intersect all vertical edges. The resulting splitting will be comprised of at least one "non-standard" generalized hyperbolic tetrahedron.}\label{fig:cut}
 \end{figure}
 On the other hand, the Kauffman bracket of $\Gamma$ and coloring $col_r$ assigning $b_1^{(r)},b_2^{(r)},b_3^{(r)}$ to one of the bases, $c_1^{(r)},c_2^{(r)},c_3^{(r)}$ to the other and $a^{(r)}_1,a_2^{(r)},a_3^{(r)}$ to the vertical faces is given by the product of two $6j$-symbols:
 $$\langle\Gamma,col\rangle_q=\Bigg|\begin{matrix}
a_1^{(r)}  & a_2^{(r)}  & a_3^{(r)} \\
   b_1^{(r)}  & b_2^{(r)} & b_3^{(r)} 
  \end{matrix}\Bigg|_q\Bigg|\begin{matrix}
a_1^{(r)}  & a_2^{(r)}  & a_3^{(r)} \\
   c_1^{(r)}  & c_2^{(r)} & c_3^{(r)} 
  \end{matrix}\Bigg|_{q}.$$
  
  Suppose that, as in the hypotheses of the volume conjecture, $\lim_{r\rightarrow+\infty} \frac{2\pi a_k}{r}=\pi\pm\alpha_k$, $\lim_{r\rightarrow+\infty} \frac{2\pi b_k}{r}=\pi\pm\beta_k$ and $\lim_{r\rightarrow+\infty} \frac{2\pi c_k}{r}=\pi\pm\gamma_k.$ Then, because $\alpha_1+\alpha_2+\alpha_3<\pi$, we can apply Theorem \ref{asymp1} to both the $6j$-symbols and we get 
  $$\lim_{r\rightarrow+\infty} \frac{2\pi}{r}\ln \big|\langle \Gamma,col_r\rangle_q\big|=\textrm{Vol}(T_1)+\textrm{Vol}(T_2)=\textrm{Vol}(P),$$ and the proof is complete.
\end{proof}

\begin{figure}
 \centering\includegraphics[scale=0.25]{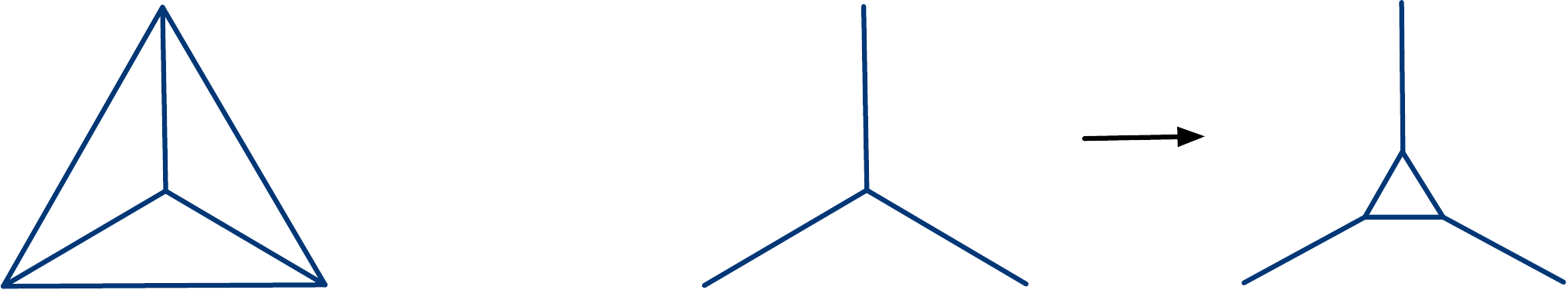}
 \caption{A tetrahedral graph (left) and a blow-up move (right).}\label{fig:blowup}
\end{figure}

\begin{theorem}
 Suppose that $P$ is a polyhedron with the following properties:
 \begin{enumerate}[(1)]
  \item Its $1$-skeleton is obtained from the tetrahedral graph through a sequence of blow-up moves (see Figure \ref{fig:blowup});
  \item For each triple of edges of $P$ that are involved in a blow-up move, the sum of their dihedral angles is less than $\pi$.\label{condpi}
 \end{enumerate}
 Then Conjecture \ref{conjmk} is true for $P$.
\end{theorem}
\begin{proof}
 The idea of the proof is the same as in Theorem \ref{thm:volconjprism}. On the one hand we can decompose $P$ into generalized hyperbolic tetrahedra along each triple of edges involved in a blow-up; Condition \ref{condpi} implies that each resulting tetrahedron has a hyperideal vertex (or equivalently $G_{ii}<0$ for some cofactor of the Gram matrix). On the other hand, the Kauffman bracket of the $1$-skeleton of $\Gamma$ is a product of $6j$-symbols each corresponding to one of the tetrahedra in the decomposition; Theorem \ref{asymp1} can be applied in the same way as in the end of the proof of Theorem \ref{thm:volconjprism}.
\end{proof}

\section{Further problems}

\begin{problem} It would be great if  Condition (2) of Theorem \ref{asymp1}  could be relaxed to only require that the signature of $G$ is $(3,1)$ (so that $(\theta_1,\dots,\theta_6)$ is the set of dihedral angles of a generalized hyperbolic tetrahedra without any further restriction).  It would be natural to conjecture that in this case the exponential growth rate of the corresponding sequences of $6j$-symbols would equal the volume of the generalized hyperbolic tetrahedron. However, it seems not to be the case. See the following example. 
\end{problem}

\begin{example}\label{example}  $(\theta_1,\dots,\theta_6)=(1.2, \pi-1.2, \pi-1.2, 1.2, \pi-1.2, \pi-1.2).$
The Gram matrix is
$$\left[\begin{matrix}
1& -0.36 & 0.36&0.36\\
-0.36& 1& 0.36 & 0.36 \\
0.36& 0.36 & 1&  -0.36 \\
0.36 & 0.36 & -0.36 &  1\\
 \end{matrix}\right].$$
 The cofactors $G_{14},$ $G_{23}$ are positive, and the cofactors $G_{12},$ $G_{13},$ $G_{23},$ $G_{34}$ are negative, hence is in the case of Proposition \ref{class} (1c). Then the volume is the hyperbolic of the geometric piece, which is positive.  One can, however, compute numerically the corresponding sequence of $6j$-symbols and see that the growth rate is negative.

\end{example}



 \noindent 
Giulio Belletti\\
Universit\'e Paris Saclay\\
Orsay, France\\ 
(gbelletti451@gmail.com)
\\

\noindent
Tian Yang\\
Department of Mathematics\\  Texas A\&M University\\
College Station, TX 77843, USA\\
(tianyang@math.tamu.edu)

\end{document}